\documentclass[11pt,a4paper]{amsart}
\usepackage[a4paper,margin=1in]{geometry}
\usepackage{amsmath,amssymb,amsthm,amsfonts,mathtools}
\usepackage{mathrsfs}
\usepackage{bm}
\usepackage{graphicx}
\usepackage{hyperref}
\usepackage{enumitem}
\usepackage{cite}
\usepackage{color}
\usepackage{caption}

\hypersetup{
	colorlinks=true,
	linkcolor=blue,
	citecolor=blue,
	urlcolor=blue
}

\newtheorem{theorem}{Theorem}[section]
\newtheorem{lemma}[theorem]{Lemma}
\newtheorem{proposition}[theorem]{Proposition}
\newtheorem{corollary}[theorem]{Corollary}

\newtheorem{remark}[theorem]{Remark}

\title[Theodorsen equation and Elliptic PDEs]{Conformal Parametrisation of Star-Shaped Domains and Persistence of Semilinear Elliptic Solutions}

\author{Gianni Arioli}
\address{Department of Mathematics\\
Politecnico di Milano\\
piazza Leonardo da Vinci 32, Milano, Italy}
\email{gianni.arioli@polimi.it}
\urladdr{https://arioli.faculty.polimi.it}
\date{}

\begin{document}
	
\begin{abstract}
We study the conformal parametrisation of smooth star-shaped planar
domains and its dependence on perturbations of the boundary. The
boundary correspondence of the normalized Riemann map is described by
the Theodorsen equation, which we formulate in weighted Wiener
algebras. We derive an explicit representation of the inverse of the
linearized Theodorsen operator through a canonical Wiener--Hopf
factorisation and obtain quantitative a posteriori estimates that are
stable under perturbations of the boundary function.

For a collection of nonperturbative reference domains, rigorous
computer-assisted estimates yield explicit neighborhoods in a
weighted Wiener algebra such that every boundary function in these
neighborhoods admits a uniquely determined conformal parametrisation in the validated ball,
with quantitative control of the corresponding Riemann map. Thus the
validation applies to open, infinite-dimensional families of
star-shaped domains.

As an application, we consider the semilinear Dirichlet problem
\[
-\Delta v=v^3\quad\text{in }\Omega,
\qquad
v=0\quad\text{on }\partial\Omega.
\]
Pulling the equation back by the validated conformal maps yields a
problem on the unit disk. Uniform
a posteriori estimates then imply persistence of the elliptic
solutions throughout explicit neighborhoods of the reference
domains.
\end{abstract}

	\maketitle

\begin{center}
	\emph{In memory of Hans Koch,\\
		a valued friend and collaborator for over twenty years,\\
		from whose insight and experience I learned much.}
\end{center}
\bigskip

	\tableofcontents

\section{Introduction}

We study the conformal parametrisation of smooth star-shaped planar
domains and its application to the semilinear Dirichlet problem
\begin{equation}\label{eq:elliptic-intro}
	\begin{cases}
		-\Delta v=v^3,&\text{in }\Omega,\\
		v=0,&\text{on }\partial\Omega.
	\end{cases}
\end{equation}
The domains under consideration are bounded by Jordan curves admitting
a polar parametrisation
\begin{equation}\label{eq:boundary-intro}
	\gamma(t)=r(t)(\cos t,\sin t),
	\qquad t\in[0,2\pi],
\end{equation}
where $r$ is positive and belongs to a weighted Wiener algebra.

Our starting point is the Theodorsen equation \cite{Theodorsen}, which characterizes the
boundary correspondence of the normalized Riemann map from the unit
disk onto $\Omega$. More precisely, if
\[
\theta(t)=t+u(t)
\]
denotes the change of angular variable induced by the conformal map,
then $u$ satisfies
\begin{equation}\label{eq:theodorsen-intro}
	u(t)
	=
	\mathcal H\bigl(\log r(t+u(t))\bigr),
\end{equation}
where $\mathcal H$ denotes the periodic Hilbert transform. Once a
solution of \eqref{eq:theodorsen-intro} has been obtained, the boundary
values of the Riemann map are given by
\begin{equation}\label{eq:boundary-riemann-intro}
	f(e^{it})
	=
	r(t+u(t))e^{i(t+u(t))}.
\end{equation}

The principal purpose of this paper is to develop a quantitative theory
for \eqref{eq:theodorsen-intro} which is stable under perturbations of
the boundary function $r$. We work in weighted Wiener algebras and
derive an explicit inverse for the linearized Theodorsen operator by
means of a Wiener--Hopf factorisation. This yields quantitative
a posteriori conditions which imply not only existence and uniqueness
of the boundary correspondence for a given domain, but also its
persistence for all boundary functions in an explicitly determined
neighborhood of that domain.

More precisely, let $1<\rho<\sigma$ and consider
\[
F(r,u)
=
u-\mathcal H\bigl(\log r(\,\cdot+u(\cdot)\,)\bigr).
\]
The derivative with respect to $u$ is
\[
D_uF(r,u)
=
I-\mathcal H M_{\psi_{r,u}},\qquad
\psi_{r,u}=\frac{r'(\,\cdot+u(\cdot)\,)}{r(\,\cdot+u(\cdot)\,)},
\]
where $M_{\psi_{r,u}}$ denotes the multiplication operator by $\psi_{r,u}$.
The loss from the weight $\sigma$ to the smaller weight $\rho$ is
essential in controlling composition and differentiation. In
particular, if $\|u\|_\rho\leq c$ and
\[
\rho e^c<\sigma,
\]
then composition by the boundary correspondence defines a bounded map
from $\mathcal A_\sigma$ to $\mathcal A_\rho$. This permits quantitative
control of both
\[
r(\,\cdot+u(\cdot)\,)
\quad\text{and}\quad
\frac{r'}{r}(\,\cdot+u(\cdot)\,)
\]
under perturbations of $r$ in $\mathcal A_\sigma$.

A central analytical ingredient is the inversion of operators of the
form
\[
\Phi=I+\mathcal H M_a.
\]
Introducing
\[
A_+=1-ia,
\qquad
A_-=1+ia,
\qquad
W=\frac{A_-}{A_+},
\]
we show that, whenever $W$ admits a canonical Wiener--Hopf
factorisation and a scalar nondegeneracy condition is satisfied,
$\Phi$ possesses an explicit inverse. The resulting formula provides
quantitative operator bounds and forms the basis of the stability
analysis for the Theodorsen equation.

The abstract estimates are then applied to several strongly
noncircular star-shaped domains. For each reference boundary function
$r_0$ considered in the paper, all quantities entering the a posteriori
estimates are evaluated with interval arithmetic. As a consequence, we
obtain an explicit number $\delta_r>0$ such that the Theodorsen equation
admits a uniquely determined solution for every boundary function
$r\in\mathcal A_\sigma$ satisfying
\[
\|r-r_0\|_\sigma<\delta_r.
\]
Thus the validated examples are not isolated domains: each of them
generates an open, infinite-dimensional family of star-shaped domains
for which the conformal parametrisation is rigorously controlled.

We next use this conformal description to study
\eqref{eq:elliptic-intro}. If
\[
f:\mathbb D\longrightarrow\Omega
\]
is the normalized Riemann map and
\[
w(z)=v(f(z)),
\]
then conformal invariance of the planar Laplacian transforms
\eqref{eq:elliptic-intro} into
\begin{equation}\label{eq:elliptic-disk-intro}
	\begin{cases}
		-\Delta w=qw^3,&\text{in }\mathbb D,\\
		w=0,&\text{on }\partial\mathbb D,
	\end{cases}
	\qquad
	q=|f'|^2.
\end{equation}
The analytic control of $f$ obtained from the Theodorsen equation
therefore yields quantitative control of the coefficient $q$ in a
weighted Zernike algebra.

Using a posteriori estimates for \eqref{eq:elliptic-disk-intro}, we
prove existence of positive and sign-changing solutions on the
reference domains. More importantly, the quantitative dependence of
the conformal map on $r$ allows these existence results to be continued
to explicit neighborhoods of the reference domains. In particular,
for each of the domains considered below, there exists an open set of
boundary functions in $\mathcal A_\sigma$ for which
\eqref{eq:elliptic-intro} possesses a rigorously controlled solution
with the prescribed qualitative properties.

The results therefore establish a chain of quantitative stability
properties
\[
r
\longmapsto
u_r
\longmapsto
f_r
\longmapsto
|f_r'|^2
\longmapsto
v_r,
\]
linking perturbations of the boundary of the domain to the persistence
of solutions of the nonlinear elliptic equation. Computer assistance
enters only in the rigorous evaluation of the explicit inequalities
which determine the admissible neighborhoods; the analytical
implications of these inequalities are established independently.

The paper is organized as follows. In Section~\ref{sec:main-results} we introduce the
weighted Wiener and Zernike spaces and state the main results.
Section~\ref{sec:theo} recalls the relation between the Theodorsen equation and the
Riemann map. Section~\ref{sec:aposteriori} develops the a posteriori formulation of the
Theodorsen equation. In Section~\ref{sec:WH} we derive the Wiener--Hopf
representation of the inverse linearized operator. Section~\ref{sec:stability} combines
these estimates with quantitative composition bounds to establish
stability of the conformal parametrisation under perturbations of the
boundary. The computer-assisted verification for the reference
domains is then described in Section~\ref{sec:verification}. The application to the
semilinear elliptic equation and its persistence under perturbations
of the domain is developed in Section~\ref{sec:elliptic}. The reference domains, the corresponding elliptic solutions, and the coefficients \(q\) are illustrated in Section~\ref{sec:pics}. Technical results concerning
weighted Wiener algebras, Wiener--Hopf factorisation, Zernike
representations, positivity, and the rigorous implementation are
collected in the appendices.

\section{Functional setting and main results}
\label{sec:main-results}

For $\rho>1$, let $\mathcal A_\rho$ be the weighted Wiener algebra of Fourier series defined by
\[
\mathcal A_\rho
=
\left\{
u(t)
=
a_0+\sum_{k\geq1}
\bigl(a_k\cos(kt)+b_k\sin(kt)\bigr):
\|u\|_\rho<\infty
\right\},
\]
where
\[
\|u\|_\rho
=
|a_0|
+
\sum_{k\geq1}
\bigl(|a_k|+|b_k|\bigr)\rho^k.
\]
We also use the weighted Wiener algebra of Laurent series
\[
\mathcal B_\rho
=
\left\{
f(z)=\sum_{k\in\mathbb Z}f_kz^k:
\|f\|_\rho
=
\sum_{k\in\mathbb Z}|f_k|\rho^{|k|}
<\infty
\right\}.
\]
The precise correspondence between $\mathcal A_\rho$ and the
real-symmetric subspace of $\mathcal B_\rho$ is recalled in
Appendix~\ref{sec:appendix-wiener}.
The distinction between two exponential weights will play an
important role. Throughout the stability analysis we fix
\[
1<\rho<\sigma.
\]
Boundary functions are measured in $\mathcal A_\sigma$, whereas the
boundary correspondence is controlled in $\mathcal A_\rho$. This loss
of analyticity permits uniform estimates for nonlinear composition and
differentiation.

Indeed, if $r_1,r_2\in\mathcal A_\sigma$ and
$u\in\mathcal A_\rho$ satisfies $\|u\|_\rho\leq c$, then
\[
\rho e^c\leq\sigma
\]
implies
\begin{equation}\label{eq:composition-main}
	\bigl\|
	r_1(\,\cdot+u(\cdot)\,)
	-
	r_2(\,\cdot+u(\cdot)\,)
	\bigr\|_\rho
	\leq
	\|r_1-r_2\|_\sigma.
\end{equation}
Analogous estimates for the logarithmic derivatives give quantitative
control of
\[
\frac{r_1'}{r_1}(\,\cdot+u(\cdot)\,)
-
\frac{r_2'}{r_2}(\,\cdot+u(\cdot)\,),
\]
provided the boundary functions remain uniformly bounded away from
zero. These estimates are proved in Section~\ref{sec:stability}.

We can now state the first main result. For a positive
$r\in\mathcal A_\sigma$, define
\[
F_r(u)
=
u-\mathcal H\bigl(\log r(\,\cdot+u(\cdot)\,)\bigr).
\]

\begin{theorem}[Stability of the conformal parametrisation]
	\label{thm:conformal-stability}
	Let $r_0\in\mathcal A_\sigma$ be one of the reference boundary
	functions listed in Table~\ref{tab:domains}. There exist explicitly
	computable constants
	\[
	\delta_r>0,\qquad R_u>0,
	\]
	and a function $\bar u\in\mathcal A_\rho$ such that, for every
	$r\in\mathcal A_\sigma$ satisfying
	\[
	\|r-r_0\|_\sigma<\delta_r,
	\]
	the following assertions hold:
	\begin{enumerate}
		\item $r$ is positive and determines a star-shaped Jordan domain
		$\Omega_r$;
		
		\item the Theodorsen equation
		\[
		F_r(u)=0
		\]
		has a unique solution $u_r$ in
		\[
		B_{R_u}(\bar u)\subset\mathcal A_\rho;
		\]
		
		\item the function
		\[
		f_r(e^{it})
		=
		r(t+u_r(t))e^{i(t+u_r(t))}
		\]
		is the boundary trace of the normalized Riemann map
		\[
		f_r:\mathbb D\longrightarrow\Omega_r;
		\]
		
		\item $f_r$ admits a holomorphic continuation beyond the closed unit
		disk and satisfies explicit quantitative bounds in a weighted
		Wiener algebra.
	\end{enumerate}
	In the case where $\sigma=2$ and $\rho=65/64$, the constants $\delta_r$ and $R_u$ are given in Table~\ref{tab:domains}. The explicit coefficients of $\bar u$ can be found in \cite{program}.
\end{theorem}

The significance of Theorem~\ref{thm:conformal-stability} is that its
conclusion concerns an open subset of the infinite-dimensional space
$\mathcal A_\sigma$. The role of the computer-assisted estimates is to
identify reference points $r_0$ and to produce
explicit lower bounds for the radii $\delta_r$; the persistence result
itself follows from the quantitative estimates developed in
Sections~\ref{sec:aposteriori}--\ref{sec:stability}.

The proof is based on the structure of the linearized operator
\[
D F_r(u)
=
I-\mathcal H M_{\psi_{r,u}},
\qquad
\psi_{r,u}
=
\frac{r'(\,\cdot+u(\cdot)\,)}
{r(\,\cdot+u(\cdot)\,)}.
\]
The Wiener--Hopf analysis of Section~\ref{sec:WH} gives an explicit inverse at the
reference configuration, while
\eqref{eq:composition-main} and the corresponding logarithmic
derivative estimate control its variation with respect to $r$ and
$u$.

The reference boundary functions are of the form
\[
r_0(t)
=
1+\sum_{k=1}^{N}
\left(a_k\cos(kt)+b_k\sin(kt)\right),
\]
where the coefficients are specified in Table~\ref{tab:domains}.  The solutions to \eqref{eq:elliptic-intro} are positive, except for those marked with an asterisk, which take both signs.

\begin{table}
	\centering\small
	\begin{tabular}{lllrcc}
		Solution & $\{A_1,\ldots,A_N\}$ & $\{B_1,\ldots,B_N\}$ & $d_r$ & $r_u$ & $d_\mathrm{PDE}$\\
		\hline
		Triblob
		& $\{51,-100,55,80\}$
		& $\{0,50,-100,-90\}$
		& 24&16&48\\
		
		Triblob$_*$
		& $\{51,-100,55,80\}$
		& $\{0,50,-100,-90\}$
		& 24&16&58\\
		
		Tripuff
		& $\{51,-1,55,80,1\}$
		& $\{50,-10,-90,20,12\}$
		& 23&16&44\\
		
		Cusp
		& $\{500\}$
		& $\{800\}$
		& 35&23&61\\
		
		Eight
		& $\{0,200\}$
		& $\{200,-500\}$
		& 30&20&61\\
		
		Eight$_*$
		& $\{0,200\}$
		& $\{200,-500\}$
		& 30&20&64\\
		
		Five Star
		& $\{0,0,0,0,0\}$
		& $\{0,0,0,0,180\}$
		& 23&16&46\\
		
		Five Star$_*$
		& $\{0,0,0,0,0\}$
		& $\{0,0,0,0,120\}$
		& 22&16&53\\
		
		Mushroom
		& $\{100,0,0,200\}$
		& $\{300,0,0,0\}$
		& 27&18&55\\
		
		Shamrock
		& $\{0,0,300\}$
		& $\{0,0,0\}$
		& 23&16&45\\
		
		Six Star
		& $\{0,0,0,0,0,100\}$
		& $\{0,0,0,0,0,0\}$
		& 22&16&39\\
		
		Six Star$_*$
		& $\{0,0,0,0,0,100\}$
		& $\{0,0,0,0,0,0\}$
		& 22&16&51\\
		
		Pillow
		& $\{-51,50,-55,-80\}$
		& $\{0,30,5,-10\}$
		& 22&16&43\\
		
		Pillow$_*$
		& $\{-51,50,-55,-80\}$
		& $\{0,30,5,-10\}$
		& 22&16&55\\
		
		Square
		& $\{0,0,0,-143,0,0,0,45\}$
		& $\{0,0,0,0,0,0,0,0\}$
		& 23&16&47\\
	\end{tabular}
\bigskip
	
	\caption{Parameters defining the domains considered in the computer-assisted validation and values of the constants defined in Theorems \ref{thm:conformal-stability} and \ref{thm:elliptic-stability}. More precisely, to keep the numbers exact and the table readable, we write $a_k=2^{-10}A_k$, $b_k=2^{-10}B_k$ and $\delta_r=2^{-d_r}$,  $R_u=2^{-r_u}$, $\delta_\mathrm{PDE}=2^{-d_\mathrm{PDE}}$.}
	\label{tab:domains}
\end{table}

We next introduce the functional setting for the elliptic problem.
Let
\[
\mathcal C_\varrho
=
\left\{
u(r,\vartheta)
=
\sum_{m,l\geq0}
R_{m+2l}^{m}(r)
\left[
a_{m,l}\cos(m\vartheta)
+
b_{m,l}\sin(m\vartheta)
\right]:
\|u\|_\varrho<\infty
\right\},
\]
where
\[
\|u\|_\varrho
=
\sum_{m,l\geq0}
\bigl(|a_{m,l}|+|b_{m,l}|\bigr)
\varrho^{m+2l},
\]
with $b_{0,l}=0$. Here $R_{m+2l}^{m}$ denotes the corresponding
Zernike polynomial. Note that $\mathcal C_\varrho$ is a Banach algebra \cite{ArioliKoch1}.

For every $r$ covered by
Theorem~\ref{thm:conformal-stability}, let $f_r$ be the corresponding
Riemann map and set
\[
q_r=|f_r'|^2.
\]
After possibly decreasing the exponential weight, the estimates of
Appendix~\ref{sec:appendix-wiener} and the Taylor-to-Zernike correspondence of Appendix~\ref{sec:appendix-zernike} yield
\[
q_r\in\mathcal C_\varrho
\]
with quantitative bounds uniform for
\[
\|r-r_0\|_\sigma<\delta_r.
\]

Our second main result concerns the persistence of solutions of the
semilinear equation under such perturbations.

\begin{theorem}[Persistence of elliptic solutions under domain
	perturbations]\label{thm:elliptic-stability}
	Let $\sigma=2$ and $\varrho=513/512$.
	For each reference boundary function $r_0$ listed in
	Table~\ref{tab:domains}, there exists an explicitly computable
	constant
	\[
	0<\delta_{\mathrm{PDE}}\leq\delta_r
	\]
	such that, for every
	\[
	r\in\mathcal A_\sigma,
	\qquad
	\|r-r_0\|_\sigma<\delta_{\mathrm{PDE}},
	\]
	the Dirichlet problem
	\[
	\begin{cases}
		-\Delta v=v^3,&\text{in }\Omega_r,\\
		v=0,&\text{on }\partial\Omega_r
	\end{cases}
	\]
	admits a solution $v_r$.
	
	More precisely, after pullback by the normalized Riemann map $f_r$,
	the solution belongs to a fixed ball in $\mathcal C_\varrho$ around
	the validated reference solution. The solution \(v_r\) is positive for the unstarred reference configurations and sign-changing for the starred ones.
	
	The radii $\delta_{\mathrm{PDE}}$ are displayed in Table \ref{tab:domains}, while the explicit coefficients in $\mathcal C_\varrho$ of the validated reference solution corresponding to $r_0$ can be found in \cite{program}.
\end{theorem}
\begin{remark}
The values of \(\delta_r\), \(R_u\), and \(\delta_{\rm PDE}\) are not intended to be sharp. They could all be improved by carrying out the numerical computations with greater accuracy, at the cost of a substantial increase in computational effort, which is already considerable.\end{remark}
Theorem~\ref{thm:elliptic-stability} is obtained by combining
Theorem~\ref{thm:conformal-stability} with uniform a posteriori
estimates for
\[
-\Delta v=q_r v^3\qquad\text{in }\mathbb D.
\]
Thus the existence result is stable with respect to perturbations of
the geometry itself, rather than only with respect to perturbations of
the coefficient in a fixed-domain equation.

% ============================================================
\section{Conformal maps and the Theodorsen equation}
\label{sec:theo}
% ============================================================

Let $\mathbb D\subset\mathbb C$ denote the open unit disk centred at 0. Given a
positive $2\pi$-periodic function $r$, we denote by $\Omega_r$ the
star-shaped domain bounded by
\begin{equation}\label{eq:polarboundary}
	\gamma_r(\theta)
	=
	r(\theta)e^{i\theta},
	\qquad
	\theta\in\mathbb R.
\end{equation}
Whenever $\gamma_r$ is a positively oriented Jordan curve, the Riemann
mapping theorem yields a unique biholomorphic map
\[
f_r:\mathbb D\longrightarrow\Omega_r
\]
satisfying
\[
f_r(0)=0,
\qquad
f_r'(0)>0.
\]

The boundary correspondence associated with $f_r$ can be described by
the Theodorsen equation. We first recall this relation and then
formulate the equation in the weighted Wiener spaces introduced in
Section~\ref{sec:main-results}.

Let $\mathcal H$ denote the periodic Hilbert transform, defined on
Fourier series by
\begin{equation}\label{eq:Hilbert}
	\mathcal H
	\left(
	\sum_{k\in\mathbb Z}g_ke^{ikt}
	\right)
	=
	-i\sum_{k\ne0}
	\operatorname{sgn}(k)g_ke^{ikt}.
\end{equation}
Thus $\mathcal H$ annihilates the constant Fourier mode. With this
convention, if $h$ is real-valued, then
\[
h+i\mathcal Hh
\]
contains only nonnegative Fourier modes.

The angular parametrisation of $\partial\Omega_r$ need not agree with
the angular parametrisation induced by the Riemann map. We therefore
write
\[
\theta(t)=t+u(t),
\]
where $u$ is $2\pi$-periodic. The Theodorsen equation is
\begin{equation}\label{eq:Theodorsen}
	u(t)
	=
	\mathcal H
	\left[
	\log r\bigl(t+u(t)\bigr)
	\right].
\end{equation}
The following classical proposition identifies its solutions with the
boundary correspondence of the Riemann map; see also
\cite{Gutknecht, Hubner,Kythe,Wegmann1978,Wegmann1986}.

\begin{proposition}\label{prop:Theodorsen}
	Let $\gamma_r$ be a smooth positively oriented Jordan curve of the
	form \eqref{eq:polarboundary}, and let
	$u:\mathbb R\to\mathbb R$ be a $2\pi$-periodic solution of
	\eqref{eq:Theodorsen}. Assume that
	\[
	\theta(t)=t+u(t)
	\]
	induces an orientation-preserving homeomorphism of
	$\mathbb R/2\pi\mathbb Z$. Then
	\begin{equation}\label{eq:boundarymap}
		f_r(e^{it})
		=
		r(\theta(t))e^{i\theta(t)}
	\end{equation}
	is the boundary trace of the normalized Riemann map
	\[
	f_r:\mathbb D\longrightarrow\Omega_r,
	\qquad
	f_r(0)=0,
	\qquad
	f_r'(0)>0.
	\]
\end{proposition}

\begin{proof}
	Set
	\[
	h(t):=\log r(\theta(t)).
	\]
	Since $\theta(t)=t+u(t)$,
	\begin{equation}\label{eq:boundaryfactor}
		r(\theta(t))e^{i\theta(t)}
		=
		e^{it}\exp\bigl(h(t)+iu(t)\bigr).
	\end{equation}
	By \eqref{eq:Theodorsen},
	\[
	u=\mathcal Hh.
	\]
	Hence
	\[
	h+i\mathcal Hh
	\]
	contains only nonnegative Fourier modes and is the boundary trace of
	a holomorphic function $G$ on $\mathbb D$. Since its constant Fourier
	coefficient is the mean of $h$ because \(\mathcal Hh\) has zero mean, then
	\[
	G(0)\in\mathbb R.
	\]
	
	Define
	\[
	f(z):=z e^{G(z)}.
	\]
	Then $f$ is holomorphic in $\mathbb D$,
	\[
	f(0)=0,
	\qquad
	f'(0)=e^{G(0)}>0,
	\]
	and \eqref{eq:boundaryfactor} gives
	\[
	f(e^{it})
	=
	r(\theta(t))e^{i\theta(t)}
	=
	\gamma_r(\theta(t)).
	\]
	Since $\theta$ is an orientation-preserving homeomorphism of the
	circle, this boundary curve traverses $\partial\Omega_r$ exactly once
	in the positive direction.
	
	For $w\notin\partial\Omega_r$, the argument principle therefore gives
	\[
	N(f-w)
	=
	\operatorname{Ind}(\partial\Omega_r,w)
	=
	\begin{cases}
		1,&w\in\Omega_r,\\
		0,&w\in\mathbb C\setminus\overline{\Omega_r},
	\end{cases}
	\]
	where zeros are counted with multiplicity. It follows that
	$f(\mathbb D)=\Omega_r$ and that every point of $\Omega_r$ has exactly
	one preimage. Thus $f$ is biholomorphic and, by its normalisation,
	$f=f_r$.
\end{proof}

\subsection{The Theodorsen operator}

For a positive boundary function $r$, define
\begin{equation}\label{eq:Fr}
	\mathcal F_r(u)
	:=
	u
	-
	\mathcal H
	\left[
	\log r\bigl(\,\cdot+u(\cdot)\bigr)
	\right].
\end{equation}
Thus the Theodorsen equation is
\[
\mathcal F_r(u)=0.
\]
It is useful to retain explicitly the dependence on $r$, since one of
our main objectives is to compare the solutions corresponding to
nearby boundary functions.

Equivalently, define
\begin{equation}\label{eq:GTheodorsen}
	\mathcal G_r(u)
	:=
	\mathcal H
	\left[
	\log r\bigl(\,\cdot+u(\cdot)\bigr)
	\right].
\end{equation}
Then
\[
\mathcal F_r=I-\mathcal G_r.
\]

Whenever the composition is well defined in the weighted Wiener
algebra, differentiation with respect to $u$ gives
\begin{equation}\label{eq:DG}
	D\mathcal G_r(u)
	=
	\mathcal H M_{\psi_{r,u}},
	\qquad
	\psi_{r,u}
	:=
	\frac{r'}{r}
	\bigl(\,\cdot+u(\cdot)\bigr),
\end{equation}
and therefore
\begin{equation}\label{eq:DF}
	D\mathcal F_r(u)
	=
	I-\mathcal H M_{\psi_{r,u}}.
\end{equation}
The quantitative dependence of $\psi_{r,u}$ on both $r$ and $u$ will
be studied in Section~\ref{sec:stability}.

\subsection{Construction and analytic continuation of the Riemann map}

Let $u$ be a solution of \eqref{eq:Theodorsen} satisfying the
hypotheses of Proposition~\ref{prop:Theodorsen}, and define
\begin{equation}\label{eq:Rboundary}
	R_r(t)
	:=
	r\bigl(t+u(t)\bigr)e^{i(t+u(t))}.
\end{equation}
By Proposition~\ref{prop:Theodorsen}, $R_r$ is the boundary trace of
$f_r$. Consequently,
\begin{equation}\label{eq:RFourier}
	R_r(t)
	=
	\sum_{k\ge1}f_{r,k}e^{ikt},
\end{equation}
and
\begin{equation}\label{eq:conformal}
	f_r(z)
	=
	\sum_{k\ge1}f_{r,k}z^k.
\end{equation}

The weighted Wiener setting gives quantitative information on the
analytic continuation of $f_r$.

\begin{lemma}\label{lem:conformalradius}
	Suppose that $R_r\in\mathcal B_\rho$ and that its nonpositive Fourier
	modes vanish. Then
	\[
	\sum_{k\ge1}|f_{r,k}|\rho^k<\infty.
	\]
	Consequently, the series \eqref{eq:conformal} converges absolutely and
	uniformly on $|z|\le\rho$ and defines a holomorphic function on
	$|z|<\rho$.
\end{lemma}

\begin{proof}
	Since $R_r\in\mathcal B_\rho$,
	\[
	\sum_{k\in\mathbb Z}|(R_r)_k|\rho^{|k|}<\infty.
	\]
	The assumption on the Fourier modes gives
	\[
	(R_r)_k=
	\begin{cases}
		f_{r,k},&k\ge1,\\
		0,&k\le0,
	\end{cases}
	\]
	and hence
	\[
	\sum_{k\ge1}|f_{r,k}|\rho^k<\infty.
	\]
	The remaining assertions follow immediately.
\end{proof}

Thus quantitative control of a solution of the Theodorsen equation
in a weighted Wiener algebra yields quantitative control of the
corresponding Riemann map. In Sections~\ref{sec:aposteriori}--
\ref{sec:stability} we derive estimates which remain uniform when the
boundary function varies in a neighborhood in $\mathcal A_\sigma$.

% ============================================================
\section{Quantitative a posteriori analysis of the Theodorsen equation}
\label{sec:aposteriori}
% ============================================================

We now formulate the quantitative existence argument. Let $r$ be a positive boundary
function and let $\bar u\in\mathcal A_\rho$ be an approximate zero of
$\mathcal F_r$. For $R>0$, write
\[
B_R(\bar u)
=
\left\{
u\in\mathcal A_\rho:
\|u-\bar u\|_\rho\le R
\right\}.
\]

There are two useful forms of the argument. The first applies when the
Theodorsen map itself is contractive; the second uses an inverse, or
approximate inverse, of its linearisation.

\subsection{Direct contraction}
The following proposition makes the computer assisted part very straightforward, but it only applies to domains that are sufficiently close to a disk. We do not use it, but we include it here for completeness.
\begin{proposition}\label{prop:direct-contraction}
	Let
	\[
	Y
	:=
	\|\mathcal G_r(\bar u)-\bar u\|_\rho.
	\]
	Assume that there exist $R>0$ and $0\le\kappa<1$ such that
	\begin{equation}\label{eq:direct-derivative}
		\sup_{u\in B_R(\bar u)}
		\|\psi_{r,u}\|_\rho
		\le\kappa
	\end{equation}
	and
	\begin{equation}\label{eq:direct-radius}
		Y+\kappa R\le R.
	\end{equation}
	Then the Theodorsen equation has a unique solution in
	$B_R(\bar u)$.
\end{proposition}

\begin{proof}
	Since $\|\mathcal H\|\le1$, equation \eqref{eq:DG} gives
	\[
	\sup_{u\in B_R(\bar u)}
	\|D\mathcal G_r(u)\|
	\le\kappa.
	\]
	Moreover,
	\[
	\|\mathcal G_r(\bar u)-\bar u\|_\rho
	\le Y.
	\]
	Condition \eqref{eq:direct-radius} shows that
	$\mathcal G_r$ maps $B_R(\bar u)$ into itself. The contraction mapping
	theorem gives the conclusion.
\end{proof}

\subsection{Newton-like a posteriori theorem}

For general star-shaped domains the norm of
$D\mathcal G_r(\bar u)$ need not be smaller than one. Then a Newton map may be used to prove the existence of a fixed point. Let $L$ be an
invertible bounded operator which approximates
$D\mathcal F_r(\bar u)^{-1}$ and define
\begin{equation}\label{eq:NewtonMap}
	\mathcal N_r(u):=u-L\mathcal F_r(u).
\end{equation}
Since $L$ is invertible,
\[
\mathcal N_r(u)=u
\quad\Longleftrightarrow\quad
\mathcal F_r(u)=0.
\]
Furthermore,
\begin{equation}\label{eq:DN}
	D\mathcal N_r(u)
	=
	I-LD\mathcal F_r(u).
\end{equation}

The following elementary result is the form of the a posteriori
theorem that will later be used uniformly with respect to $r$.

\begin{proposition}\label{prop:newtonvalidation}
	Let
	\[
	Y:=\|\mathcal N_r(\bar u)-\bar u\|_\rho.
	\]
	Suppose that there exist $R>0$ and $0\le\kappa<1$ such that
	\begin{equation}\label{eq:derivativebound}
		\sup_{u\in B_R(\bar u)}
		\|I-LD\mathcal F_r(u)\|
		\le\kappa
	\end{equation}
	and
	\begin{equation}\label{eq:newtoncontraction}
		Y+\kappa R\le R.
	\end{equation}
	Then $\mathcal F_r$ has a unique zero in $B_R(\bar u)$.
\end{proposition}

\begin{proof}
	Equation \eqref{eq:derivativebound} implies that $\mathcal N_r$ is a
	contraction on $B_R(\bar u)$, while
	\eqref{eq:newtoncontraction} implies that this ball is invariant under
	$\mathcal N_r$. The contraction mapping theorem gives a unique fixed
	point of $\mathcal N_r$ in the ball, and the invertibility of $L$
	identifies this fixed point with a zero of $\mathcal F_r$.
\end{proof}

The essential issue is therefore to obtain an effective inverse of
\[
D\mathcal F_r(u)
=
I-\mathcal H M_{\psi_{r,u}}.
\]
This operator is not a compact perturbation of the identity in the
weighted Wiener algebra. Its particular structure, however, permits
an explicit inversion through a Wiener--Hopf factorisation. This is
the subject of the next section.

% ============================================================
\section{Wiener--Hopf inversion of the linearized Theodorsen operator}
\label{sec:WH}
% ============================================================

We study bounded operators on $\mathcal B_\rho$ of the form
\begin{equation}\label{eq:Phi_a}
	\Phi_a
	:=
	I+\mathcal H M_a,
\end{equation}
where $a\in\mathcal B_\rho$. For the linearized Theodorsen equation,
\[
a=-\psi_{r,u}.
\]
The purpose of this section is to characterize the invertibility of
$\Phi_a$ in terms of a scalar Wiener--Hopf symbol and to derive an
explicit representation of its inverse.

Let
\[
P_+u=\sum_{n>0}u_nz^n,
\qquad
P_-u=\sum_{n<0}u_nz^n,
\qquad
P_0u=u_0.
\]
Then
\[
I=P_++P_-+P_0,
\qquad
\mathcal H=-iP_++iP_-.
\]
We also introduce
\[
\mathcal B_\rho^+
=
\left\{
\sum_{n\ge0}u_nz^n\in\mathcal B_\rho
\right\},
\qquad
\mathcal B_\rho^-
=
\left\{
\sum_{n\le0}u_nz^n\in\mathcal B_\rho
\right\},
\]
and
\[
H^+:=P_+(\mathcal B_\rho),
\qquad
H^-:=P_-(\mathcal B_\rho).
\]

Set
\begin{equation}\label{eq:Apm}
	A_+:=1-ia,
	\qquad
	A_-:=1+ia.
\end{equation}
Then
\[
P_+(\Phi_a u)=P_+(A_+u),
\qquad
P_-(\Phi_a u)=P_-(A_-u),
\qquad
P_0(\Phi_a u)=P_0u.
\]
Consequently, $\Phi_a u=g$ is equivalent to
\begin{equation}\label{eq:RHsystem}
	P_+(A_+u)=g_+,
	\qquad
	P_-(A_-u)=g_-,
	\qquad
	P_0u=g_0.
\end{equation}

Assume that $A_+$ is invertible in $\mathcal B_\rho$ and introduce
the Wiener--Hopf symbol
\begin{equation}\label{eq:Wsymbol}
	W:=A_-A_+^{-1}.
\end{equation}
Suppose that $W$ admits a canonical Wiener--Hopf factorisation
\begin{equation}\label{eq:WHfactorisation}
	W=W_-W_+,
\end{equation}
where
\[
W_-^{\pm1}\in\mathcal B_\rho^-,
\qquad
W_+^{\pm1}\in\mathcal B_\rho^+.
\]
For example, this holds when the standard nonvanishing and zero-index
conditions are satisfied.

Set
\[
\alpha:=W_-^{-1},
\qquad
\beta:=W_+^{-1},
\]
and define
\[
T_W^-:H^-\longrightarrow H^-,
\qquad
T_W^-:=P_-M_W|_{H^-}.
\]

\begin{lemma}\label{lem:ToeplitzInverse}
	The operator $T_W^-$ is invertible and
	\begin{equation}\label{eq:Boperator}
		B:=(T_W^-)^{-1}
		=
		M_\alpha P_-M_\beta.
	\end{equation}
	In particular,
	\begin{equation}\label{eq:Bnorm}
		\|B\|
		\le
		\|\alpha\|_\rho\|\beta\|_\rho.
	\end{equation}
\end{lemma}

\begin{proof}
	The inversion formula is the standard one associated with a canonical
	Wiener--Hopf factorisation; see, for instance,
	\cite{BottcherSilbermann}. The triangular structure of the factors
	with respect to the positive and negative Fourier modes gives
	\[
	T_W^-M_\alpha P_-M_\beta=I_{H^-},
	\]
	and similarly on the other side. Estimate \eqref{eq:Bnorm} follows
	from the Banach algebra property and $\|P_-\|=1$.
\end{proof}

We can now give an explicit representation of the inverse of
$\Phi_a$. Set
\[
h:=A_+^{-1},
\]
and define
\begin{equation}\label{eq:Qoperator}
	\mathcal Q
	:=
	P_--P_-M_WP_+,
	\qquad
	p:=P_-W,
\end{equation}
\begin{equation}\label{eq:CDF}
	C:=1-Bp,
\end{equation}
\begin{equation}\label{eq:Dscalar}
	D:=P_0(hC),
\end{equation}
and
\begin{equation}\label{eq:Soperator}
	\mathcal S
	:=
	P_0-P_0M_hP_+-P_0M_hB\mathcal Q.
\end{equation}
Here $D$ is a scalar.

\begin{theorem}\label{thm:PhiInverse}
	Assume that $A_+$ is invertible in $\mathcal B_\rho$, that the symbol
	$W$ admits the canonical factorisation
	\eqref{eq:WHfactorisation}, and that
	\[
	D\neq0.
	\]
	Then $\Phi_a$ is invertible and
	\begin{equation}\label{eq:explicitL}
		\Phi_a^{-1}
		=
		M_h
		\left(
		P_+
		+
		B\mathcal Q
		+
		CD^{-1}\mathcal S
		\right).
	\end{equation}
	Moreover,
	\begin{equation}\label{eq:inverse-norm}
		\|\Phi_a^{-1}\|
		\le
		\|h\|_\rho
		\left(
		1+
		\|B\|\,\|\mathcal Q\|
		+
		\|C\|\,|D|^{-1}\|\mathcal S\|
		\right).
	\end{equation}
\end{theorem}

\begin{proof}
	Let
	\[
	v:=A_+u.
	\]
	The first two equations in \eqref{eq:RHsystem} become
	\[
	P_+v=g_+,
	\qquad
	P_-(Wv)=g_-.
	\]
	Write
	\[
	v=g_++\xi+v_-,
	\qquad
	\xi\in\mathbb C,
	\qquad
	v_-\in H^-.
	\]
	Then
	\[
	P_-(Wv_-)
	=
	g_--P_-(Wg_+)-\xi P_-W,
	\]
	and therefore
	\[
	T_W^-v_-
	=
	g_--P_-(Wg_+)-\xi p.
	\]
	By Lemma~\ref{lem:ToeplitzInverse},
	\[
	v_-
	=
	B\mathcal Qg-\xi Bp.
	\]
	Thus
	\[
	v
	=
	P_+g+B\mathcal Qg+\xi C.
	\]
	
	Since $u=hv$, the condition $P_0u=g_0$ gives
	\[
	P_0M_h
	\left(
	P_+g+B\mathcal Qg+\xi C
	\right)
	=
	P_0g.
	\]
	Equivalently,
	\[
	\xi D=\mathcal Sg.
	\]
	Since $D\ne0$,
	\[
	\xi=D^{-1}\mathcal Sg.
	\]
	Substitution yields \eqref{eq:explicitL}.
	
	Finally, taking operator norms in \eqref{eq:explicitL} and using
	$\|P_+\|=1$ gives \eqref{eq:inverse-norm}.
\end{proof}

\begin{remark}\label{rem:quantitative-inverse}
	The importance of Theorem~\ref{thm:PhiInverse} for the stability
	problem is quantitative. Invertibility of the linearized Theodorsen
	operator is reduced to estimates on the elements
	\[
	h,\quad W,\quad \alpha,\quad\beta,\quad C,\quad D
	\]
	of the weighted Wiener algebra. In particular, the right-hand side of
	\eqref{eq:inverse-norm} gives an explicit bound for the inverse which
	can be propagated under perturbations of the boundary function.
\end{remark}

For the rigorous computations we replace the objects entering
\eqref{eq:explicitL} by finite Laurent approximations and bound the
resulting residual in operator norm. Those finite-dimensional
estimates will be described separately in
Section~\ref{sec:verification}; they are not needed for the abstract
stability argument developed next.

% ============================================================
\section{Stability under perturbations of the boundary}
\label{sec:stability}
% ============================================================

We now study the dependence of the Theodorsen equation on the
boundary function. The main point is that the analyticity required
for the composition
\[
z\longmapsto r(z+u(z))
\]
can be controlled by separating the validated solution into the
numerically known approximate solution and a small remainder. This
allows us to exploit a bound on the imaginary part of the approximate
solution in a larger complex strip, while measuring the unknown
remainder only in the Wiener norm in which the contraction argument
is performed.

Throughout this section we use four weights
\begin{equation}\label{eq:four-weights}
	1<\rho<\widehat\rho<\tau<\sigma.
\end{equation}
Their roles are as follows:
\[
\begin{array}{cl}
	\sigma
	&
	\text{is used for the boundary perturbation }r-r_0,
	\\[1mm]
	\tau
	&
	\text{is used for }r^{-1},\,r'/r
	\text{ and their derivatives},
	\\[1mm]
	\widehat\rho
	&
	\text{is an auxiliary weight used to control the analytic continuation
		of the approximate solution},
	\\[1mm]
	\rho
	&
	\text{is the weight in which the Theodorsen equation and its stability
		are validated}.
\end{array}
\]

For $\eta>1$, set
\[
S_\eta
:=
\left\{
z\in\mathbb C:
|\operatorname{Im}z|<\log\eta
\right\}.
\]

Let $r_0\in\mathcal A_\sigma$ be a reference boundary function and
let
\[
\bar u\in\mathcal A_{\widehat\rho}
\]
be an approximate solution of the corresponding Theodorsen equation.
Fix $R>0$ and set
\begin{equation}\label{eq:U}
	\mathcal U
	:=
	B_R(\bar u)
	=
	\left\{
	u\in\mathcal A_\rho:
	\|u-\bar u\|_\rho\le R
	\right\}.
\end{equation}
Thus every $u\in\mathcal U$ can be written as
\[
u=\bar u+h,
\qquad
\|h\|_\rho\le R.
\]

The analyticity information used below concerns only the known
approximate solution $\bar u$. Set
\begin{equation}\label{eq:dbar}
	\bar d
	:=
	\sup_{z\in S_{\widehat\rho}}
	|\operatorname{Im}\bar u(z)|.
\end{equation}
We assume that
\begin{equation}\label{eq:composition-margin}
	\rho e^{\bar d+R}<\tau.
\end{equation}

% ------------------------------------------------------------
\subsection{Composition estimates}
% ------------------------------------------------------------

We first derive a composition estimate which exploits the strip bound
\eqref{eq:dbar} without requiring the unknown function $u$ to belong
to $\mathcal A_{\widehat\rho}$.

We shall use the following strip-to-Wiener estimate.
\begin{lemma}\label{lem:strip-to-wiener}
	Let $H$ be a $2\pi$-periodic function, holomorphic in
	$S_{\widehat\rho}$, and suppose that
	\[
	\sup_{z\in S_{\widehat\rho}}|H(z)|\le M.
	\]
	Let
	\[
	H(t)=\sum_{k\in\mathbb Z}\widehat H_k e^{ikt}
	\]
	be its Fourier expansion on the real axis, and define the associated
	Laurent function
	\[
	\widetilde H(\zeta)
	:=
	\sum_{k\in\mathbb Z}\widehat H_k\zeta^k.
	\]
	Then $\widetilde H\in\mathcal B_\rho$ and
	\begin{equation}\label{eq:strip-to-wiener}
		\|\widetilde H\|_{\mathcal B_\rho}
		\le
		C_{\rho,\widehat\rho}^{\rm strip}M,
	\end{equation}
	where
	\begin{equation}\label{eq:Cstrip}
		C_{\rho,\widehat\rho}^{\rm strip}
		:=
		\frac{\widehat\rho+\rho}
		{\widehat\rho-\rho}.
	\end{equation}
\end{lemma}

\begin{proof}
	For $k\ne0$, shifting the contour in the formula for the Fourier
	coefficient $\widehat H_k$ to
	\[
	\operatorname{Im}z=\pm\log\widehat\rho,
	\]
	or first to a slightly smaller strip and then passing to the limit,
	gives
	\[
	|\widehat H_k|
	\le
	M\widehat\rho^{-|k|}.
	\]
	For $k=0$ we simply have
	\[
	|\widehat H_0|\le M.
	\]
	Thus
	\[
	|\widehat H_k|
	\le
	M\widehat\rho^{-|k|},
	\qquad k\in\mathbb Z.
	\]
	Consequently,
	\[
		\|\widetilde H\|_{\mathcal B_\rho}=
		\sum_{k\in\mathbb Z}|\widehat H_k|\rho^{|k|}\le
		M\sum_{k\in\mathbb Z}\left(\frac{\rho}{\widehat\rho}\right)^{|k|}=
		M\left[1+2\sum_{k\ge1}\left(\frac{\rho}{\widehat\rho}\right)^k\right]=
		M\frac{\widehat\rho+\rho}{\widehat\rho-\rho}.
	\]
	This proves the result.
\end{proof}
For later use, set
\[
C_{\rm strip}
:=
C_{\rho,\widehat\rho}^{\rm strip}
=
\frac{\widehat\rho+\rho}
{\widehat\rho-\rho}.
\]
Since, by Lemma~\ref{lem:AandB}, the real Fourier and symmetric
Laurent norms satisfy
\[
\|f\|_{\mathcal B_\rho}
\le
\|f\|_{\mathcal A_\rho}
\le
\sqrt2\,\|f\|_{\mathcal B_\rho},
\]
we also define
\begin{equation}\label{eq:Ccomp}
	C_{\rm comp}
	:=
	\sqrt2\,C_{\rm strip}
	=
	\sqrt2\,
	\frac{\widehat\rho+\rho}
	{\widehat\rho-\rho}.
\end{equation}
We can now obtain the required composition estimate.

\begin{proposition}\label{prop:composition-strip}
	Let $g\in\mathcal A_\tau$ and let
	\[
	u=\bar u+h\in\mathcal U.
	\]
	If \eqref{eq:composition-margin} holds, then
	\[
	g(\,\cdot+u(\cdot)\,)\in\mathcal A_\rho
	\]
	and
	\begin{equation}\label{eq:composition-strip}
		\|g(\,\cdot+u(\cdot)\,)\|_\rho
		\le
		C_{\rm comp}\|g\|_\tau.
	\end{equation}
	The estimate is uniform for $u\in\mathcal U$.
\end{proposition}

\begin{proof}
	Write $g$ in complex Fourier form,
	\[
	g(t)
	=
	\sum_{k\in\mathbb Z}
	\widehat g_k e^{ikt}.
	\]
	For $u=\bar u+h$,
	\[
	e^{iku}
	=
	e^{ik\bar u}e^{ikh}.
	\]
	
	By \eqref{eq:dbar}, for $z\in S_{\widehat\rho}$,
	\[
	|e^{ik\bar u(z)}|
	=
	e^{-k\operatorname{Im}\bar u(z)}
	\le
	e^{|k|\bar d}.
	\]
	Lemma~\ref{lem:strip-to-wiener} therefore gives
	\begin{equation}\label{eq:exp-ubar}
		\|e^{ik\bar u}\|_{\mathcal B_\rho}
		\le
		C_{\rm strip}e^{|k|\bar d}.
	\end{equation}
	
	On the other hand, since $\mathcal A_\rho$ is a Banach algebra and
	$\|h\|_\rho\le R$,
	\[
	\|e^{ikh}\|_\rho
	\le
	e^{|k|\|h\|_\rho}
	\le
	e^{|k|R}.
	\]
	Hence, using the Banach algebra property in the corresponding
	complex Fourier representation,
	\begin{equation}\label{eq:exp-u}
		\|e^{iku}\|_{\mathcal B_\rho}
		\le
		C_{\rm strip}
		e^{|k|(\bar d+R)}.
	\end{equation}
	
	We now have
	\[
	g(t+u(t))
	=
	\sum_{k\in\mathbb Z}
	\widehat g_k e^{ikt}e^{iku(t)}.
	\]
	Since
	\[
	\|e^{ikt}\|_{\mathcal B_\rho}
	=
	\rho^{|k|},
	\]
	it follows from \eqref{eq:exp-u} that
	\[
		\|g(\,\cdot+u(\cdot)\,)\|_{\mathcal B_\rho}\le C_{\rm strip}\sum_{k\in\mathbb Z}
		|\widehat g_k|\left(\rho e^{\bar d+R}\right)^{|k|}\le
		C_{\rm strip}\sum_{k\in\mathbb Z}|\widehat g_k|\tau^{|k|}=
		C_{\rm strip}\|g\|_{\mathcal B_\tau},
	\]
	where we used \eqref{eq:composition-margin}.
	Therefore,
	\[
		\|g(\,\cdot+u(\cdot)\,)\|_{\mathcal A_\rho}\le \sqrt2\,
		\|g(\,\cdot+u(\cdot)\,)\|_{\mathcal B_\rho}
		\le \sqrt2\,C_{\rm strip}\|g\|_{\mathcal B_\tau}
		\le C_{\rm comp}\|g\|_{\mathcal A_\tau}.
	\]
\end{proof}

In particular, if $r_1,r_2\in\mathcal A_\sigma$, then
\begin{equation}\label{eq:r-composition-difference}
	\left\|
	r_1(\,\cdot+u\,)
	-
	r_2(\,\cdot+u\,)
	\right\|_\rho
	\le
	C_{\rm comp}
	\|r_1-r_2\|_\sigma,
\end{equation}
uniformly for $u\in\mathcal U$.

% ------------------------------------------------------------
\subsection{The logarithmic derivative}\label{ssec:logarithmic-derivative}
% ------------------------------------------------------------

The derivative of the Theodorsen operator involves
\[
\psi_{r,u}
=
\frac{r'}r(\,\cdot+u(\cdot)\,).
\]
We next control this quantity under perturbations of $r$.

Assume that
\begin{equation}\label{eq:r0-inverse-tau}
	r_0^{-1}\in\mathcal A_\tau,
	\qquad
	M_0:=\|r_0^{-1}\|_\tau.
\end{equation}
Let
\[
r=r_0+h,
\qquad
\|h\|_\sigma\le\delta.
\]
Since
\[
r=r_0(1+r_0^{-1}h),
\]
if
\begin{equation}\label{eq:inverse-condition}
	M_0\delta<1,
\end{equation}
then every
\[
r\in B_\delta(r_0)\subset\mathcal A_\sigma
\]
is invertible in $\mathcal A_\tau$, with
\begin{equation}\label{eq:Mdelta}
	\|r^{-1}\|_\tau
	\le
	M_\delta,
	\qquad
	M_\delta
	:=
	\frac{M_0}{1-M_0\delta}.
\end{equation}

Differentiation from $\mathcal A_\sigma$ to $\mathcal A_\tau$ gives
\begin{equation}\label{eq:fourier-derivative}
	\|h'\|_\tau
	\le
	C_{\sigma,\tau}^{(1)}
	\|h\|_\sigma,
\end{equation}
where
\begin{equation}\label{eq:C1}
	C_{\sigma,\tau}^{(1)}
	:=
	\sup_{k\ge1}
	k\left(\frac{\tau}{\sigma}\right)^k
	\le
	\frac1{e\log(\sigma/\tau)}.
\end{equation}

Set
\begin{equation}\label{eq:Kdelta}
	K_\delta:=\|r_0'\|_\tau+C_{\sigma,\tau}^{(1)}\delta.
\end{equation}
Then
\[
\|r'\|_\tau\le K_\delta
\]
for every $r\in B_\delta(r_0)$.

\begin{lemma}\label{lem:logderivative}
	For every
	\[
	r_1,r_2\in B_\delta(r_0)
	\]
	one has
	\begin{equation}\label{eq:logderivative}
		\left\|
		\frac{r_1'}{r_1}
		-
		\frac{r_2'}{r_2}
		\right\|_\tau
		\le
		L_{\log}(\delta)
		\|r_1-r_2\|_\sigma,
	\end{equation}
	where
	\begin{equation}\label{eq:Llog}
		L_{\log}(\delta)
		:=
		M_\delta C_{\sigma,\tau}^{(1)}
		+
		M_\delta^2K_\delta.
	\end{equation}
\end{lemma}

\begin{proof}
	Set
	\[
	h:=r_1-r_2.
	\]
	Then
	\[
	\frac{r_1'}{r_1}
	-
	\frac{r_2'}{r_2}
	=
	\frac{h'}{r_1}
	+
	r_2'
	\left(
	\frac1{r_1}-\frac1{r_2}
	\right).
	\]
	The first term satisfies
	\[
	\left\|\frac{h'}{r_1}\right\|_\tau
	\le
	M_\delta
	C_{\sigma,\tau}^{(1)}
	\|h\|_\sigma.
	\]
	Moreover,
	\[
	\frac1{r_1}-\frac1{r_2}
	=
	-\frac{r_1-r_2}{r_1r_2},
	\]
	and hence
	\[
	\left\|
	\frac1{r_1}-\frac1{r_2}
	\right\|_\tau
	\le
	M_\delta^2\|h\|_\sigma.
	\]
	Since
	\[
	\|r_2'\|_\tau\le K_\delta,
	\]
	the result follows.
\end{proof}

Combining Lemma~\ref{lem:logderivative} with
Proposition~\ref{prop:composition-strip} gives the required stability
of the coefficient in the linearized Theodorsen operator.

\begin{proposition}\label{prop:psi-r-stability}
	For every
	\[
	r_1,r_2\in B_\delta(r_0),
	\qquad
	u\in\mathcal U,
	\]
	one has
	\begin{equation}\label{eq:psi-r-stability}
		\|\psi_{r_1,u}-\psi_{r_2,u}\|_\rho
		\le
		C_\Phi(\delta)
		\|r_1-r_2\|_\sigma,
	\end{equation}
	where
	\begin{equation}\label{eq:Cphi}
		C_\Phi(\delta)
		:=
		C_{\rm comp}L_{\log}(\delta).
	\end{equation}
\end{proposition}

\begin{proof}
	Apply Proposition~\ref{prop:composition-strip} to
	\[
	g
	=
	\frac{r_1'}{r_1}
	-
	\frac{r_2'}{r_2}
	\]
	and use Lemma~\ref{lem:logderivative}.
\end{proof}

% ------------------------------------------------------------
\subsection{Variation with respect to the boundary correspondence}
% ------------------------------------------------------------

We also need to control the variation of $\psi_{r,u}$ with respect
to $u$. Set
\[
\phi_r:=\frac{r'}r.
\]
For $u_1,u_2\in\mathcal U$, define
\[
u_s:=u_2+s(u_1-u_2),
\qquad
0\le s\le1.
\]
Since $\mathcal U$ is convex,
\[
u_s\in\mathcal U
\]
for every $s\in[0,1]$.

We first obtain a uniform bound on $\phi_r'$. Since
\[
\phi_r'
=
\frac{r''}{r}
-
\left(\frac{r'}r\right)^2,
\]
define
\begin{equation}\label{eq:C2}
	C_{\sigma,\tau}^{(2)}:=\sup_{k\ge1}k^2\left(\frac{\tau}{\sigma}\right)^k
	\le\frac{4}{e^2\log^2(\sigma/\tau)}
\end{equation}
Then
\[
\|h''\|_\tau
\le
C_{\sigma,\tau}^{(2)}
\|h\|_\sigma.
\]
Consequently, if
\begin{equation}\label{eq:K2delta}
	K_{2,\delta}
	:=
	\|r_0''\|_\tau
	+
	C_{\sigma,\tau}^{(2)}\delta,
\end{equation}
then
\[
\|r''\|_\tau
\le
K_{2,\delta}.
\]
It follows that
\begin{equation}\label{eq:phi-prime-bound}
	\|\phi_r'\|_\tau
	\le
	M_\delta K_{2,\delta}
	+
	M_\delta^2K_\delta^2.
\end{equation}
Set
\begin{equation}\label{eq:Kphi}
	K_\phi(\delta)
	:=
	M_\delta K_{2,\delta}
	+
	M_\delta^2K_\delta^2.
\end{equation}

\begin{lemma}\label{lem:psi-u-stability}
	For every
	\[
	r\in B_\delta(r_0)
	\]
	and every
	\[
	u_1,u_2\in\mathcal U,
	\]
	one has
	\begin{equation}\label{eq:psi-u-stability}
		\|\psi_{r,u_1}-\psi_{r,u_2}\|_\rho
		\le
		L_u(\delta)
		\|u_1-u_2\|_\rho,
	\end{equation}
	where
	\begin{equation}\label{eq:Lu}
		L_u(\delta)
		:=
		C_{\rm comp}K_\phi(\delta).
	\end{equation}
\end{lemma}

\begin{proof}
	We have
	\[
		\psi_{r,u_1}-\psi_{r,u_2}=
		\phi_r(\,\cdot+u_1\,)-\phi_r(\,\cdot+u_2\,)=
		\int_0^1\phi_r'(\,\cdot+u_s\,)(u_1-u_2)\,ds.
	\]
	Since $u_s\in\mathcal U$, Proposition~\ref{prop:composition-strip}
	gives
	\[
	\|\phi_r'(\,\cdot+u_s\,)\|_\rho
	\le
	C_{\rm comp}\|\phi_r'\|_\tau
	\le
	C_{\rm comp}K_\phi(\delta).
	\]
	The Banach algebra property of $\mathcal A_\rho$ therefore yields
	\[
		\|\psi_{r,u_1}-\psi_{r,u_2}\|_\rho\le
		\int_0^1\|\phi_r'(\,\cdot+u_s\,)\|_\rho\|u_1-u_2\|_\rho\,ds\le
		C_{\rm comp}K_\phi(\delta)\|u_1-u_2\|_\rho.
	\]
\end{proof}

\subsection{Variation of the nonlinear Theodorsen operator}
\label{sec:variation-F}

For a positive boundary function $r$, recall that
\[
\mathcal F_r(u)
=
u-\mathcal H\bigl[\log r(\,\cdot+u(\cdot)\,)\bigr].
\]
We now estimate the variation of $\mathcal F_r$ with respect to the
boundary function $r$.

As in Section~\ref{ssec:logarithmic-derivative}, assume that
\[
r_0^{-1}\in\mathcal A_\tau,
\qquad
M_0:=\|r_0^{-1}\|_\tau,
\]
and let
\[
r\in B_\delta(r_0)\subset\mathcal A_\sigma,
\qquad
M_0\delta<1.
\]
Since $\tau<\sigma$, we have
\[
\|r-r_0\|_\tau
\le
\|r-r_0\|_\sigma
\le
\delta,
\]
and therefore
\[
\|r_0^{-1}(r-r_0)\|_\tau
\le
M_0\delta
<
1.
\]
Consequently,
\[
r
=
r_0\bigl(1+r_0^{-1}(r-r_0)\bigr)
\]
is invertible in $\mathcal A_\tau$. Since $r_0$ is positive on the real axis and invertible in
$\mathcal A_\tau$, its winding number is zero; hence it admits a
$2\pi$-periodic logarithm in $\mathcal A_\tau$, which we fix once and
for all. Moreover, the logarithm of the
second factor is well defined by its convergent Banach-algebra series
\[
\log\bigl(1+r_0^{-1}(r-r_0)\bigr)
=
\sum_{n\ge1}
\frac{(-1)^{n+1}}{n}
\bigl[r_0^{-1}(r-r_0)\bigr]^n.
\]
Thus, we define
\[
\log r-\log r_0
:=
\log\bigl(1+r_0^{-1}(r-r_0)\bigr)
\in\mathcal A_\tau.
\]
On the real axis this agrees with the usual real logarithm, since
$r_0$ and, after possibly decreasing $\delta$, every
$r\in B_\delta(r_0)$ are positive.

For the quantitative estimate it is convenient to use the line
segment
\[
r_s:=r_0+s(r-r_0),
\qquad
0\le s\le1.
\]
For every $s\in[0,1]$,
\[
\|r_s-r_0\|_\sigma
\le
\delta,
\]
and hence
\[
\|r_s^{-1}\|_\tau
\le
M_\delta,
\qquad
M_\delta
:=
\frac{M_0}{1-M_0\delta}.
\]
The map
\[
s\longmapsto \log r_s
\]
is differentiable as an $\mathcal A_\tau$-valued map and
\[
\frac{d}{ds}\log r_s
=
r_s^{-1}(r-r_0).
\]
It follows that
\[
\log r-\log r_0
=
\int_0^1
r_s^{-1}(r-r_0)\,ds.
\]
Therefore, by the Banach algebra property,
\begin{equation}\label{eq:log-r-variation}
	\|\log r-\log r_0\|_\tau
	\le
	M_\delta\|r-r_0\|_\sigma.
\end{equation}

We can now estimate the nonlinear Theodorsen operator.

\begin{proposition}\label{prop:F-boundary-variation}
	Let $r\in B_\delta(r_0)$ and let
	\[
	u\in U=B_R(\bar u)\subset\mathcal A_\rho.
	\]
	Assume that the hypotheses of
	Proposition~\ref{prop:composition-strip} hold. Then
	\begin{equation}\label{eq:F-boundary-variation}
		\|\mathcal F_r(u)-\mathcal F_{r_0}(u)\|_\rho
		\le
		C_F(\delta)\|r-r_0\|_\sigma,
	\end{equation}
	uniformly for $u\in U$, where
	\begin{equation}\label{eq:CF}
		C_F(\delta)
		:=
		C_{\rm comp}M_\delta.
	\end{equation}
\end{proposition}

\begin{proof}
	By definition,
	\[
	\begin{aligned}
		\mathcal F_r(u)-\mathcal F_{r_0}(u)
		&=
		-\mathcal H
		\left[
		\log r(\,\cdot+u(\cdot)\,)
		-
		\log r_0(\,\cdot+u(\cdot)\,)
		\right]
		\\
		&=
		-\mathcal H
		\left[
		(\log r-\log r_0)(\,\cdot+u(\cdot)\,)
		\right].
	\end{aligned}
	\]
	Since
	\[
	\log r-\log r_0\in\mathcal A_\tau,
	\]
	Proposition~\ref{prop:composition-strip} gives
	\[
	\|(\log r-\log r_0)(\,\cdot+u(\cdot)\,)\|_\rho
	\le
	C_{\rm comp}
	\|\log r-\log r_0\|_\tau.
	\]
	Using
	\[
	\|\mathcal H\|_{\mathcal A_\rho\to\mathcal A_\rho}\le1
	\]
	and \eqref{eq:log-r-variation}, we obtain
	\[
	\begin{aligned}
		\|\mathcal F_r(u)-\mathcal F_{r_0}(u)\|_\rho
		&\le
		C_{\rm comp}
		\|\log r-\log r_0\|_\tau
		\\
		&\le
		C_{\rm comp}M_\delta
		\|r-r_0\|_\sigma.
	\end{aligned}
	\]
	This proves \eqref{eq:F-boundary-variation}.
\end{proof}
% ------------------------------------------------------------
\subsection{Uniform a posteriori stability}
% ------------------------------------------------------------

We can now formulate the quantitative stability theorem. Importantly,
the whole fixed-point argument is performed in the single Banach
space $\mathcal A_\rho$; the stronger weight $\widehat\rho$ is used
only to obtain the strip estimate \eqref{eq:dbar} for the known
approximate solution $\bar u$.

Let
\[
L:\mathcal A_\rho\longrightarrow\mathcal A_\rho
\]
be an invertible bounded operator used in the Newton-like map
\[
\mathcal N_r(u)
:=
u-L\mathcal F_r(u).
\]
In the applications, $L$ is obtained from the Wiener--Hopf
construction of Section~\ref{sec:WH}.

Assume that
\begin{equation}\label{eq:Y0}
	Y_0
	\ge
	\|\mathcal N_{r_0}(\bar u)-\bar u\|_\rho
\end{equation}
and
\begin{equation}\label{eq:Z0}
	Z_0
	\ge
	\sup_{u\in\mathcal U}
	\|I-LD\mathcal F_{r_0}(u)\|
	_{\mathcal A_\rho\to\mathcal A_\rho}.
\end{equation}

\begin{theorem}[Uniform stability of the Theodorsen equation]
	\label{thm:uniform-theodorsen}
	Let
	\[
	1<\rho<\widehat\rho<\tau<\sigma
	\]
	and let
	\[
	\mathcal U=B_R(\bar u)\subset\mathcal A_\rho,
	\qquad
	\bar u\in\mathcal A_{\widehat\rho}.
	\]
	Assume that:
	\begin{enumerate}
		\item $r_0\in\mathcal A_\sigma$ and $r_0^{-1}\in\mathcal A_\tau;$
		\item $\bar d:=\sup_{z\in S_{\widehat\rho}}|\operatorname{Im}\bar u(z)|$
satisfies $\rho e^{\bar d+R}<\tau;$
		\item $M_0\delta<1;$
		\item the reference estimates \eqref{eq:Y0} and \eqref{eq:Z0} hold.
		\item $\kappa_\delta:=Z_0+\|L\|C_\Phi(\delta)\delta<1$
		\item $Y_0+\|L\|C_F(\delta)\delta+\kappa_\delta R\le R.$
	\end{enumerate}
Then, for every
	\[
	r\in\mathcal A_\sigma,
	\qquad
	\|r-r_0\|_\sigma\le\delta,
	\]
	the Theodorsen equation
	\[
	\mathcal F_r(u)=0
	\]
	has a unique solution in
	\[
	B_R(\bar u)\subset\mathcal A_\rho.
	\]
\end{theorem}

\begin{proof}
	For $r\in B_\delta(r_0)$,
	\[
		\|\mathcal N_r(\bar u)-\bar u\|_\rho\le
		\|\mathcal N_{r_0}(\bar u)-\bar u\|_\rho
+\|L\|\|\mathcal F_r(\bar u)-\mathcal F_{r_0}(\bar u)\|_\rho
\le Y_0+\|L\|C_F(\delta)\delta.
	\]
	
	Furthermore,
	\[
	D\mathcal F_r(u)
	-
	D\mathcal F_{r_0}(u)
	=
	-\mathcal H
	M_{\psi_{r,u}-\psi_{r_0,u}}.
	\]
	Hence Proposition~\ref{prop:psi-r-stability} gives, uniformly for
	$u\in\mathcal U$,
	\[
	\begin{aligned}
		\|I-LD\mathcal F_r(u)\|
		&\le
		\|I-LD\mathcal F_{r_0}(u)\|
+\|L\|\|D\mathcal F_r(u)-D\mathcal F_{r_0}(u)\|
		\\
		&\le
		Z_0
		+
		\|L\|C_\Phi(\delta)\delta=\kappa_\delta.
	\end{aligned}
	\]
	Thus $\mathcal N_r$ is a contraction on $\mathcal U$, uniformly
	with respect to $r\in B_\delta(r_0)$.
	
	For $u\in\mathcal U$, the mean value theorem gives
	\[
	\begin{aligned}
		\|\mathcal N_r(u)-\bar u\|_\rho
		&\le
		\|\mathcal N_r(\bar u)-\bar u\|_\rho
		+
		\|\mathcal N_r(u)-\mathcal N_r(\bar u)\|_\rho
		\\
		&\le
		Y_0
		+
		\|L\|C_F(\delta)\delta
		+
		\kappa_\delta\|u-\bar u\|_\rho
		\\
		&\le
		Y_0
		+
		\|L\|C_F(\delta)\delta
		+
		\kappa_\delta R
		\\
		&\le
		R.
	\end{aligned}
	\]
	Hence
	\[
	\mathcal N_r(\mathcal U)\subset\mathcal U.
	\]
	Since $\mathcal U$ is a closed subset of the Banach space
	$\mathcal A_\rho$, the contraction mapping theorem gives a unique
	fixed point $u_r\in\mathcal U$. Since $L$ is invertible,
	\[
	\mathcal N_r(u_r)=u_r
	\quad\Longleftrightarrow\quad
	\mathcal F_r(u_r)=0.
	\]
\end{proof}

% ------------------------------------------------------------
\subsection{Dependence of the solution on the boundary}
% ------------------------------------------------------------

The uniform contraction argument also yields a Lipschitz estimate for
the solution in $\mathcal A_\rho$.

\begin{corollary}\label{cor:u-lipschitz-r}
	Let $r_1,r_2\in B_\delta(r_0)$ and let $u_{r_1}$ and $u_{r_2}$ be
	the corresponding solutions. Then
	\begin{equation}\label{eq:u-lipschitz-r}
		\|u_{r_1}-u_{r_2}\|_\rho \le C_u\|r_1-r_2\|_\sigma,
	\end{equation}
	where
	\begin{equation}\label{eq:Cu}
		C_u:=C_u(\delta)=
		\frac{\Lambda C_{\rm comp}M_\delta}{1-\kappa_\delta},
	\end{equation}
	where $\Lambda$ is an upper bound for $\|L\|$.
\end{corollary}

\begin{proof}
	Since
	\[
	u_{r_i}
	=
	\mathcal N_{r_i}(u_{r_i}),
	\qquad i=1,2,
	\]
	we have
	\[
	\begin{aligned}
		\|u_{r_1}-u_{r_2}\|_\rho
		&\le\|\mathcal N_{r_1}(u_{r_1})-\mathcal N_{r_1}(u_{r_2})\|_\rho+
		\|\mathcal N_{r_1}(u_{r_2})-\mathcal N_{r_2}(u_{r_2})\|_\rho\\
		&\le\kappa_\delta\|u_{r_1}-u_{r_2}\|_\rho+\|L\|C_{\rm comp}M_\delta\|r_1-r_2\|_\sigma.
	\end{aligned}
	\]
	Rearranging proves the result.
\end{proof}

% ------------------------------------------------------------
\subsection{Orientation of the boundary correspondence}
% ------------------------------------------------------------

Finally, to recover the Riemann map from the solution of the
Theodorsen equation, we verify uniformly that
\[
\theta_r(t)=t+u_r(t)
\]
is orientation preserving.

Assume that a rigorous estimate
\begin{equation}\label{eq:orientation}
	\inf_{u\in B_R(\bar u)}\left(\min_t u'(t)\right)\ge\eta_\theta>-1
\end{equation}
is available. Then every validated solution satisfies
\[
\theta_r'(t)
=
1+u_r'(t)
\ge
1-\eta_\theta
>
0.
\]
Consequently, $\theta_r$ is an orientation-preserving diffeomorphism
of the circle.
Note that, if
\[
u=\bar u+h,\qquad \|h\|_\rho\le R,
\]
then
\[
\|h'\|_\infty
\le
C_{\rho}^{\rm der}R,
\qquad
C_{\rho}^{\rm der}:=
\sup_{k\ge1}\frac{k}{\rho^k}
\le \frac{1}{e\log\rho},
\]
therefore
\[
	\inf_{u\in B_R(\bar u)}\left(\min_t u'(t)\right)
	\ge\min_t \bar u'(t)-C_{\rho}^{\rm der}R.
\]
If, in addition,
\[
\delta
<
\min_{t\in\mathbb R}r_0(t),
\]
then every
\[
r\in B_\delta(r_0)
\]
is positive on the real axis and therefore determines a star-shaped
Jordan domain $\Omega_r$. Proposition~\ref{prop:Theodorsen} then
shows that
\[
f_r(e^{it})
=
r(t+u_r(t))e^{i(t+u_r(t))}
\]
is the boundary trace of the normalized Riemann map
\[
f_r:\mathbb D\longrightarrow\Omega_r.
\]

% ============================================================
\section{Rigorous verification at the reference domains}
\label{sec:verification}
% ============================================================

The results of Sections~\ref{sec:aposteriori}--
\ref{sec:stability} reduce the stability problem to the verification
of a finite collection of explicit inequalities at a reference
boundary function $r_0$. In this section we describe how the
corresponding quantities are evaluated rigorously.

The reference solution of the Theodorsen equation is represented by a finite Fourier
approximation $\bar u$. In particular,
\[
\bar u\in\mathcal A_{\widehat\rho}
\subset\mathcal A_\rho,
\]
and the strip quantity
\[
\bar d
=
\sup_{z\in S_{\widehat\rho}}
|\operatorname{Im}\bar u(z)|
\]
can be evaluated rigorously from its finite Fourier expansion. At the reference configuration we write
\[
a_0:=-\psi_{r_0,\bar u}=-\frac{r_0'}{r_0}\bigl(\,\cdot+\bar u(\cdot)\bigr)
\]
and
\[
\Phi_0
:=
D\mathcal F_{r_0}(\bar u)
=
I+\mathcal H M_{a_0}.
\]
The Wiener--Hopf construction of Section~5 provides the formula from which an approximate inverse is constructed. In the computer-assisted proof, the Laurent functions entering this formula are replaced by finite Laurent approximations, yielding an infinite-dimensional banded operator \(L_K\). The rigorous invertibility argument below does not require separate validation of the hypotheses of Theorem~5.2.

We need rigorous estimates of the form
\begin{equation}\label{eq:reference-Y}
	Y_0
	\ge
	\|\mathcal N_{r_0}(\bar u)-\bar u\|_\rho,
\end{equation}
\begin{equation}\label{eq:reference-Z}
	Z_0
	\ge
	\sup_{u\in B_R(\bar u)}
	\|I-L_KD\mathcal F_{r_0}(u)\|,
\end{equation}
and
\begin{equation}\label{eq:reference-L}
	\Lambda
	\ge
	\|L_K\|.
\end{equation}
These quantities are then inserted into
Theorem~\ref{thm:uniform-theodorsen}, together with the analytic
constants of Section~\ref{sec:stability}, to obtain an explicit
radius $\delta_r$ in $\mathcal A_\sigma$.
To prove that the operators $L_K$ and $\Phi_0$ are invertible, we do not use  Theorem~\ref{thm:PhiInverse}. Instead, we prove the sufficient condition 
\[
\|I-L_K\Phi_0\|_{\mathcal B_\rho\to\mathcal B_\rho}
\le
\varepsilon_{\rm mat}^{L}+\Lambda_{\mathcal B}\eta_a,
\]
and
\[
\|I-\Phi_0L_K\|_{\mathcal B_\rho\to\mathcal B_\rho}
\le
\varepsilon_{\rm mat}^{R}+\Lambda_{\mathcal B}\eta_a.
\]
If both bounds are strictly smaller than one, then
$L_K\Phi_0$ and $\Phi_0L_K$ are invertible by the Neumann series.
Hence $L_K$ has both a right inverse and a left inverse, and is
therefore invertible; the same is true of $\Phi_0$.

\subsection{Finite approximation of the inverse}

Fix a truncation order $K$. Every Laurent function appearing in the
representation \eqref{eq:explicitL} is replaced by a finite Laurent
polynomial. We denote these approximations by
\[
a_K,\quad
h_K,\quad
W_K,\quad
\alpha_K,\quad
\beta_K,
\]
and similarly for the remaining quantities.

Define
\[
\Phi_K
:=
I+\mathcal H M_{a_K}.
\]
Let
\[
\eta_a
\ge
\|a_0-a_K\|_\rho.
\]
Then
\begin{equation}\label{eq:PhiError}
	\|\Phi_0-\Phi_K\|
	\le
	\eta_a.
\end{equation}

Set
\[
B_K
=
M_{\alpha_K}P_-M_{\beta_K},
\qquad
p_K=P_-W_K,
\]
and
\[
\mathcal Q_K
=
P_--P_-M_{W_K}P_+,
\qquad
C_K
=
1-B_Kp_K.
\]
Define moreover
\[
D_K:=P_0(h_KC_K)
\]
and
\[
\mathcal S_K
:=
P_0
-
P_0M_{h_K}P_+
-
P_0M_{h_K}B_K\mathcal Q_K.
\]
The finite approximation of the inverse is
\begin{equation}\label{eq:LK}
	L_K
	:=
	M_{h_K}
	\left(
	P_+
	+
	B_K\mathcal Q_K
	+
	C_KD_K^{-1}\mathcal S_K
	\right).
\end{equation}

\subsection{The banded infinite-dimensional residual}

Consider
\[
E_K:=I-L_K\Phi_K\qquad\text{and}\qquad \widetilde{E}_K:=I-\Phi_K L_K.
\]
Although $E_K$ and $\widetilde{E}_K$ act on the infinite-dimensional space
$\mathcal B_\rho$, they are banded because all multipliers entering
$L_K$ and $\Phi_K$ are finite Laurent polynomials.

For a finite Laurent multiplier $M_f$, let $q(M_f)$ denote its
bandwidth. Then
\[
q(M_f)=\deg f,
\]
while
\[
q(P_+)=q(P_-)=q(P_0)=q(\mathcal H)=0.
\]
Moreover,
\[
q(AB)\le q(A)+q(B),
\qquad
q(A+B)\le\max\{q(A),q(B)\}.
\]

\begin{lemma}\label{lem:bandwidth}
	The operators $E_K$ and $\widetilde{E}_K$ are banded. If all Laurent approximations entering \eqref{eq:LK} have degree at most $K$, then
	\begin{equation}\label{eq:bandwidth}
		q(E_K)\le9K\qquad\text{and}\qquad q(\widetilde{E}_K)\le9K.
	\end{equation}
	In particular,
	\[
	(E_K)_{mn}=(\widetilde{E}_K)_{mn}=0
	\qquad
	\text{whenever }
	|m-n|>9K.
	\]
\end{lemma}

\begin{proof}
	From \eqref{eq:LK},
	\[
	q(L_K)
	\le
	2d_h+2(d_\alpha+d_\beta+d_W),
	\]
	where the symbols on the right denote the degrees of the
	corresponding finite Laurent polynomials. Furthermore,
	\[
	q(\Phi_K)\le d_a.
	\]
	If all these degrees are bounded by $K$, then
	\[
	q(E_K)
	\le
	q(L_K)+q(\Phi_K)
	\le
	9K,
	\]
	and the same holds for $\widetilde{E}_K$.
\end{proof}
For each column $n\in\mathbb Z$, define
\[
S_n
:=
\sum_{m\in\mathbb Z}
|(E_K)_{mn}|
\rho^{|m|-|n|}.
\]
If $q=q(E_K)$, then
\[
S_n
=
\sum_{j=-q}^{q}
|(E_K)_{n+j,n}|
\rho^{|n+j|-|n|}.
\]

For $n>q$, all indices $n+j$ are positive. On such columns the
operators $P_\pm$, $P_0$, and $\mathcal H$ act in a
translation-invariant way. Hence
\[
e_j^+
:=
(E_K)_{n+j,n},
\qquad
-q\le j\le q,
\]
is independent of $n>q$. Similarly,
\[
e_j^-
:=
(E_K)_{n+j,n}
\]
is independent of $n<-q$. It follows that
\[
S_n
=
\sum_{j=-q}^{q}|e_j^+|\rho^j,
\qquad
n>q,
\]
and
\[
S_n
=
\sum_{j=-q}^{q}|e_j^-|\rho^{-j},
\qquad
n<-q.
\]

Consequently,
\begin{equation}\label{eq:matrixResidual}
	\|E_K\|_\rho
	\le
	\varepsilon^L_{\mathrm{mat}},
\end{equation}
where
\[
\varepsilon^L_{\mathrm{mat}}
:=
\max
\left\{
\max_{|n|\le q}
\sum_{j=-q}^{q}
|(E_K)_{n+j,n}|
\rho^{|n+j|-|n|},
\;
\sum_{j=-q}^{q}|e_j^+|\rho^j,
\;
\sum_{j=-q}^{q}|e_j^-|\rho^{-j}
\right\}.
\]
Since $E_K$ preserves the symmetric Laurent subspace corresponding to
real-valued Fourier functions, Lemma~\ref{lem:operator-A-B} gives
\begin{equation}\label{eq:matrixResidual-A}
	\|E_K\|_{\mathcal A_\rho\to\mathcal A_\rho}
	\le
	\sqrt2\,\varepsilon^L_{\rm mat}.
\end{equation}
A similar computation holds for $\widetilde{E}_K$, giving
\begin{equation}\label{eq:matrixResidual-B}
	\|\widetilde{E}_K\|_{\mathcal A_\rho\to\mathcal A_\rho}
	\le
	\sqrt2\,\varepsilon^R_{\rm mat}.
\end{equation}
Set
\[
\varepsilon_{\rm mat}
:=
\max\{\varepsilon^L_{\rm mat},
\varepsilon_{\rm mat}^{R}\},
\]

\subsection{Residual with respect to the exact linearisation}

The previous estimate concerns $\Phi_K$. We now compare it with the
exact reference linearisation $\Phi_0$.

Set
\[
E_0
:=
I-L_K\Phi_0.
\]
Then
\[
E_0-E_K
=
L_K(\Phi_K-\Phi_0),
\]
so that
\[
\|E_0-E_K\|
\le
\|L_K\|\eta_a.
\]
Combining this with \eqref{eq:matrixResidual} yields
\begin{equation}\label{eq:reference-linear-residual}
	\|I-L_K\Phi_0\|\le\varepsilon_{\mathrm{mat}}+\|L_K\|\eta_a\quad\text{and}\quad
	\|I-\Phi_0L_K\|\le\varepsilon_{\mathrm{mat}}+\|L_K\|\eta_a.
\end{equation}

This gives the reference inverse defect at $\bar u$. To control the
whole ball $B_R(\bar u)$ we use
\[
D\mathcal F_{r_0}(u)-D\mathcal F_{r_0}(\bar u)
=
-\mathcal H
M_{\psi_{r_0,u}-\psi_{r_0,\bar u}}.
\]
If
\begin{equation}\label{eq:reference-u-lipschitz}
	\|\psi_{r_0,u}-\psi_{r_0,\bar u}\|_\rho
	\le
	L_u\|u-\bar u\|_\rho
\end{equation}
for all $u\in B_R(\bar u)$, then
\[
\begin{aligned}
	\|I-L_KD\mathcal F_{r_0}(u)\|
	&\le
	\|I-L_K\Phi_0\|+\|L_K\|\|D\mathcal F_{r_0}(u)-D\mathcal F_{r_0}(\bar u)\| \\
	&\le\varepsilon_{\mathrm{mat}}+\|L_K\|\eta_a+\|L_K\|L_uR.
\end{aligned}
\]
Thus one may take
\begin{equation}\label{eq:Z0-computable}
Z_0
=
\sqrt2\left(
\varepsilon_{\rm mat}
+\Lambda_{\mathcal B}\eta_a
+\Lambda_{\mathcal B}L_uR
\right),
\end{equation}

The constant $L_u$ is obtained from
Lemma~\ref{lem:psi-u-stability} applied to
\[
\frac{r_0'}{r_0}.
\]
Since \(\|f\|_{\mathcal B_\rho}\le \|f\|_{\mathcal A_\rho}\) on the symmetric subspace, the estimate \eqref{eq:reference-u-lipschitz} also provides the corresponding bound in \(\mathcal B_\rho\).

\subsection{Reference residual}

The residual entering the Newton-like map is
\[
\mathcal N_{r_0}(\bar u)-\bar u
=
-L_K\mathcal F_{r_0}(\bar u).
\]
Hence
\begin{equation}\label{eq:Y0-computable}
	Y_0
	:=
	\|L_K\mathcal F_{r_0}(\bar u)\|_\rho
\end{equation}
is evaluated directly by interval arithmetic. A rigorous bound
\begin{equation}\label{eq:lambda2}
\Lambda\ge\|L_K\|
\end{equation}
is obtained with the same procedure used for $E_K$.
The weighted column-sum estimate gives a rigorous bound
\[
\Lambda_{\mathcal B}
\ge
\|L_K\|_{\mathcal B_\rho\to\mathcal B_\rho}.
\]
Since $L_K$ preserves the symmetric Laurent subspace,
Lemma~\ref{lem:operator-A-B} yields
\[
\|L_K\|_{\mathcal A_\rho\to\mathcal A_\rho}
\le
\sqrt2\,\Lambda_{\mathcal B}.
\]
We therefore set
\begin{equation}\label{eq:lambda}
	\Lambda
	:=
	\sqrt2\,\Lambda_{\mathcal B}.
\end{equation}
\subsection{From reference estimates to a neighborhood of domains}

The estimates obtained above are combined with the stability bounds
of Section~\ref{sec:stability}. Let
\[
\delta>0
\]
be such that all $r$ satisfying
\[
\|r-r_0\|_\sigma\le\delta
\]
remain in the admissible class of star-shaped boundary functions.

Let
\[
\Lambda,\qquad
C_F(\delta),
\qquad
C_\Phi(\delta)
\]
be the constants obtained from
\eqref{eq:lambda}, \eqref{eq:CF}, and \eqref{eq:Cphi}. Then
Theorem~\ref{thm:uniform-theodorsen} applies as soon as
\begin{equation}\label{eq:verification-condition-1}
	Z_0+\Lambda C_\Phi(\delta)\delta<1
\end{equation}
and
\begin{equation}\label{eq:verification-condition-2}
	Y_0
	+
	\Lambda C_F(\delta)\delta
	+
	\left(
	Z_0
	+
	\Lambda C_\Phi(\delta)\delta
	\right)R
	\le
	R.
\end{equation}

Thus the computer-assisted part of the conformal-mapping argument is
reduced to the rigorous evaluation of the finite collection of
constants and the verification of
\eqref{eq:verification-condition-1}--\eqref{eq:verification-condition-2}.

\subsection{Validated neighborhoods}

The final numerical values are reported in
Table~\ref{tab:domains}. For each reference boundary function $r_0$,
the table contains a rigorously verified radius
\[
\delta_r>0
\]
such that every
\[
r\in\mathcal A_\sigma,
\qquad
\|r-r_0\|_\sigma<\delta_r,
\]
satisfies the conclusions of
Theorem~\ref{thm:conformal-stability}.

\begin{proposition}\label{prop:validated-neighborhoods}
	For each reference boundary function $r_0$ listed in
	Table~\ref{tab:domains}, the interval computations verify the hypotheses of Theorem~\ref{thm:uniform-theodorsen} for the explicit radius $\delta_r>0$ given in the Table. Additionally
	$$
	\min_t \bar u'(t)-C_{\rho}^{\rm der}R>-1\qquad\text{and}\qquad
	\delta_r<\min_t r_0(t).
	$$	
	Consequently, for every
	\[
	r\in\mathcal A_\sigma,
	\qquad
	\|r-r_0\|_\sigma<\delta_r,
	\]
	the Theodorsen equation admits a unique solution
	\[
	u_r\in B_R(\bar u),
	\]
	and the corresponding boundary map is the normalized Riemann map onto
	$\Omega_r$.
\end{proposition}

\begin{proof}
	The proof consists of the computations of the bounds
	\eqref{eq:reference-Y},
	\eqref{eq:reference-Z},
	\eqref{eq:reference-L},
	and the verification of
	\eqref{eq:verification-condition-1}--\eqref{eq:verification-condition-2}
	for each reference configuration, see Section \ref{sec:cap} and \cite{program} for more details.
	Finally, Proposition \ref{prop:composition-strip} implies that the boundary trace belongs to \(\mathcal B_\rho\); since its nonpositive Fourier modes vanish, Lemma \ref{lem:conformalradius} yields the asserted holomorphic continuation and weighted Wiener bound.
\end{proof}

% ============================================================
\section{The semilinear elliptic equation}
\label{sec:elliptic}
% ============================================================

We now turn to the semilinear Dirichlet problem
\begin{equation}\label{eq:elliptic2}
	\begin{cases}
		-\Delta v=v^3,&\text{in }\Omega_r,\\
		v=0,&\text{on }\partial\Omega_r.
	\end{cases}
\end{equation}
The purpose of this section is not only to establish the existence of solutions on the
reference domains, but to show that the solutions persist under the
boundary perturbations considered in Theorem~\ref{thm:conformal-stability}.

Let
\[
f_r:\mathbb D\longrightarrow\Omega_r
\]
be the normalized Riemann map and define
\[
w(z):=v(f_r(z)).
\]
Since $f_r$ is conformal,
\[
\Delta w(z)=|f_r'(z)|^2\Delta(v)(f_r(z)).
\]
Consequently, \eqref{eq:elliptic2} is equivalent to
\begin{equation}\label{eq:elliptic-disk}
	\begin{cases}
		-\Delta w=q_r w^3,&\text{in }\mathbb D,\\
		w=0,&\text{on }\partial\mathbb D,
	\end{cases}
\end{equation}
where
\begin{equation}\label{eq:qr}
	q_r:=|f_r'|^2.
\end{equation}

Thus perturbations of the domain enter the transformed equation only
through the coefficient $q_r$. The first objective of this section is
therefore to quantify the dependence
\[
r\longmapsto f_r\longmapsto q_r.
\]

\subsection{Dependence of the conformal map on the boundary}

Let $r_1,r_2$ belong to the neighborhood
\[
B_{\delta}(r_0)\subset\mathcal A_\sigma
\]
provided by Theorem~\ref{thm:conformal-stability}, where we set $\delta=\delta_r$, and let
$u_1,u_2$ denote the corresponding solutions of the Theodorsen
equation. By Corollary~\ref{cor:u-lipschitz-r},
\begin{equation}\label{eq:u-r-lipschitz-elliptic}
	\|u_1-u_2\|_\rho
	\le
	C_u\|r_1-r_2\|_\sigma,
\end{equation}
where $C_u$ is defined in \eqref{eq:Cu}.
Choose
\[
1<\rho_f<\rho,
\]
set
\[
R_\delta:=\|r_0\|_\sigma+\delta,
\]
and recall \eqref{eq:Kdelta}.
By Theorem~\ref{thm:uniform-theodorsen},
\[
u_r=\bar u+h_r,
\qquad
\|h_r\|_\rho\le R.
\]
Set
\[
d:=\bar d+R.
\]
\begin{lemma}\label{lem:Cf}
	Let $r_1,r_2\in B_\delta(r_0)$, and let $u_1,u_2$ be the
	corresponding solutions of the Theodorsen equation. Assume that
	\[
	\|u_1-u_2\|_\rho
	\le
	C_u\|r_1-r_2\|_\sigma .
	\]
	Then, for every $1<\rho_f<\rho$,
	\begin{equation}\label{eq:f-lipschitz}
		\|f_{r_1}-f_{r_2}\|_{\rho_f}
		\le
		C_f\|r_1-r_2\|_\sigma ,
	\end{equation}
	where
	\begin{equation}\label{eq:Cf}
		C_f
		:=
		\frac{\rho_f}{\rho-\rho_f}\,
		\rho e^d
		\left[
		1+(K_\delta+R_\delta)C_u
		\right],
		\qquad
		d:=\bar d+R .
	\end{equation}
\end{lemma}

\begin{proof}
	Set
	\[
	\Delta_r:=\|r_1-r_2\|_\sigma
	\]
	and, for $z\in S_\rho$, define
	\[
	F_j(z)
	:=
	r_j\bigl(z+u_j(z)\bigr)
	e^{i(z+u_j(z))},
	\qquad j=1,2.
	\]
	On the real axis,
	\[
	F_j(t)=f_{r_j}(e^{it}).
	\]
	Moreover, $F_j$ is holomorphic and $2\pi$-periodic in $S_\rho$.
	Since its boundary values are those of the Riemann map, its
	nonpositive Fourier modes vanish. Hence, by analytic continuation,
	\begin{equation}\label{eq:F-f-identification}
		F_j(z)=f_{r_j}(e^{iz}),
		\qquad z\in S_\rho.
	\end{equation}
	
	We decompose
	\[
	\begin{aligned}
		F_1(z)-F_2(z)
		={}&
		\bigl[
		r_1(z+u_1(z))-r_2(z+u_1(z))
		\bigr]e^{i(z+u_1(z))}
		\\
		&+
		\bigl[
		r_2(z+u_1(z))-r_2(z+u_2(z))
		\bigr]e^{i(z+u_1(z))}
		\\
		&+
		r_2(z+u_2(z))
		\bigl[
		e^{i(z+u_1(z))}-e^{i(z+u_2(z))}
		\bigr].
	\end{aligned}
	\]
	
	For $z\in S_\rho$, the strip inclusion used in
	Proposition~\ref{prop:composition-strip} gives
	\[
	z+u_j(z)\in S_\tau.
	\]
	Indeed,
	\[
	|\Im(z+u_j(z))|
	<
	\log\rho+\bar d+R
	<
	\log\tau.
	\]
	Therefore
	\[
	|r_1(z+u_1(z))-r_2(z+u_1(z))|
	\le
	\Delta_r.
	\]
	Furthermore,
	\[
	|e^{i(z+u_j(z))}|
	=
	e^{-\Im(z+u_j(z))}
	\le
	\rho e^d.
	\]
	Thus the first term is bounded by
	\[
	\rho e^d\Delta_r.
	\]
	For the second term, we have
	\[r_2(z+u_1(z))-r_2(z+u_2(z))=
		\int_0^1 r_2'\bigl(z+u_2(z)+s(u_1(z)-u_2(z))\bigr)
		\bigl(u_1(z)-u_2(z)\bigr)\,ds .
	\]
	Since the intermediate arguments remain in $S_\tau$ and
	$\|r_2'\|_\tau\le K_\delta$, we obtain
	\[
	|r_2(z+u_1(z))-r_2(z+u_2(z))|
	\le
	K_\delta |u_1(z)-u_2(z)|.
	\]
	Moreover,
	\[
	|u_1(z)-u_2(z)|
	\le
	\|u_1-u_2\|_\rho
	\le
	C_u\Delta_r.
	\]
	Hence the second term is bounded by
	\[
	\rho e^d K_\delta C_u\Delta_r.
	\]
	Similarly,
	\[
	\begin{aligned}
		&e^{i(z+u_1(z))}-e^{i(z+u_2(z))}
		\\
		&\qquad =
		i\int_0^1
		e^{i(z+u_2(z)+s(u_1(z)-u_2(z)))}
		\bigl(u_1(z)-u_2(z)\bigr)\,ds .
	\end{aligned}
	\]
	The intermediate values satisfy the same bound on their imaginary
	parts, and therefore
	\[
	\left|
	e^{i(z+u_1(z))}-e^{i(z+u_2(z))}
	\right|
	\le
	\rho e^d C_u\Delta_r.
	\]
	Since
	\[
	|r_2(z+u_2(z))|
	\le
	\|r_2\|_\tau
	\le
	R_\delta,
	\]
	the third term is bounded by
	\[
	\rho e^d R_\delta C_u\Delta_r.
	\]
	
	Combining the three estimates gives
	\begin{equation}\label{eq:F-strip-bound}
		\sup_{z\in S_\rho}
		|F_1(z)-F_2(z)|
		\le
		M,
	\end{equation}
	where
	\[
	M
	:=
	\rho e^d
	\left[
	1+(K_\delta+R_\delta)C_u
	\right]\Delta_r.
	\]
	
	Write
	\[
	f_{r_1}(\zeta)-f_{r_2}(\zeta)
	=
	\sum_{k\ge1}c_k\zeta^k.
	\]
	Fix $1<\rho'<\rho$.  If $|\zeta|=\rho'$, write
	\[
	\zeta=e^{iz},
	\qquad
	\Im z=-\log\rho'.
	\]
	Then $z\in S_\rho$, and \eqref{eq:F-f-identification} and
	\eqref{eq:F-strip-bound} give
	\[
	|f_{r_1}(\zeta)-f_{r_2}(\zeta)|
	\le M.
	\]
	Cauchy's estimate therefore yields
	\[
	|c_k|\le M(\rho')^{-k}.
	\]
	Letting $\rho'\uparrow\rho$, we obtain
	\[
	|c_k|\le M\rho^{-k}.
	\]
	Consequently,
	\[
	\begin{aligned}
		\|f_{r_1}-f_{r_2}\|_{\rho_f}
		&=
		\sum_{k\ge1}|c_k|\rho_f^k
		\\
		&\le
		M\sum_{k\ge1}
		\left(\frac{\rho_f}{\rho}\right)^k
		\\
		&=
		M\frac{\rho_f}{\rho-\rho_f}.
	\end{aligned}
	\]
	Substituting the value of $M$ proves
	\eqref{eq:f-lipschitz} with $C_f$ as defined in
	\eqref{eq:Cf}.
\end{proof}

Here and below, a further loss of exponential weight is allowed
whenever differentiation is required.
In particular, taking $r_2=r_0$ gives
\begin{equation}\label{eq:f-r0}
	\|f_r-f_{r_0}\|_{\rho_f}
	\le
	C_f\|r-r_0\|_\sigma.
\end{equation}

By Lemma~\ref{lem:taylor-zernike}, the Taylor-to-Zernike
correspondence identifies $f'_r$ isometrically with an element of the
complexified Zernike algebra $\mathcal C_{\rho_q}^{\mathbb C}$.
We use the same symbol $f'_r$ for this element. In particular, the
Taylor estimate gives
\[
\|f'_{r_1}-f'_{r_2}\|_{\mathcal C_{\rho_q}^{\mathbb C}}
\le
C_D C_f\|r_1-r_2\|_\sigma .
\]
Since complex conjugation is an isometry of
$\mathcal C_{\rho_q}^{\mathbb C}$ and this space is a Banach algebra,
\[
q_r
=
f'_r\overline{f'_r}
=
|f'_r|^2
\in \mathcal C_{\rho_q}.
\]
Here $q_r$ is real-valued. Therefore, when estimating its norm in the
real Zernike algebra, we use
Lemma~\ref{lem:zernike-real-complex}.
The following lemma controls the $q$ variation with the $r$ variation.
\begin{lemma}\label{lem:q-variation}
	Let
	\[
	1<\varrho<\rho_q<\rho_f<\sigma
	\]
	and let $r,r_1,r_2\in B_\delta(r_0)$. Then
	\begin{equation}\label{eq:fprime-variation}
		\|f'_{r_1}-f'_{r_2}\|_{\mathcal C_{\rho_q}^{\mathbb C}}
		\le
		C_D C_f\|r_1-r_2\|_\sigma ,
	\end{equation}
	and
	\begin{equation}\label{eq:q-variation}
		\|q_r-q_{r_0}\|_{\mathcal C_{\rho_q}}
		\le
		C_q(\delta)\,\delta ,
	\end{equation}
	where
	\begin{equation}\label{eq:Cq}
		C_q(\delta)
		:=
		\sqrt2\,C_D C_f
		\left(
		2\|f'_{r_0}\|_{\mathcal C_{\rho_q}^{\mathbb C}}
		+
		C_D C_f\,\delta
		\right).
	\end{equation}
\end{lemma}

\begin{proof}
	By the differentiation estimate of Lemma~\ref{lem:derivative},
	\[
	\|f'_{r_1}-f'_{r_2}\|_{\mathcal C_{\rho_q}^{\mathbb C}}
	\le
	C_D\|f_{r_1}-f_{r_2}\|_{\rho_f},
	\]
	where
	\[
	C_D
	:=
	\frac{1}{\rho_q}
	\sup_{k\ge1}
	k\left(\frac{\rho_q}{\rho_f}\right)^k
	\le
	\frac{1}{e\rho_q\log(\rho_f/\rho_q)}.
	\]
	Using Lemma~\ref{lem:Cf} gives
	\[
	\|f'_{r_1}-f'_{r_2}\|_{\mathcal C_{\rho_q}^{\mathbb C}}
	\le
	C_D C_f\|r_1-r_2\|_\sigma,
	\]
	which proves \eqref{eq:fprime-variation}.
	
	Since
	\[
	q_r=f'_r\overline{f'_r},
	\]
	we have
	\[
	q_r-q_{r_0}
	=
	(f'_r-f'_{r_0})\overline{f'_{r_0}}
	+
	f'_r\,
	\overline{(f'_r-f'_{r_0})}.
	\]
	Hence, by the Banach algebra property and the fact that complex
	conjugation is an isometry,
	\[
	\|q_r-q_{r_0}\|_{\mathcal C_{\rho_q}^{\mathbb C}}
	\le
	\left(
	\|f'_{r_0}\|_{\mathcal C_{\rho_q}^{\mathbb C}}
	+
	\|f'_r\|_{\mathcal C_{\rho_q}^{\mathbb C}}
	\right)
	\|f'_r-f'_{r_0}\|_{\mathcal C_{\rho_q}^{\mathbb C}}.
	\]
	Moreover, \eqref{eq:fprime-variation} gives
	\[
	\|f'_r\|_{\mathcal C_{\rho_q}^{\mathbb C}}
	\le
	\|f'_{r_0}\|_{\mathcal C_{\rho_q}^{\mathbb C}}
	+
	C_D C_f\,\delta.
	\]
	Therefore
	\[
	\|q_r-q_{r_0}\|_{\mathcal C_{\rho_q}^{\mathbb C}}
	\le
	C_D C_f
	\left(
	2\|f'_{r_0}\|_{\mathcal C_{\rho_q}^{\mathbb C}}
	+
	C_D C_f\,\delta
	\right)\delta.
	\]
	Since $q_r-q_{r_0}$ is real-valued,
	Lemma~\ref{lem:zernike-real-complex} yields
	\[
	\|q_r-q_{r_0}\|_{\mathcal C_{\rho_q}}
	\le
	\sqrt2\,
	\|q_r-q_{r_0}\|_{\mathcal C_{\rho_q}^{\mathbb C}}.
	\]
	Combining the last two estimates proves \eqref{eq:q-variation} with
	$C_q(\delta)$ as defined in \eqref{eq:Cq}.
\end{proof}

\subsection{An a posteriori persistence theorem}

We use the fixed point approach and the treatment of the
infinite-dimensional Zernike tails developed in \cite{ArioliKoch1},
but here we formulate the argument in a way that is uniform with respect to
the coefficient $q$.

\begin{proposition}[Stability with respect to the coefficient]
	\label{prop:elliptic-q-stability}
	Let
	\[
	G_q(w)
	:=
	(-\Delta)^{-1}(qw^3),
	\]
	where $(-\Delta)^{-1}$ denotes the inverse of the Dirichlet Laplacian
	on $\mathbb D$ acting on $\mathcal{C}_\varrho$.
	
	Let
	\[
	q_0,\bar w\in\mathcal C_\varrho,
	\]
	and let $M$ be a finite-rank linear operator on $\mathcal C_\varrho$ such that
	\[
	A:=I-M
	\]
is invertible. For $q\in\mathcal C_\varrho$, define
	\begin{equation}\label{eq:Nq}
		\mathcal N_q(h)
		:=
		G_q(\bar w+Ah)-\bar w+Mh.
	\end{equation}
	For $s>0$, let
	\[B_s:=\left\{h\in\mathcal C_\varrho:\|h\|_\varrho\le s\right\},
	\]
	and define
	\begin{equation}\label{eq:Us}
		W_s:=\|\bar w\|_\varrho+\|A\|\,s.
	\end{equation}
	Let moreover
	\[
	C_\Delta:=
	\|(-\Delta)^{-1}\|_{\mathcal C_\varrho\to\mathcal C_\varrho}=\frac{1+\varrho^2}{8},
	\]
(the last equality follows from Lemma \ref{lem:invlapnorm}).
	Assume that
	\begin{equation}\label{eq:Y0ell}
		Y_0:=\|\mathcal N_{q_0}(0)\|_\varrho=\|G_{q_0}(\bar w)-\bar w\|_\varrho
	\end{equation}
	is finite, and that
	\begin{equation}\label{eq:Z0ell}
		Z_0\ge\sup_{h\in B_s}\|D\mathcal N_{q_0}(h)\|.
	\end{equation}
	Let $q\in\mathcal C_\varrho$ satisfy
	\[
	\|q-q_0\|_\varrho\le\varepsilon_q.
	\]
	Set
	\begin{equation}\label{eq:kappaell}
		\kappa:=Z_0+3C_\Delta W_s^2\|A\|\,\varepsilon_q.
	\end{equation}
	If
	\begin{equation}\label{eq:ell-contract}
		\kappa<1
	\end{equation}
	and
	\begin{equation}\label{eq:ell-radius}
		Y_0+C_\Delta W_s^3\varepsilon_q+\kappa s\le s,
	\end{equation}
	then $\mathcal N_q$ is a contraction of $B_s$ into itself. Hence it
	has a unique fixed point
	\[
	h_q\in B_s.
	\]
	The function
	\[
	w_q:=\bar w+Ah_q
	\]
	is a fixed point of $G_q$ and therefore satisfies
	\[
	\begin{cases}
		-\Delta w_q=q\,w_q^3,
		&\text{in }\mathbb D,\\
		w_q=0,
		&\text{on }\partial\mathbb D.
	\end{cases}
	\]
	Moreover,
	\begin{equation}\label{eq:u-error}
		\|w_q-\bar w\|_\varrho
		\le
		\|A\|\,s.
	\end{equation}
\end{proposition}

\begin{proof}
	For every $h\in B_s$,
	\[
	\|\bar w+Ah\|_\varrho
	\le
	\|\bar w\|_\varrho+\|A\|\,\|h\|_\varrho
	\le
	W_s.
	\]
	
	We first estimate the variation of $\mathcal N_q$ with respect to
	$q$. Since
	\[
	G_q(w)-G_{q_0}(w)
	=
	(-\Delta)^{-1}\bigl((q-q_0)w^3\bigr),
	\]
	the Banach algebra property gives
	\[
	\begin{aligned}
		\|
		\mathcal N_q(h)-\mathcal N_{q_0}(h)
		\|_\varrho
		&=
		\|
		G_q(\bar w+Ah)-G_{q_0}(\bar w+Ah)
		\|_\varrho
		\\
		&\le
		C_\Delta
		\|q-q_0\|_\varrho
		\|\bar w+Ah\|_\varrho^3
		\\
		&\le
		C_\Delta W_s^3\varepsilon_q.
	\end{aligned}
	\]
	In particular, at $h=0$,
	\begin{equation}\label{eq:Nq0bound}
		\|\mathcal N_q(0)\|_\varrho
		\le
		Y_0+C_\Delta W_s^3\varepsilon_q.
	\end{equation}
	
	Next,
	\[
	DG_q(w)v
	=
	3(-\Delta)^{-1}(qw^2v),
	\]
	and therefore
	\[
	D\mathcal N_q(h)v
	=
	3(-\Delta)^{-1}
	\left(
	q(\bar w+Ah)^2Av
	\right)
	+
	Mv.
	\]
	Hence
	\[
	\begin{aligned}
		\bigl(
		D\mathcal N_q(h)
		-
		D\mathcal N_{q_0}(h)
		\bigr)v
		&=
		3(-\Delta)^{-1}
		\left(
		(q-q_0)(\bar w+Ah)^2Av
		\right).
	\end{aligned}
	\]
	Thus
	\[
	\begin{aligned}
		\|
		D\mathcal N_q(h)
		-
		D\mathcal N_{q_0}(h)
		\|
		&\le
		3C_\Delta
		\|q-q_0\|_\varrho
		\|\bar w+Ah\|_\varrho^2
		\|A\|
		\\
		&\le
		3C_\Delta W_s^2\|A\|\varepsilon_q.
	\end{aligned}
	\]
	Using \eqref{eq:Z0ell}, we obtain
	\[
	\sup_{h\in B_s}
	\|D\mathcal N_q(h)\|
	\le
	Z_0
	+
	3C_\Delta W_s^2\|A\|\varepsilon_q
	=
	\kappa.
	\]
	By \eqref{eq:ell-contract}, $\mathcal N_q$ is a contraction on
	$B_s$.
	
	Moreover, for $h\in B_s$,
$$		\|\mathcal N_q(h)\|_\varrho
		\le
		\|\mathcal N_q(0)\|_\varrho
		+
		\sup_{\widetilde h\in B_s}
		\|D\mathcal N_q(\widetilde h)\|
		\|h\|_\varrho\le
		Y_0+C_\Delta W_s^3\varepsilon_q+\kappa s.
$$
	Condition \eqref{eq:ell-radius} therefore implies
	\[
	\mathcal N_q(B_s)\subset B_s.
	\]
	The contraction mapping theorem yields a unique fixed point
	\[
	h_q\in B_s.
	\]	
	Then, since
	\[
	h_q
	=
	G_q(\bar w+Ah_q)-\bar w+Mh_q,
	\]
	we have
	\[
	Ah_q
	=
	G_q(\bar w+Ah_q)-\bar w.
	\]
	Thus, with
	\[
	w_q:=\bar w+Ah_q,
	\]
	we obtain
	\[
	w_q=G_q(w_q).
	\]
	By the definition of $G_q$,
	\[
	w_q=(-\Delta)^{-1}(q w_q^3),
	\]
	which is equivalent to
	\[
	-\Delta w_q=q w_q^3
	\]
	with homogeneous Dirichlet boundary conditions.
	
	Finally,
	\[
	\|w_q-\bar w\|_\varrho
	=
	\|Ah_q\|_\varrho
	\le
	\|A\|\,\|h_q\|_\varrho
	\le
	\|A\|\,s.
	\]
\end{proof}

\subsection{Persistence under perturbations of the domain}

Combining Lemma \ref{lem:q-variation} with
Proposition~\ref{prop:elliptic-q-stability} gives the desired
domain-stability result.

\begin{theorem}\label{thm:elliptic-domain-persistence}
	Let $r_0$ be a reference boundary function for which the conformal
	parametrisation has been validated as in
	Theorem~\ref{thm:conformal-stability}.
	Suppose that, with $q_0=q_{r_0}$, the hypotheses of
	Proposition~\ref{prop:elliptic-q-stability} hold for
	$\bar w$, $M$, $A=I-M$, $s$, and $\varepsilon_q$.
	Then there exists an explicitly computable number
	\[
	\delta_{\mathrm{PDE}}>0
	\]
	such that, for every
	\[
	r\in\mathcal A_\sigma,
	\qquad
	\|r-r_0\|_\sigma<\delta_{\mathrm{PDE}},
	\]
	the problem
	\[
	\begin{cases}
		-\Delta v=v^3,&\text{in }\Omega_r,\\
		v=0,&\text{on }\partial\Omega_r
	\end{cases}
	\]
	admits a solution.
\end{theorem}
\begin{proof}
	Choose $0<\delta_{\rm PDE}<\delta_r$ such that $C_q(\delta_{\rm PDE})\,\delta_{\rm PDE}\le\varepsilon_q$. We have
\[
\|q_r-q_{r_0}\|_{\mathcal C_\varrho}
\le
\|q_r-q_{r_0}\|_{\mathcal C_{\rho_q}}
\le
C_q(\delta_{\rm PDE})\delta_{\rm PDE}
\le\varepsilon_q.
\]
Proposition~\ref{prop:elliptic-q-stability} therefore provides \(h_r\in B_s\), such that
\[
w_r=\bar w+Ah_r
\]
solves
\[
-\Delta w_r=q_rw_r^3,\qquad w_r|_{\partial \mathbb D}=0.
\]
Setting
\[
v_r=w_r\circ f_r^{-1}
\]
and using the conformal transformation law gives
\[
-\Delta v_r=v_r^3\quad\text{in }\Omega_r,
\qquad v_r=0\quad\text{on }\partial\Omega_r.
\]
\end{proof}

\subsection{Persistence of qualitative properties}

For \(q_1,q_2\) in the same validated coefficient ball, with corresponding fixed points \(h_1,h_2\), one has
\[
	\|h_1-h_2\|_\varrho=
	\|\mathcal N_{q_1}(h_1)-\mathcal N_{q_2}(h_2)\|_\varrho\le
	\kappa\|h_1-h_2\|_\varrho+C_\Delta W_s^3\|q_1-q_2\|_\varrho.
\]
Therefore
\[
\|h_1-h_2\|_\varrho
\le
\frac{C_\Delta W_s^3}{1-\kappa}
\|q_1-q_2\|_\varrho,
\]
and hence, since \(w_i=\bar w+Ah_i\),
\[
	\|w_1-w_2\|_\varrho
	\le
	\frac{\|A\|C_\Delta W_s^3}{1-\kappa}
	\|q_1-q_2\|_\varrho .
\]
Then, for positive reference solutions, positivity can also be made uniform
with respect to the domain perturbation. The argument uses the
criteria of Appendix~\ref{sec:positivity}.

Suppose that the reference validation yields strict inequalities in
the positivity test: positivity on the selected intermediate circles
and strict spectral gaps
\[
\lambda_1(A_{a,b})
>
\|q_{r_0}w_{r_0}^2\|_{L^\infty(A_{a,b})}.
\]
All quantities involved depend continuously on $q$ and on the
solution in the validated Zernike norm. Hence these strict
inequalities persist after decreasing
$\delta_{\mathrm{PDE}}$ if necessary.

Similarly, suppose that a reference solution is sign-changing and
that the rigorous enclosure provides points $x_+,x_-\in\mathbb D$
and a number $\eta>0$ such that
\[
w_{r_0}(x_+)>\eta,
\qquad
w_{r_0}(x_-)<-\eta.
\]
By the estimate above and the continuous dependence of \(q_r\) on \(r\), after decreasing \(\delta_{\rm PDE}\) if necessary, the same strict inequalities persist throughout the neighborhood.
This observation leads to the following corollary:
\begin{corollary}\label{cor:qualitative-persistence}
	For each reference domain, the radius
	$\delta_{\mathrm{PDE}}$ in
	Theorem~\ref{thm:elliptic-domain-persistence} may be chosen so that
	the qualitative type of the validated solution is preserved
	throughout the neighborhood. In particular:
	\begin{enumerate}
		\item if the reference solution is positive, then $v_r>0$ in
		$\Omega_r$ for every
		$\|r-r_0\|_\sigma<\delta_{\mathrm{PDE}}$;
		
		\item if the reference solution is sign-changing, then $v_r$ is
		sign-changing for every
		$\|r-r_0\|_\sigma<\delta_{\mathrm{PDE}}$.
	\end{enumerate}
\end{corollary}

\subsection{Computer-assisted verification}

The remaining computer-assisted task is now finite and quantitative.
\begin{proposition}\label{prop:validated-elliptic}
	For each reference boundary function $r_0$ listed in
	Table~\ref{tab:domains} and corresponding value of $\delta_\mathrm{PDE}$, the interval computations verify the	hypotheses of Theorem~\ref{thm:elliptic-domain-persistence} and Corollary \ref{cor:qualitative-persistence}.
\end{proposition}

\begin{proof}
	The proof consists of the computations of the bounds \eqref{eq:Y0ell}, \eqref{eq:Z0ell} and the verification of   \eqref{eq:ell-contract} and \eqref{eq:ell-radius}
	for each reference configuration, see Section \ref{sec:cap} and \cite{program} for more details.
\end{proof}

% ============================================================
\section{Open problems}
\label{sec:open}
% ============================================================
Several questions remain open. A first natural problem is to extend the present approach beyond star-shaped domains, replacing the Theodorsen equation by a suitable boundary correspondence formulation for more general analytic Jordan curves. A second direction is rigorous continuation with respect to the shape of the domain. Rather than validating a single neighborhood of a reference boundary, one could attempt to cover continuous paths in the space of boundary functions by overlapping validated neighborhoods, and thereby follow the corresponding conformal maps and elliptic solutions up to a genuine loss of invertibility or a geometric degeneration. This would also provide a framework for the computer-assisted study of bifurcations induced by domain deformation. Finally, for sign-changing solutions it would be interesting to go beyond persistence of sign and obtain quantitative control of their nodal sets, including the preservation or bifurcation of the number and topology of the nodal domains.

% ============================================================
\section{Reference configurations and graphs}
\label{sec:pics}
% ============================================================

The purpose of the examples in this section is to exhibit reference configurations to which
Theorems~\ref{thm:conformal-stability} and \ref{thm:elliptic-stability}
 apply. Each displayed domain is therefore the center of an explicitly validated open neighborhood in $\mathcal A_\sigma$, rather than an isolated example.
The choice of the reference configurations (and their names) is quite arbitrary, but it serves the purpose of showing that the techniques described in the paper apply to a wide class of domains. In particular, all domains are far from a perturbation of the unit disk. The coefficients of the domains {\tt Triblob}, {\tt Tripuff} and {\tt Pillow} have been chosen randomly.  The coefficients of the domain {\tt Cusp} are chosen in such a way that part of the boundary of the domain is close to $0$ and exhibits a behaviour close to a cusp. The names of the other domains are self-explanatory. Of course, the domain {\tt Square} does not serve the purpose of solving the PDE in a square, which could be accomplished in a much more direct way with a two-dimensional Fourier expansion, but rather to show how a domain with corners can be approximated.
For each reference boundary function $r_0$, we display the contour-plot of the coefficient
\[
q_{r_0}=|f_{r_0}'|^2
\]
of the transformed equation on the unit disk, together with the
corresponding validated elliptic solution. When two solutions are
shown for the same reference domain, they correspond respectively to
the positive and sign-changing branches described in
Corollary~\ref{cor:qualitative-persistence}.

\begin{center}
	\includegraphics[height=4cm,width=5cm]{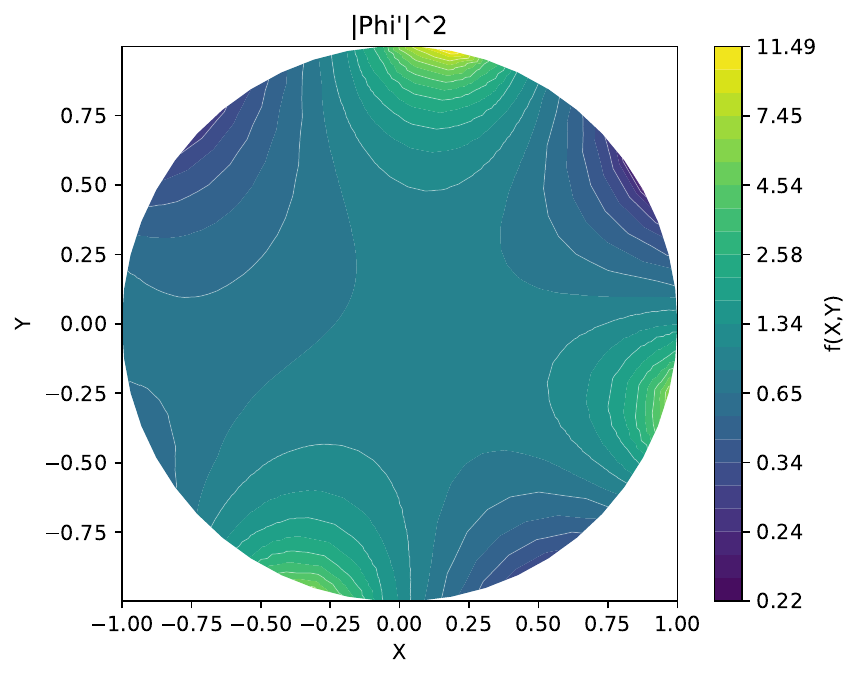}
	\qquad
	\includegraphics[height=4cm,width=4.5cm]{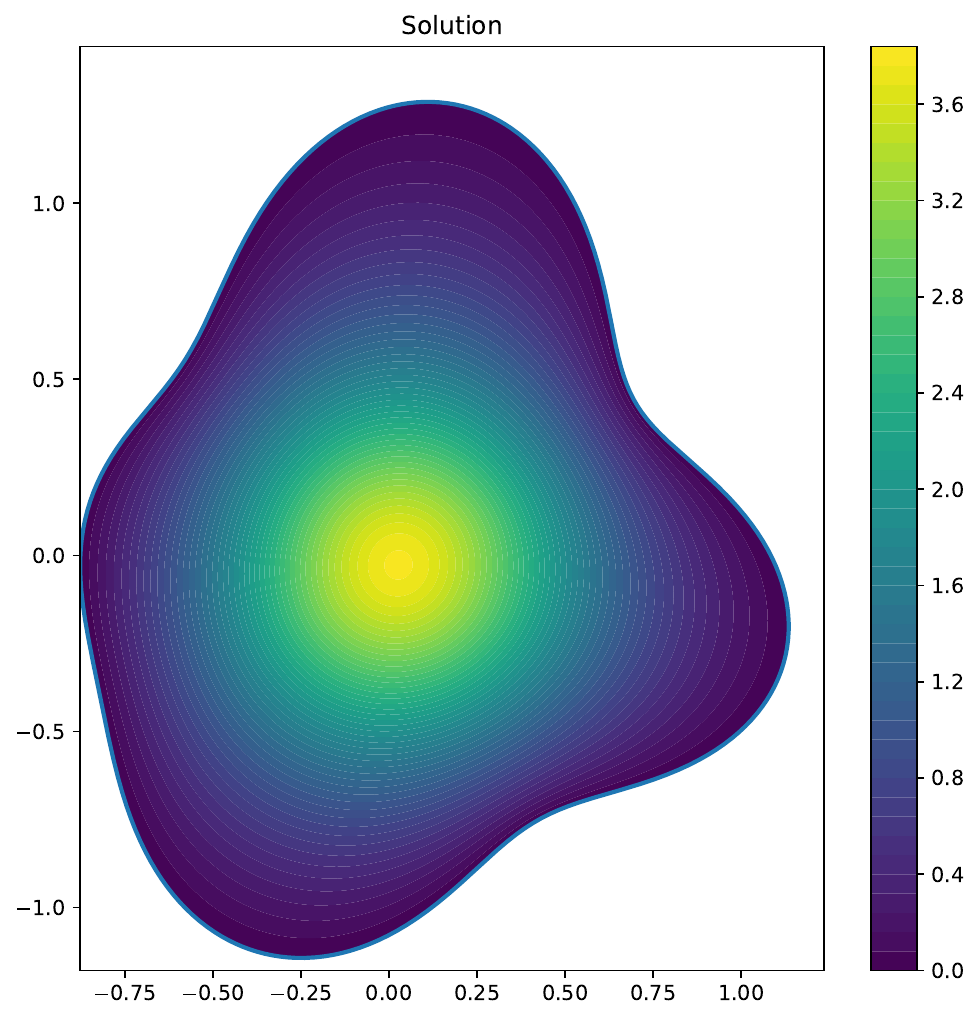}
	\qquad
	\includegraphics[height=4cm,width=4.5cm]{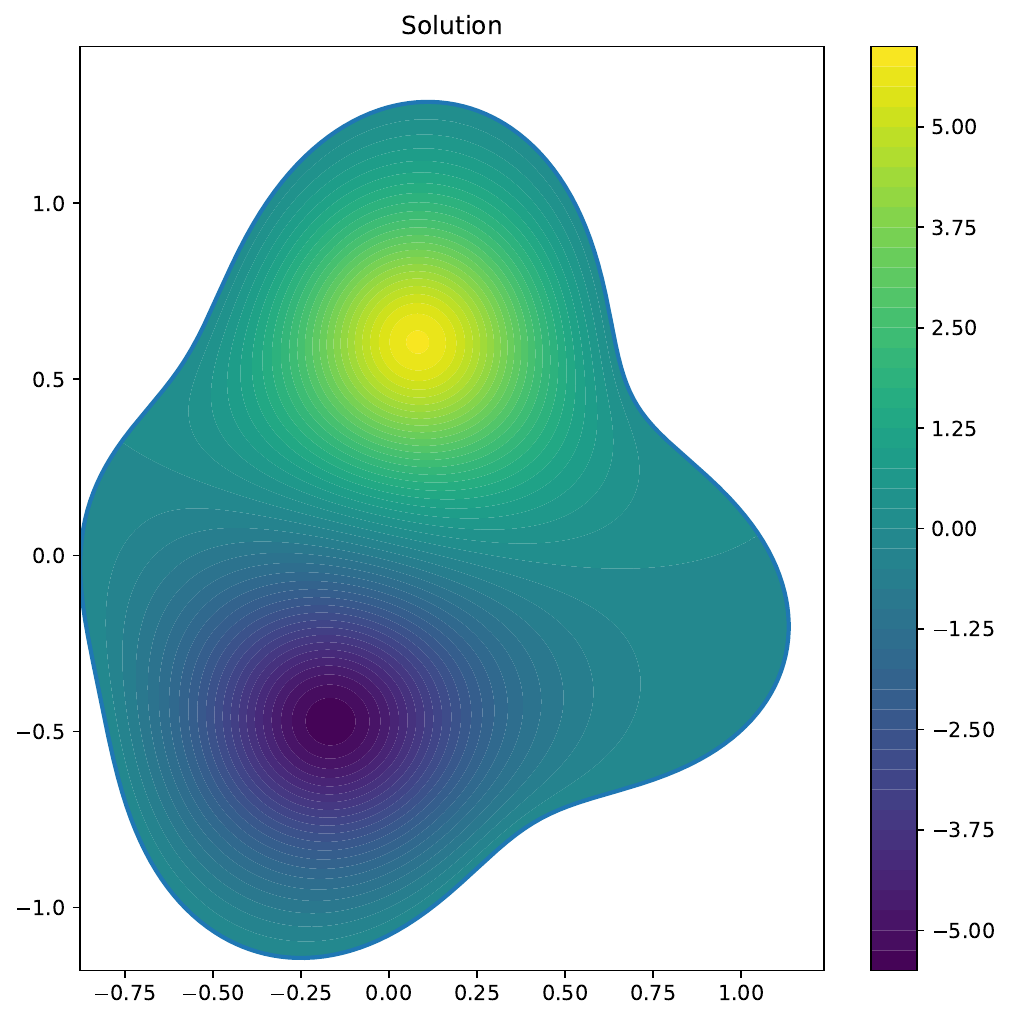}
	\captionof{figure}{{\tt Triblob} reference configuration. Left:
		$q_{r_0}=|f_{r_0}'|^2$ on the unit disk. Center and right:
		the positive and sign-changing solutions, respectively.}
	\label{fig:triblob}
\end{center}
\begin{center}
	\includegraphics[height=4cm,width=5cm]{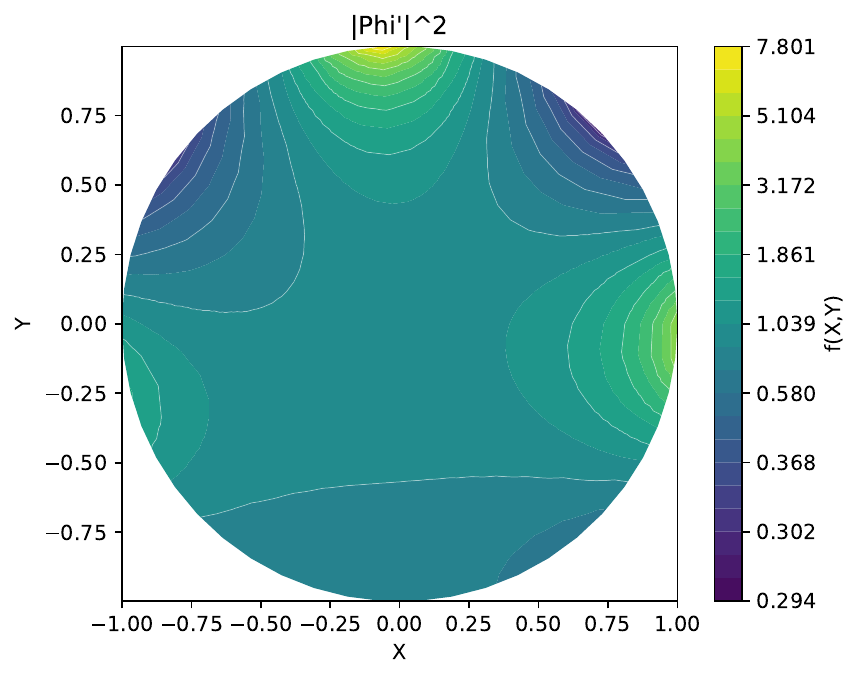}
	\qquad
	\includegraphics[height=4cm,width=4.5cm]{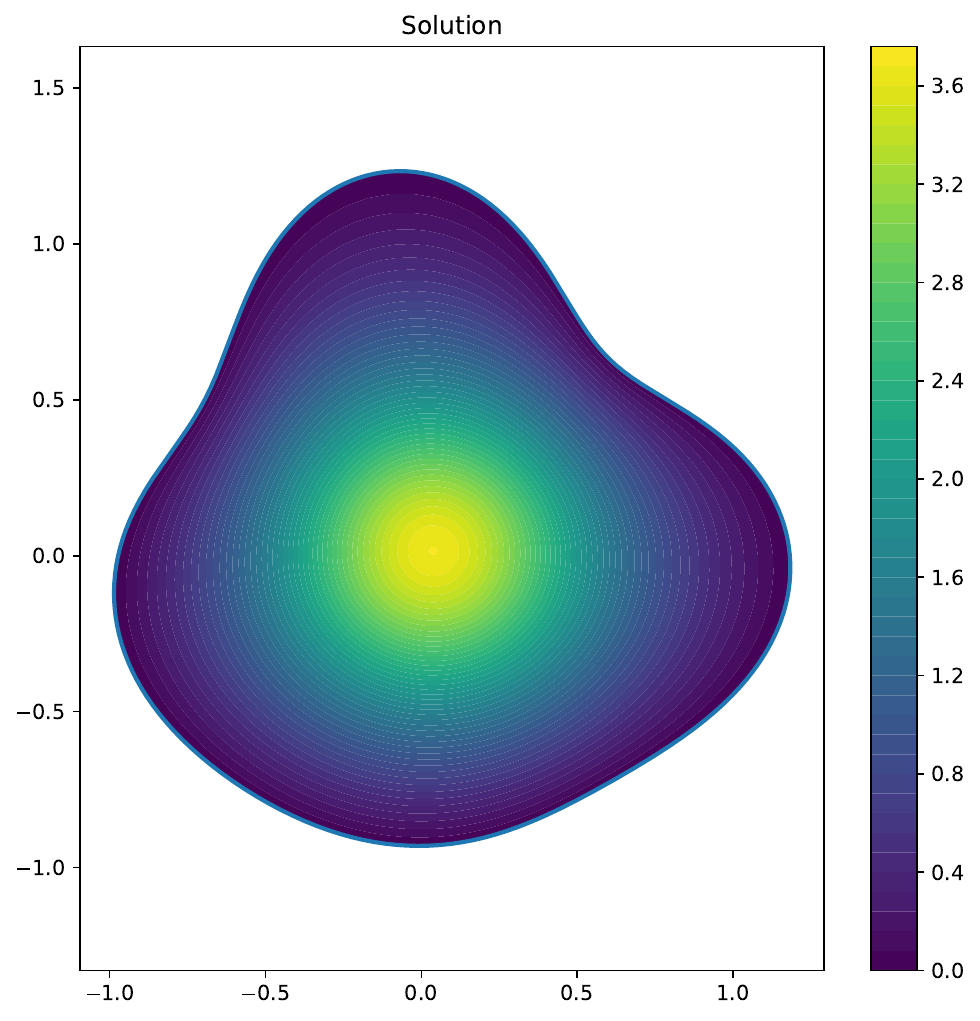}
\captionof{figure}{{\tt Tripuff} reference configuration.}
\end{center}
\begin{center}
	\includegraphics[height=4cm,width=5cm]{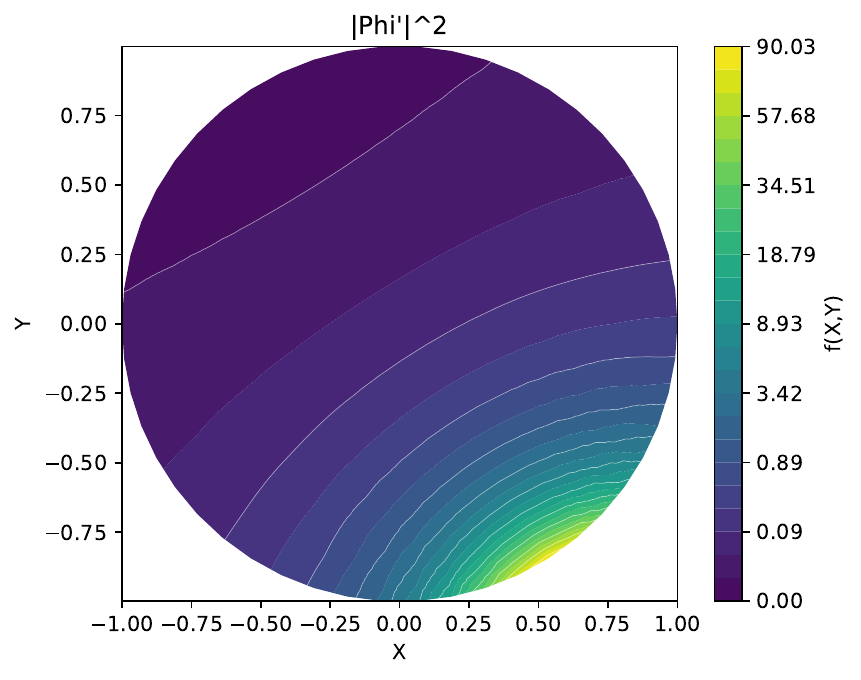}
	\qquad
	\includegraphics[height=4cm,width=4.5cm]{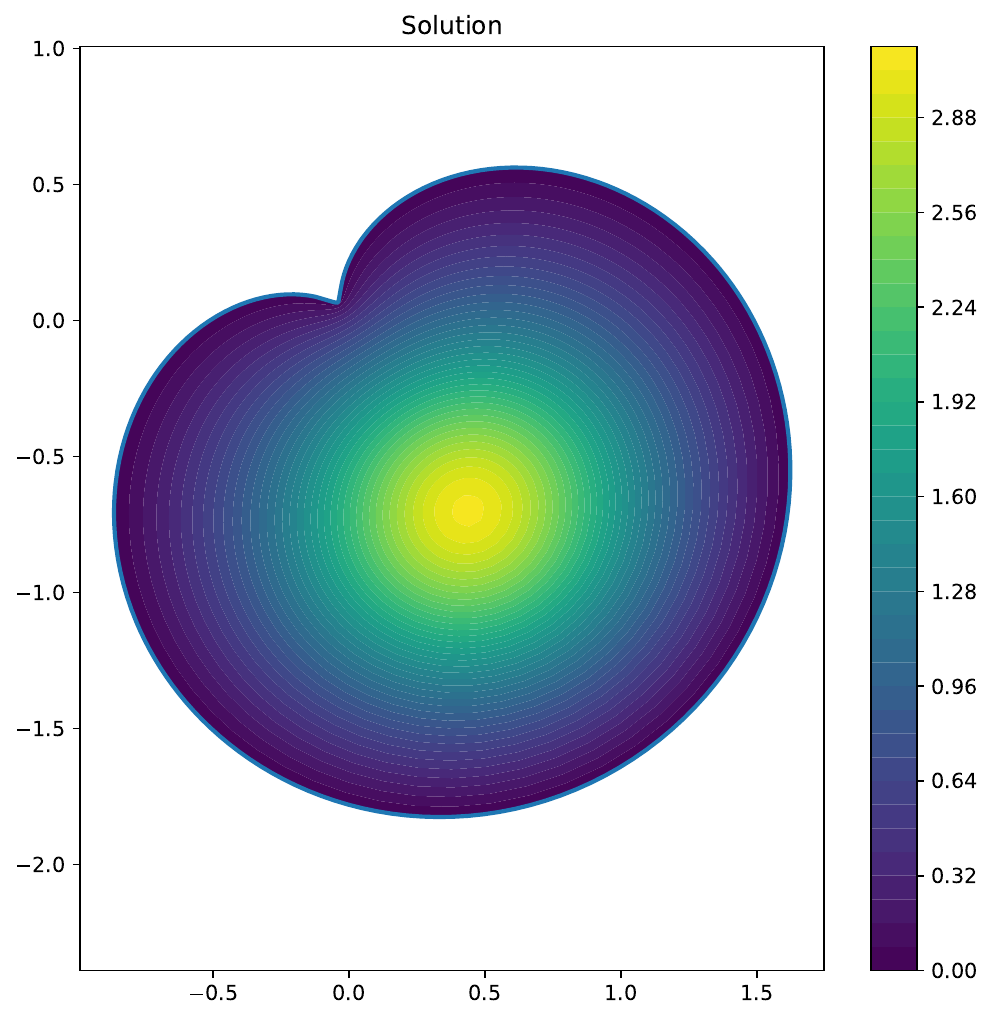}
\captionof{figure}{{\tt Cusp} reference configuration.}
\end{center}
\begin{center}
	\includegraphics[height=4cm,width=5cm]{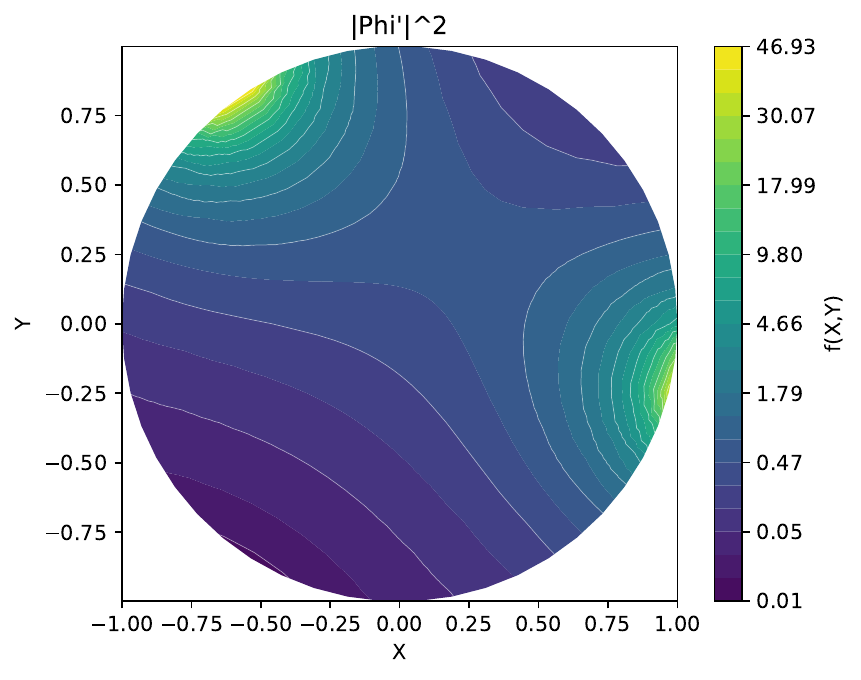}
	\qquad
	\includegraphics[height=4cm,width=4.5cm]{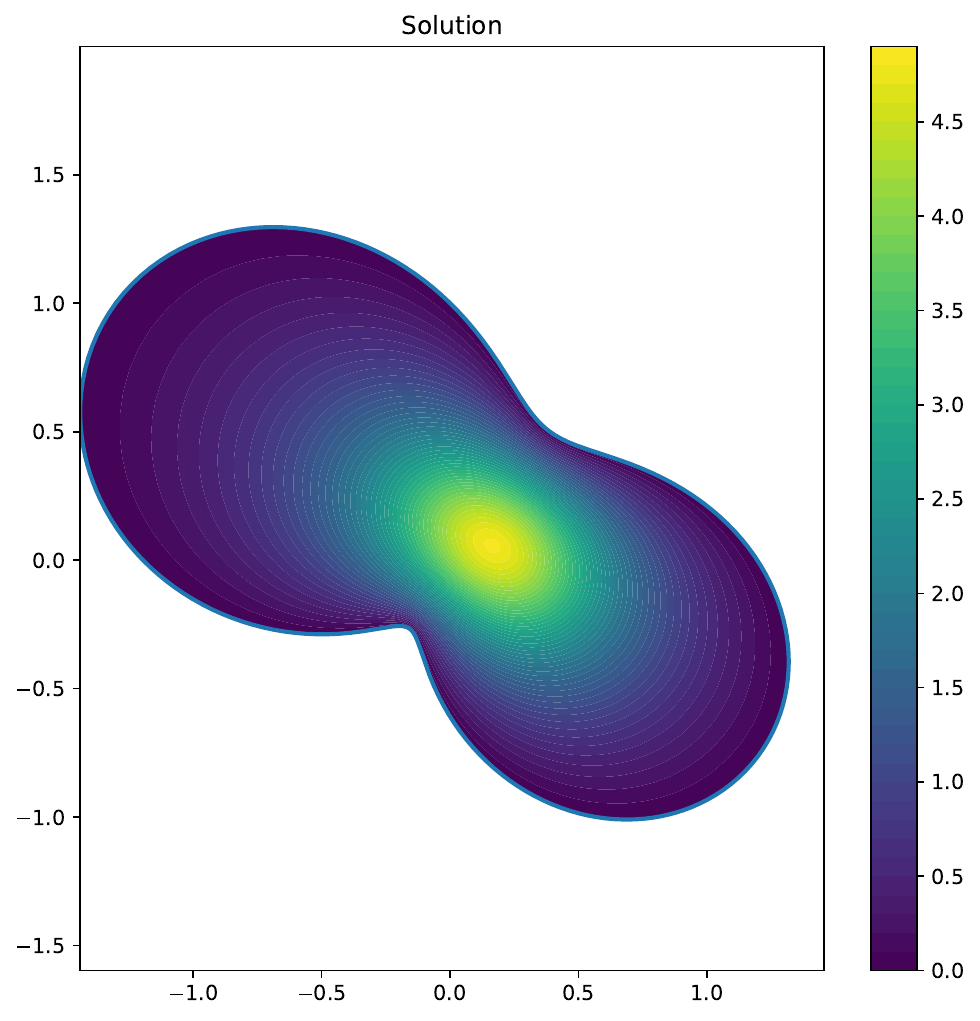}
	\qquad
	\includegraphics[height=4cm,width=4.5cm]{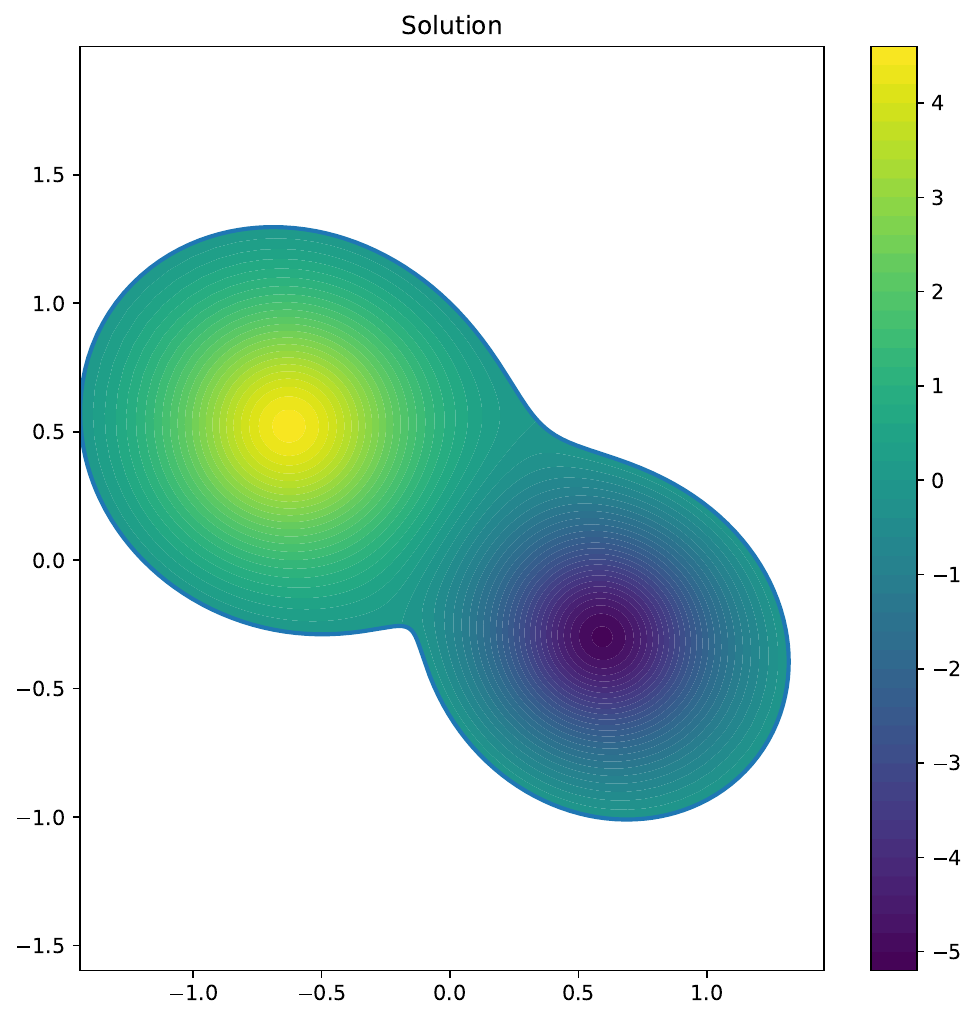}
\captionof{figure}{{\tt Eight and Eight$_*$} reference configuration.}
\end{center}
\newpage
\begin{center}
	\includegraphics[height=4cm,width=5cm]{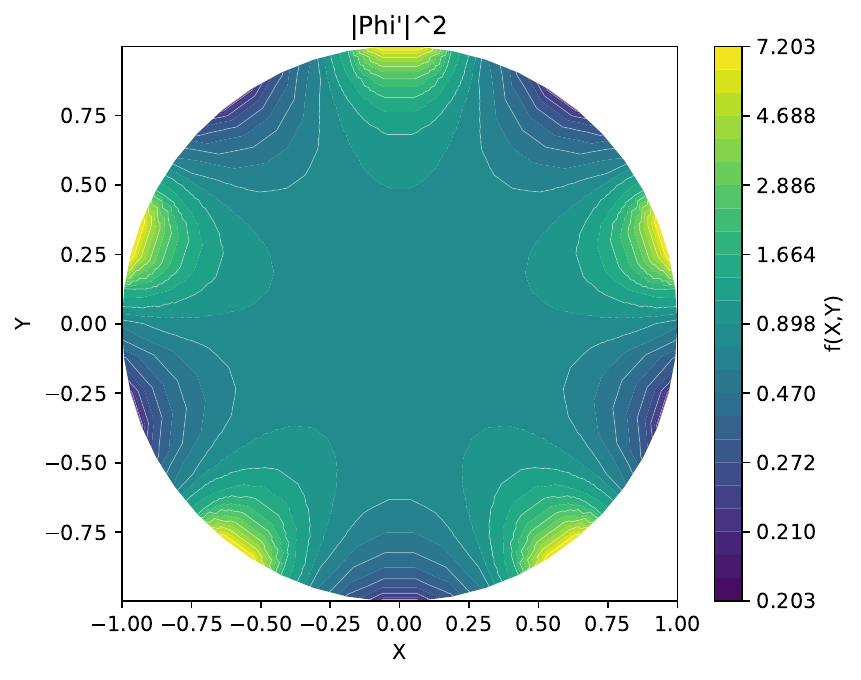}
	\qquad
	\includegraphics[height=4cm,width=4.5cm]{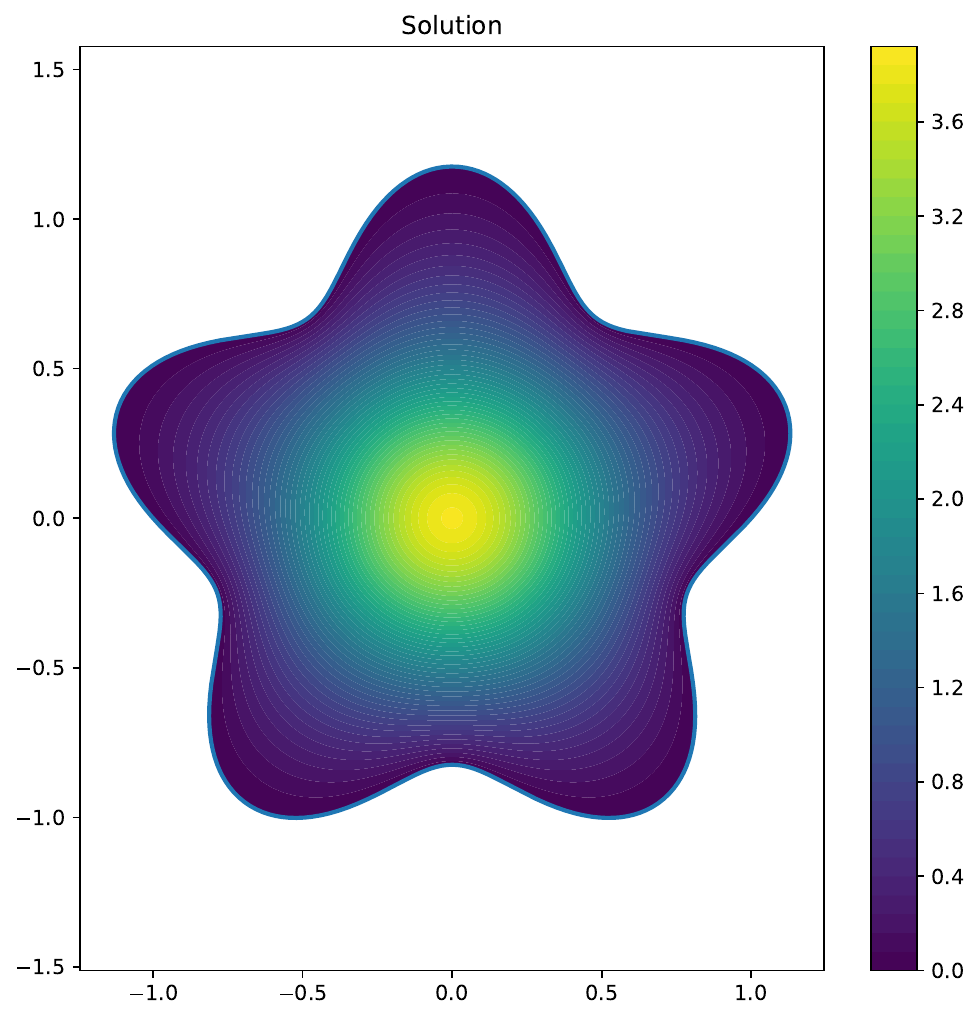}
\captionof{figure}{{\tt Five Star} reference configuration.}
\end{center}
\begin{center}
	\includegraphics[height=4cm,width=5cm]{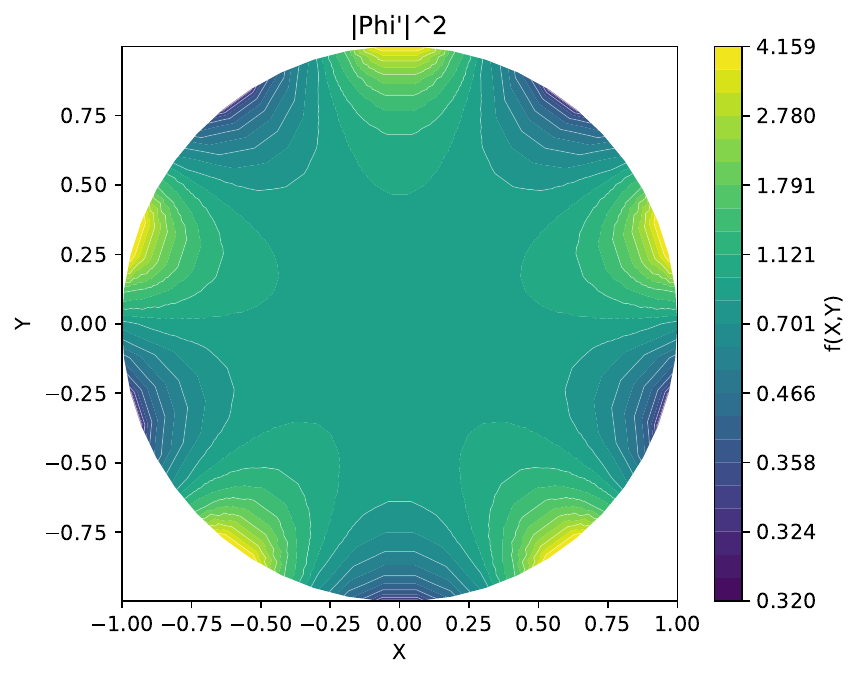}
	\qquad
	\includegraphics[height=4cm,width=4.5cm]{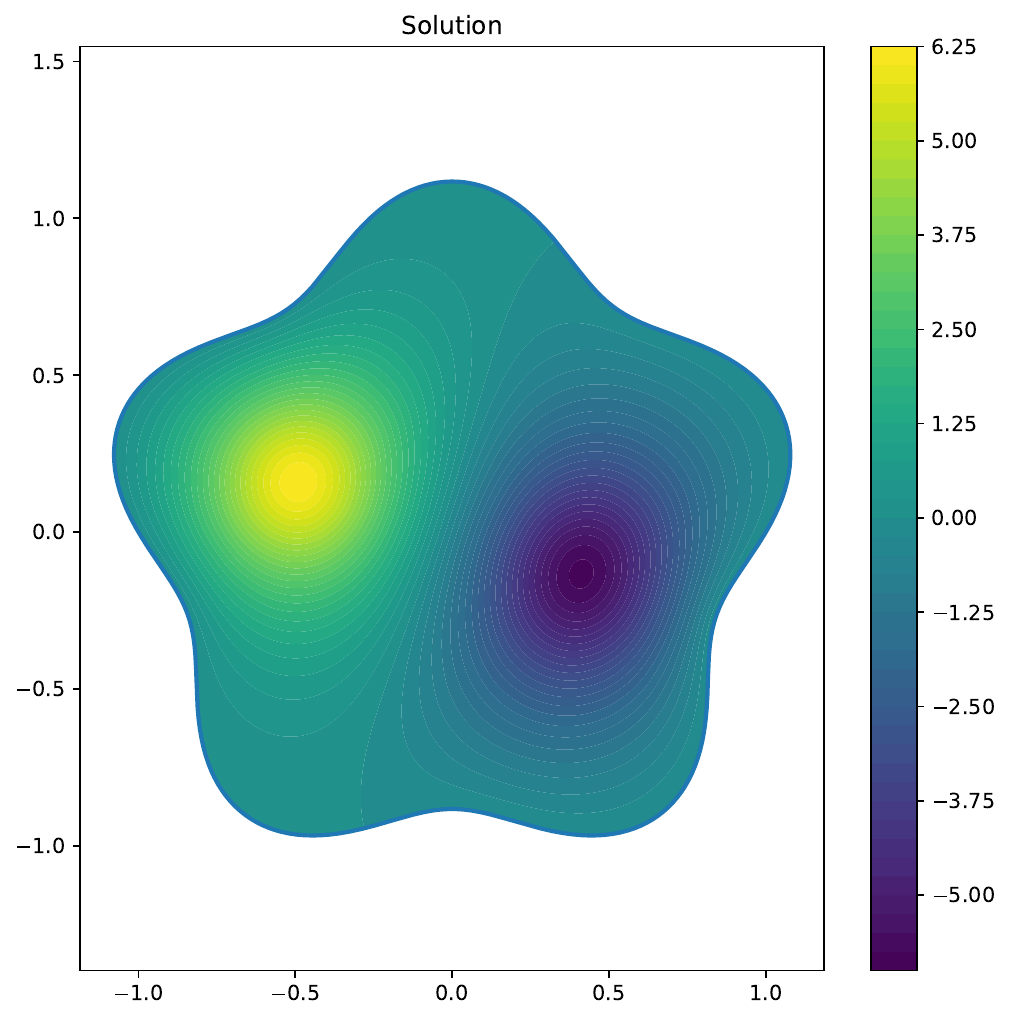}
\captionof{figure}{{\tt Five Star$_*$} reference configuration.}
\end{center}
\begin{center}
	\includegraphics[height=4cm,width=5cm]{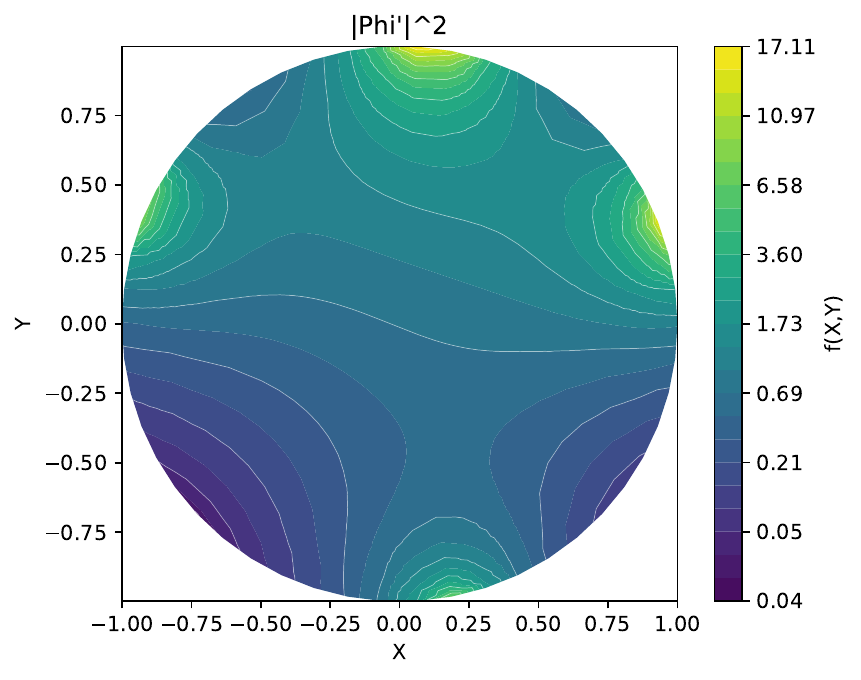}
	\qquad
	\includegraphics[height=4cm,width=4.5cm]{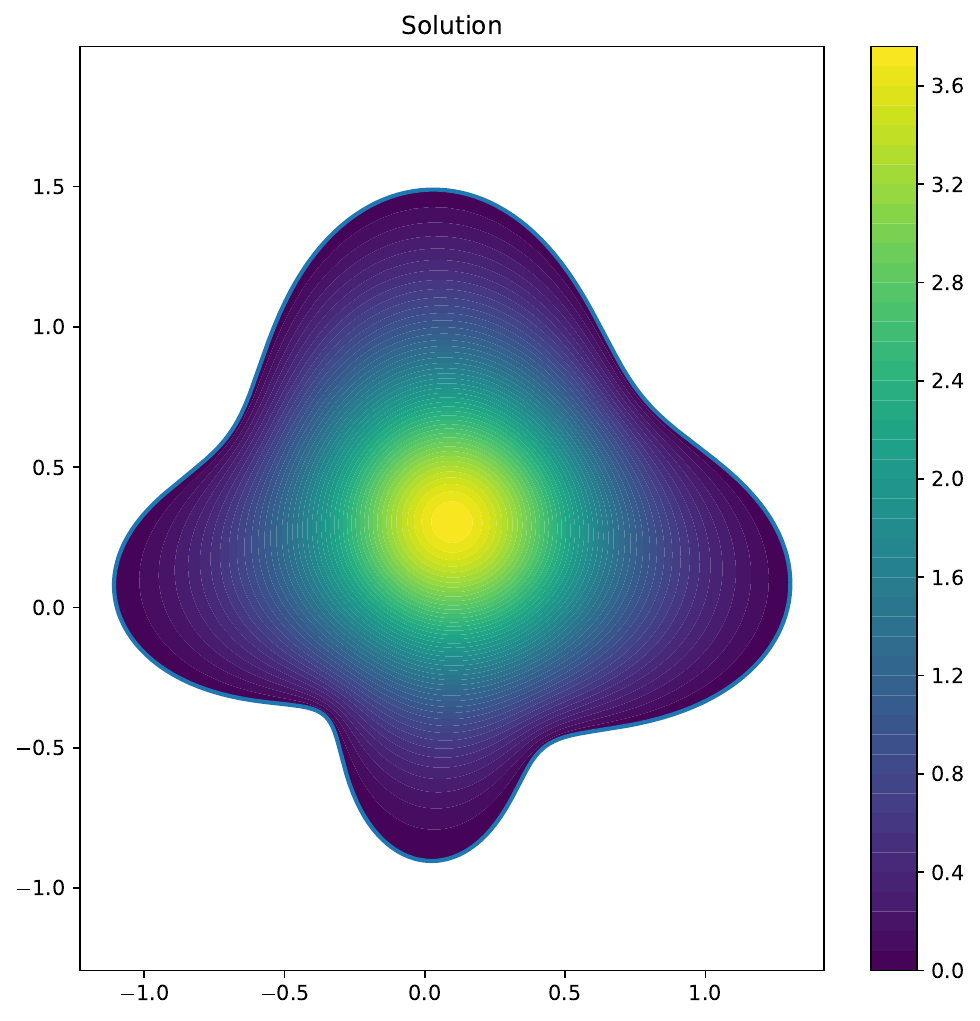}
\captionof{figure}{{\tt Mushroom} reference configuration.}
\end{center}
\begin{center}
	\includegraphics[height=4cm,width=5cm]{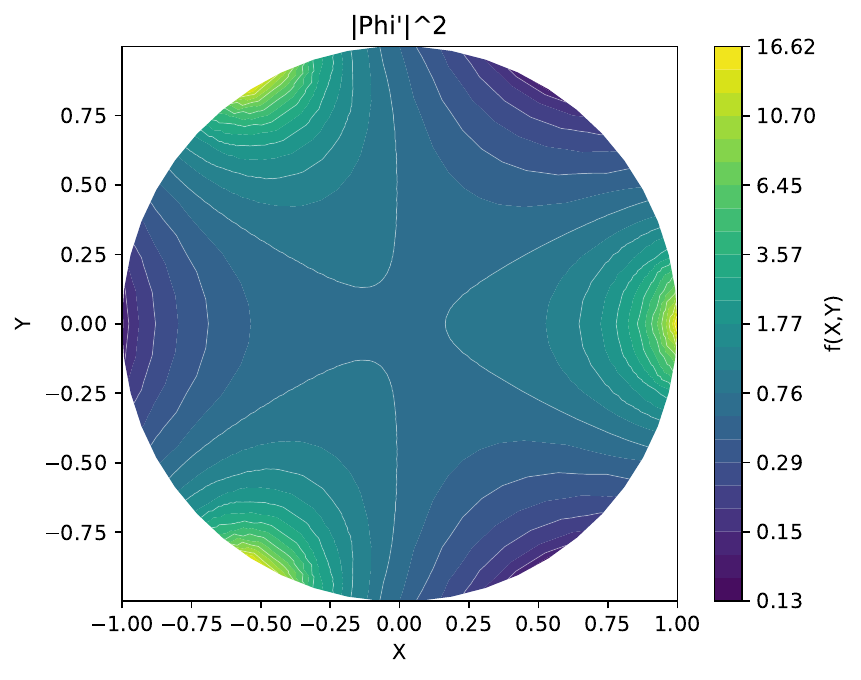}
	\qquad
	\includegraphics[height=4cm,width=4.5cm]{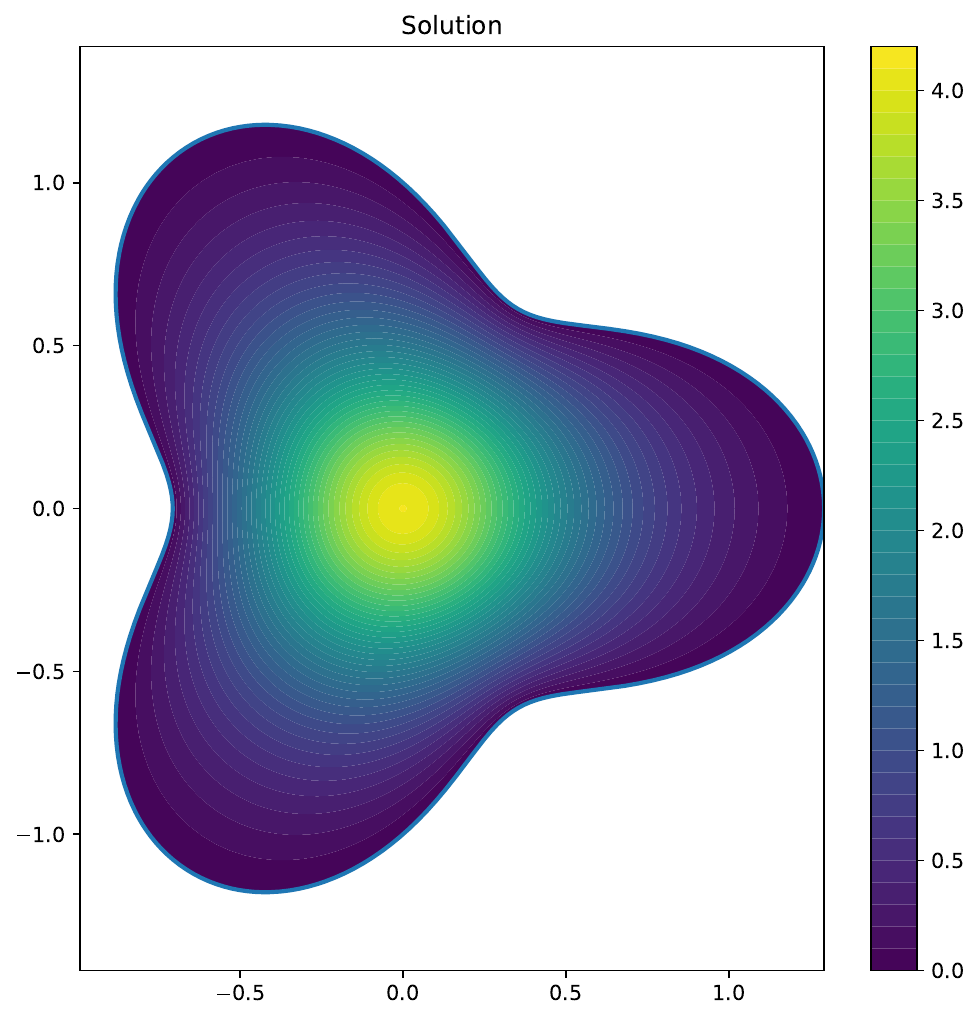}
\captionof{figure}{{\tt Shamrock} reference configuration.}
\newpage
\end{center}
\begin{center}
	\includegraphics[height=4cm,width=5cm]{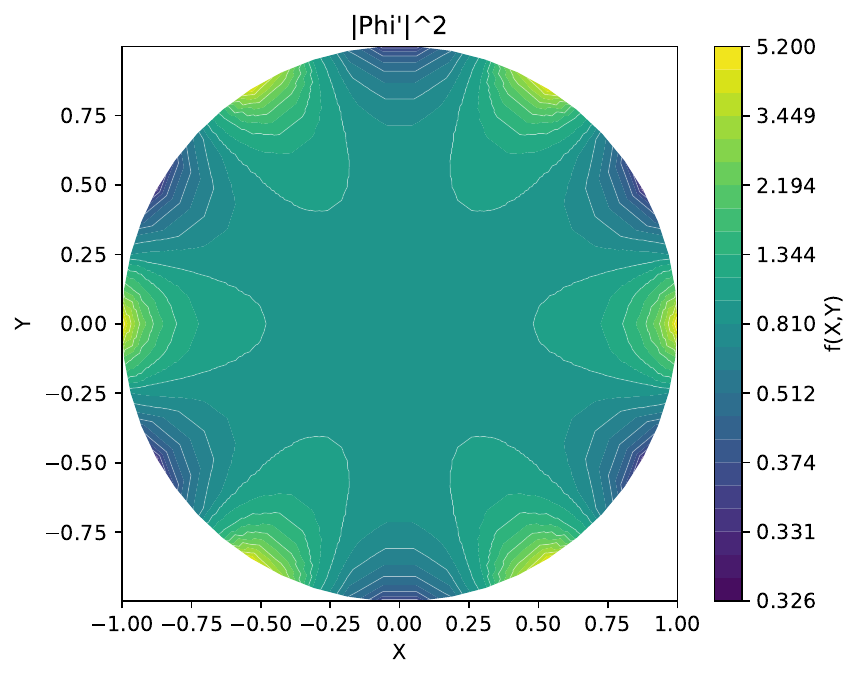}
	\qquad
	\includegraphics[height=4cm,width=4.5cm]{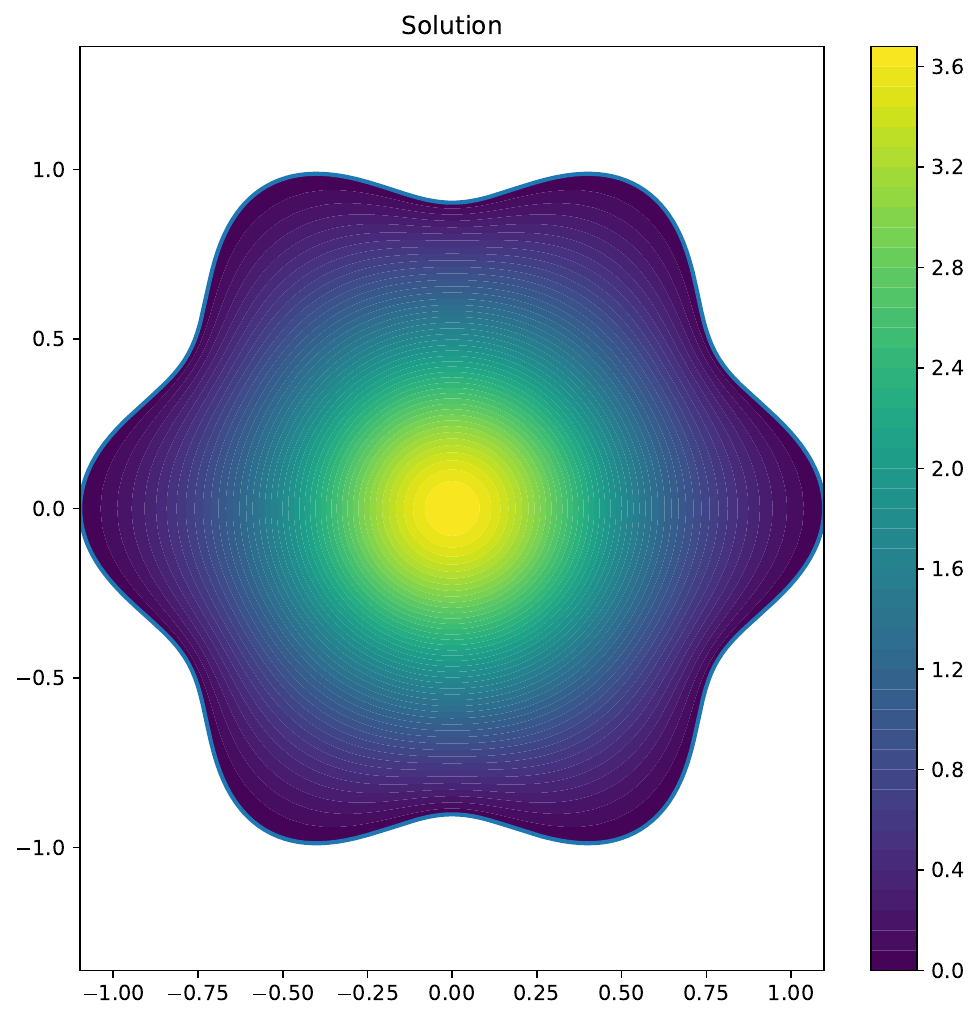}
	\qquad
	\includegraphics[height=4cm,width=4.5cm]{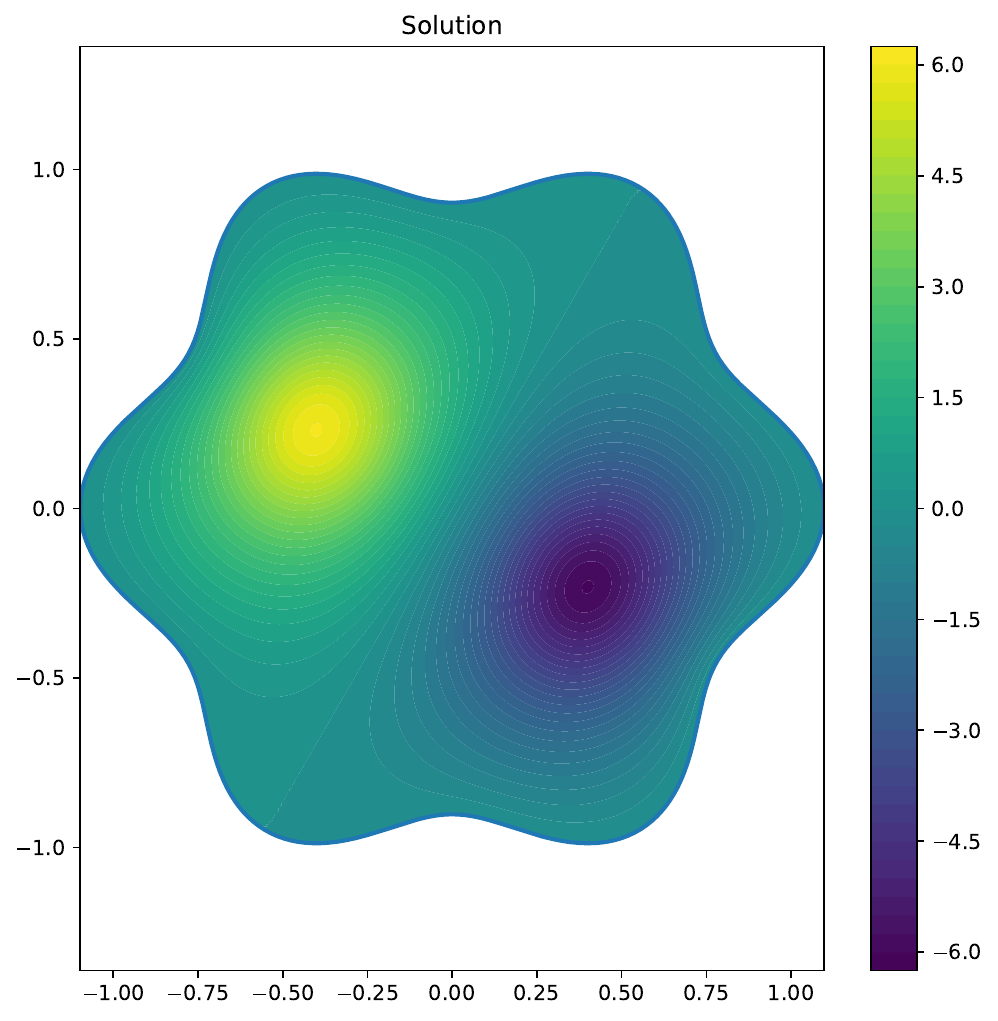}
\captionof{figure}{{\tt Six Star and Six Star$_*$} reference configuration.}
\end{center}
\begin{center}
	\includegraphics[height=4cm,width=5cm]{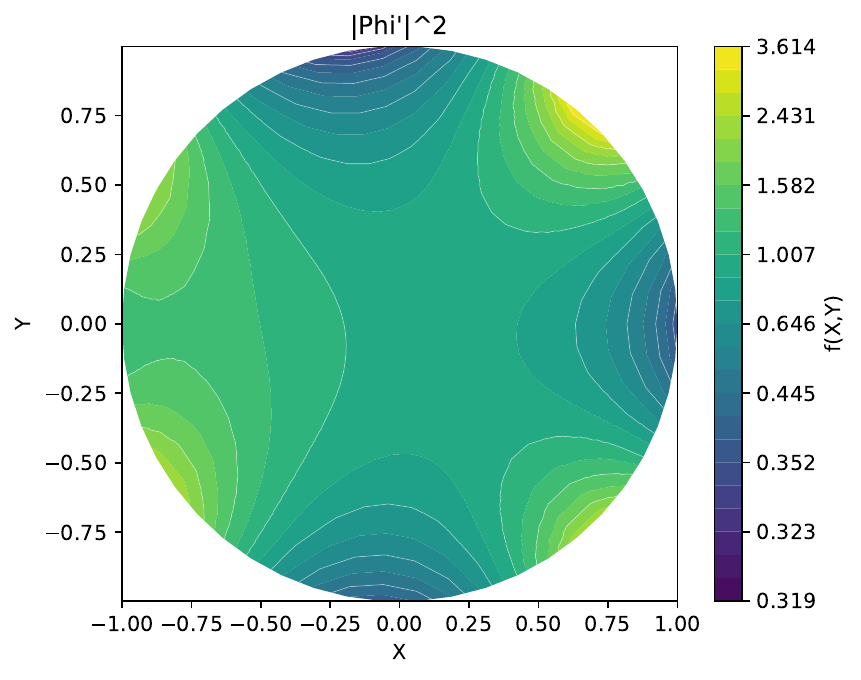}
	\qquad
	\includegraphics[height=4cm,width=4.5cm]{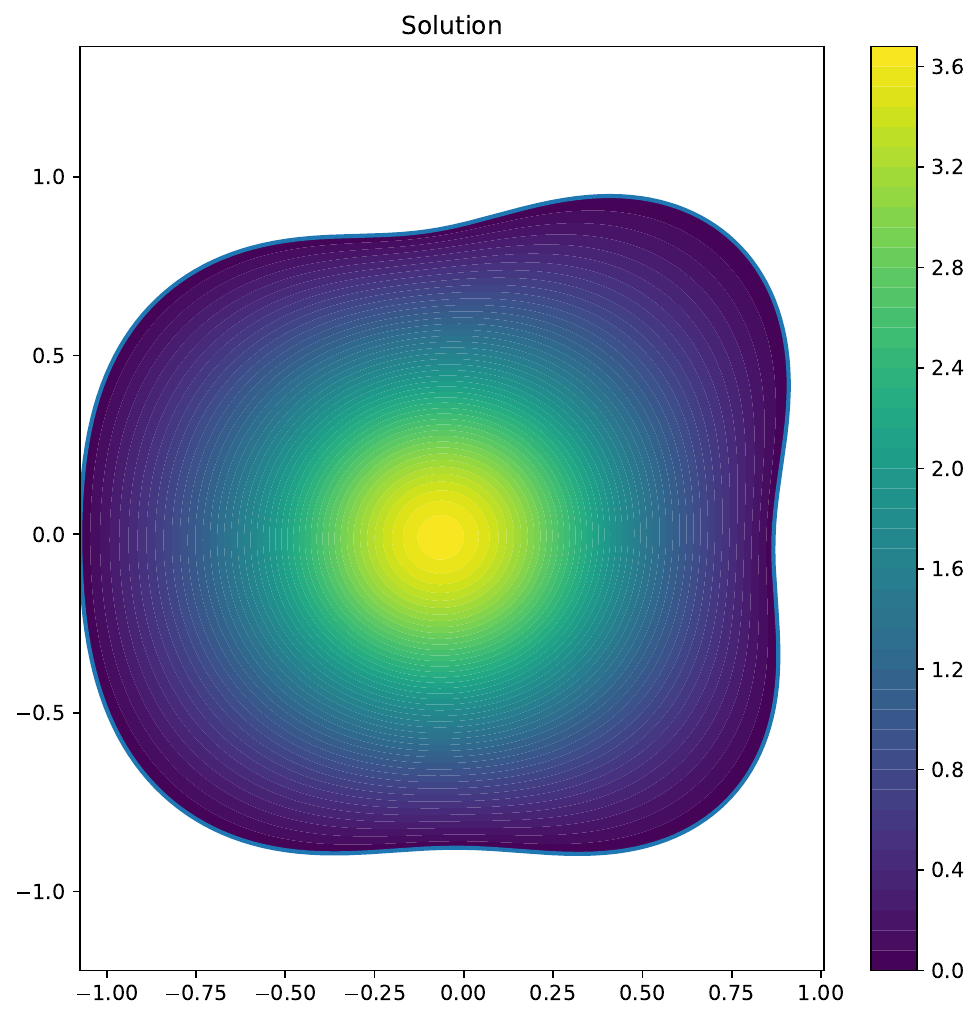}
	\qquad
	\includegraphics[height=4cm,width=4.5cm]{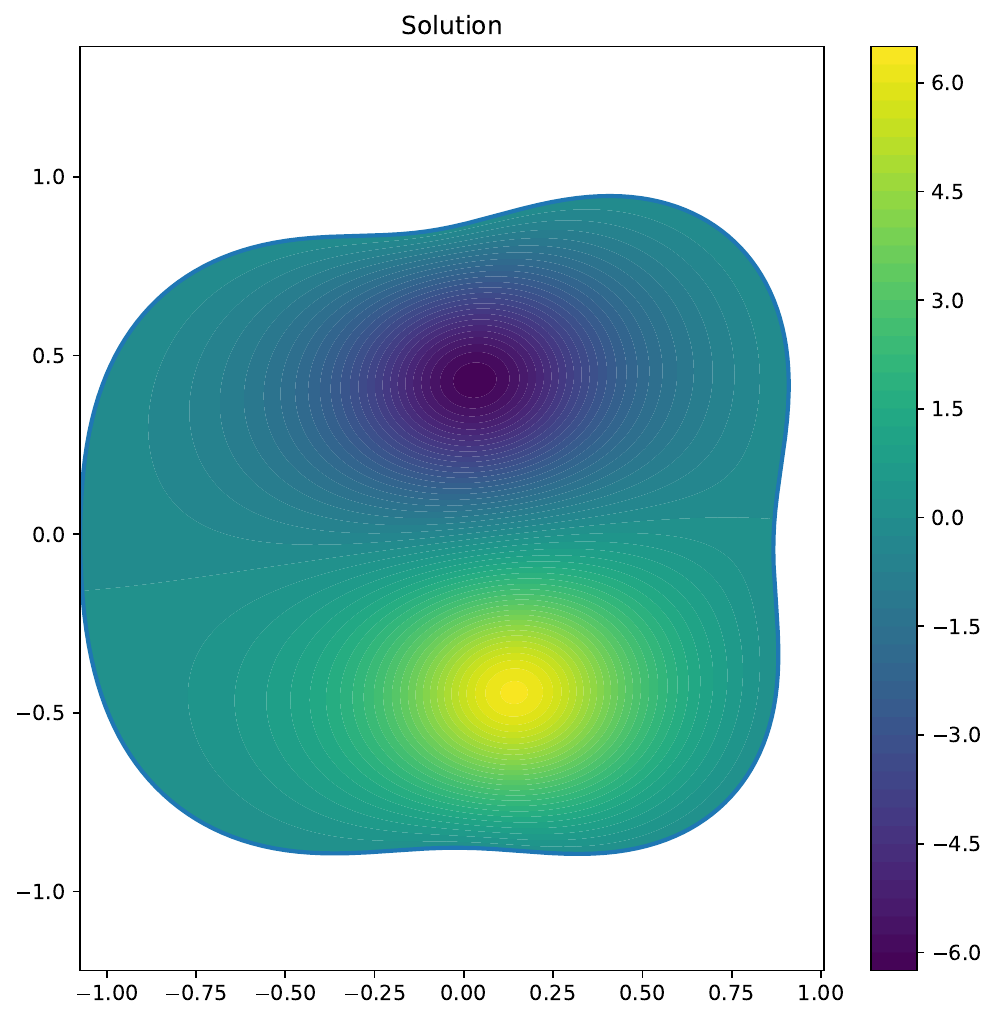}
\captionof{figure}{{\tt Pillow and Pillow$_*$} reference configuration.}
\end{center}
\begin{center}
	\includegraphics[height=4cm,width=5cm]{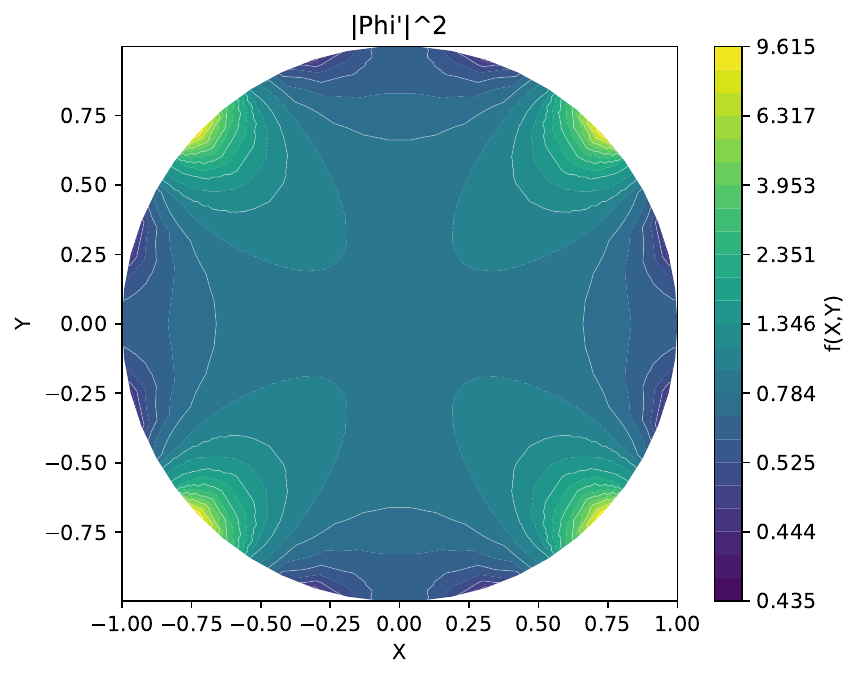}
	\qquad
	\includegraphics[height=4cm,width=4.5cm]{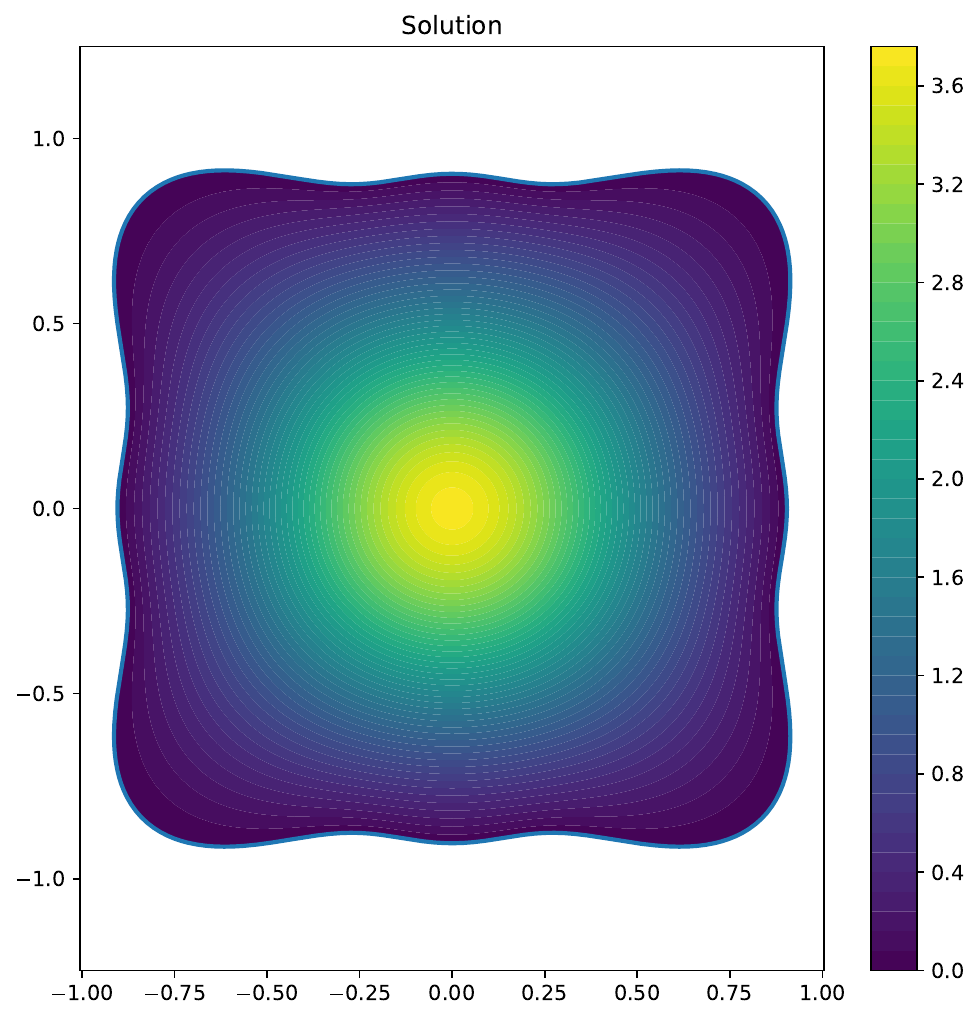}
\captionof{figure}{{\tt Square} reference configuration.}
\end{center}

\appendix

% ============================================================
\section{Weighted Wiener algebras}
\label{sec:appendix-wiener}
% ============================================================

\subsection{Fourier and Laurent representations}

Recall the weighted Fourier space
\[
\mathcal A_\rho
=
\left\{
u(t)
=
a_0+\sum_{k\ge1}
\bigl(a_k\cos(kt)+b_k\sin(kt)\bigr):
\|u\|_\rho<\infty
\right\},
\]
where
\[
\|u\|_\rho
=
|a_0|
+
\sum_{k\ge1}(|a_k|+|b_k|)\rho^k,
\qquad
\rho>1,
\]
and the weighted Laurent space
\[
\mathcal B_\rho
=
\left\{
f(z)=\sum_{k\in\mathbb Z}f_kz^k:
\|f\|_\rho
=
\sum_{k\in\mathbb Z}|f_k|\rho^{|k|}<\infty
\right\}.
\]

\begin{lemma}\label{lem:AandB}
	Let $\rho>1$.
	\begin{enumerate}
		\item Both $\mathcal A_\rho$ and $\mathcal B_\rho$ are commutative unital Banach algebras.
		\item
		Every $u\in\mathcal A_\rho$ extends holomorphically to the strip
		\[
		|\operatorname{Im}z|<\log\rho.
		\]
		
		\item
		Every $f\in\mathcal B_\rho$ defines a holomorphic function in the
		annulus
		\[
		\rho^{-1}<|z|<\rho.
		\]
		
		\item
		Under the correspondence $z=e^{it}$, the real-valued function
		\[
		u(t)
		=
		a_0+\sum_{k\ge1}
		\bigl(a_k\cos(kt)+b_k\sin(kt)\bigr)
		\]
		corresponds to
		\[
		f(z)=\sum_{k\in\mathbb Z}f_kz^k,
		\]
		where
		\[
		f_0=a_0,
		\qquad
		f_k=\frac{a_k-ib_k}{2},
		\qquad
		f_{-k}=\frac{a_k+ib_k}{2},
		\quad k\ge1.
		\]
		In particular,
		\[
		f_{-k}=\overline{f_k}.
		\]
		The correspondence is a bounded linear isomorphism between
		$\mathcal A_\rho$ and
		\[
		\mathcal B_\rho^{\rm sym}
		:=
		\left\{
		f\in\mathcal B_\rho:
		f_{-k}=\overline{f_k}
		\right\}.
		\]
		Moreover,
		\[
		\|f\|_\rho
		\le
		\|u\|_\rho
		\le
		\sqrt2\,\|f\|_\rho.
		\]
	\end{enumerate}
\end{lemma}

\begin{proof}
	The spaces being Banach algebras is a well-known fact.
	For $u\in\mathcal A_\rho$, the Fourier coefficients decay
	geometrically with weight $\rho^{-k}$. Hence the Fourier series
	converges absolutely and locally uniformly whenever
	\[
	e^{|\operatorname{Im}z|}<\rho,
	\]
	which proves (2).
	
	Similarly, if $f\in\mathcal B_\rho$, then
	\[
	\sum_{k\in\mathbb Z}|f_k|\rho^{|k|}<\infty.
	\]
	Thus the Laurent series converges absolutely and uniformly on every
	compact subannulus of
	\[
	\rho^{-1}<|z|<\rho.
	\]
	
	For the algebra property, if
	\[
	f(z)=\sum_{k\in\mathbb Z}f_kz^k,
	\qquad
	g(z)=\sum_{k\in\mathbb Z}g_kz^k,
	\]
	then
	\[
	(fg)_n
	=
	\sum_{k\in\mathbb Z}f_kg_{n-k}.
	\]
	Since
	\[
	\rho^{|n|}
	\le
	\rho^{|k|}\rho^{|n-k|},
	\]
	we obtain
	\[
	\|fg\|_\rho
	\le
	\|f\|_\rho\|g\|_\rho.
	\]
	Completeness follows from completeness of weighted $\ell^1$.
	
	Finally,
	\[
	\cos(kt)=\frac{e^{ikt}+e^{-ikt}}2,
	\qquad
	\sin(kt)=\frac{e^{ikt}-e^{-ikt}}{2i},
	\]
	which gives the stated correspondence. Moreover,
	\[
	|f_k|+|f_{-k}|
	=
	\sqrt{a_k^2+b_k^2},
	\]
	and
	\[
	\sqrt{a_k^2+b_k^2}
	\le
	|a_k|+|b_k|
	\le
	\sqrt2\,\sqrt{a_k^2+b_k^2}.
	\]
\end{proof}
We identify a periodic function with its associated Laurent function when no confusion can arise.
\begin{lemma}\label{lem:operator-A-B}
	Let $\rho>1$ and let
	\[
	T:\mathcal B_\rho\longrightarrow\mathcal B_\rho
	\]
	be a bounded linear operator such that
	\[
	T\bigl(\mathcal B_\rho^{\rm sym}\bigr)
	\subset
	\mathcal B_\rho^{\rm sym}.
	\]
	Through the correspondence of Lemma~\ref{lem:AandB}, $T$ induces
	a bounded linear operator
	\[
	T_{\mathcal A}:\mathcal A_\rho\longrightarrow\mathcal A_\rho.
	\]
	Then
	\begin{equation}\label{eq:operator-A-B}
		\|T_{\mathcal A}\|_{\mathcal A_\rho\to\mathcal A_\rho}
		\le
		\sqrt2\,
		\|T\|_{\mathcal B_\rho\to\mathcal B_\rho}.
	\end{equation}
\end{lemma}

\begin{proof}
	Let $u\in\mathcal A_\rho$ and let
	$f\in\mathcal B_\rho^{\rm sym}$ be its associated Laurent function.
	By Lemma~\ref{lem:AandB},
	\[
	\|f\|_{\mathcal B_\rho}
	\le
	\|u\|_{\mathcal A_\rho}.
	\]
	Since $T$ preserves $\mathcal B_\rho^{\rm sym}$, the function
	$Tf$ corresponds to a real-valued function $T_{\mathcal A}u$ in
	$\mathcal A_\rho$. Again by Lemma~\ref{lem:AandB},
	\[
	\|T_{\mathcal A}u\|_{\mathcal A_\rho}
	\le
	\sqrt2\,
	\|Tf\|_{\mathcal B_\rho}.
	\]
	Hence
	\[
	\begin{aligned}
		\|T_{\mathcal A}u\|_{\mathcal A_\rho}
		&\le
		\sqrt2\,
		\|T\|_{\mathcal B_\rho\to\mathcal B_\rho}
		\|f\|_{\mathcal B_\rho}
		\\
		&\le
		\sqrt2\,
		\|T\|_{\mathcal B_\rho\to\mathcal B_\rho}
		\|u\|_{\mathcal A_\rho}.
	\end{aligned}
	\]
	Taking the supremum over $u\ne0$ proves
	\eqref{eq:operator-A-B}.
\end{proof}
\begin{remark}\label{rem:symmetric-operators}
	Although some of the projections entering the Wiener--Hopf
	representation do not individually preserve
	$\mathcal B_\rho^{\rm sym}$, the operators corresponding to the
	real Theodorsen problem do. In particular, the exact linearisation
	$\Phi_0$, its validated approximation $\Phi_K$, the approximate
	inverse $L_K$, and the residual operators
	\[
	E_K=I-L_K\Phi_K,
	\qquad
	E_0=I-L_K\Phi_0,
	\]
	map $\mathcal B_\rho^{\rm sym}$ into itself.
	Consequently, their operator norms in $\mathcal A_\rho$ may be
	bounded from their Laurent-space norms by
	Lemma~\ref{lem:operator-A-B}.
\end{remark}

% ------------------------------------------------------------
\subsection{Multiplicative inversion}
% ------------------------------------------------------------

The following elementary estimate is used repeatedly.

\begin{lemma}\label{lem:inverse-banach}
	Let $\mathcal C$ be a commutative unital Banach algebra and let
	$f,g\in\mathcal C$. Assume that
	\[
	\|1-fg\|\le\varepsilon<1.
	\]
	Then $g$ is invertible and
	\[
	\|g^{-1}\|
	\le
	\frac{\|f\|}{1-\varepsilon}.
	\]
	Moreover,
	\[
	\|g^{-1}-f\|
	\le
	\frac{\varepsilon}{1-\varepsilon}\|f\|.
	\]
\end{lemma}

\begin{proof}
	Set
	\[
	r:=fg-1.
	\]
	Since $\|r\|<1$,
	\[
	(1+r)^{-1}
	=
	\sum_{n=0}^\infty(-r)^n,
	\qquad
	\|(1+r)^{-1}\|
	\le
	\frac1{1-\|r\|}.
	\]
	Since
	\[
	g^{-1}=(1+r)^{-1}f,
	\]
	the first estimate follows. Furthermore,
	\[
	g^{-1}-f=-rg^{-1},
	\]
	and hence
	\[
	\|g^{-1}-f\|
	\le
	\varepsilon\|g^{-1}\|
	\le
	\frac{\varepsilon}{1-\varepsilon}\|f\|.
	\]
\end{proof}

% ------------------------------------------------------------
\subsection{Differentiation between weighted spaces}
% ------------------------------------------------------------

Differentiation is unbounded on a fixed exponentially weighted
Wiener algebra, but becomes bounded after a loss of exponential
weight.

\begin{lemma}\label{lem:derivative-A}
	Let
	\[
	1<\rho_2<\rho_1,
	\]
	and let $u\in\mathcal A_{\rho_1}$. Then
	\[
	u'\in\mathcal A_{\rho_2}
	\]
	and
	\begin{equation}\label{eq:derivative-A}
		\|u'\|_{\rho_2}
		\le
		C^F_{\rho_1,\rho_2}\|u\|_{\rho_1},
	\end{equation}
	where
	\begin{equation}\label{eq:C-weight}
		C^F_{\rho_1,\rho_2}
		:=
		\sup_{k\ge1}
		k
		\left(
		\frac{\rho_2}{\rho_1}
		\right)^k.
	\end{equation}
	In particular,
	\begin{equation}\label{eq:C-weight-simple}
		C^F_{\rho_1,\rho_2}\le\frac{1}{e\log(\rho_1/\rho_2)
		}.
	\end{equation}
\end{lemma}

\begin{proof}
	Write
	\[
	u(t)
	=
	a_0+\sum_{k\ge1}
	\bigl(a_k\cos(kt)+b_k\sin(kt)\bigr).
	\]
	Then
	\[
	u'(t)
	=
	\sum_{k\ge1}
	k\bigl(
	b_k\cos(kt)-a_k\sin(kt)
	\bigr),
	\]
	and hence
	\[
		\|u'\|_{\rho_2}=\sum_{k\ge1}k(|a_k|+|b_k|)\rho_2^k\le
		\sup_{k\ge1}\left[k\left(\frac{\rho_2}{\rho_1}\right)^k\right]\sum_{k\ge1}
		(|a_k|+|b_k|)\rho_1^k.
	\]
	Finally, if
	\[
	q=\frac{\rho_2}{\rho_1}<1,
	\]
	then
	\[
	\sup_{x>0}xq^x
	=
	\frac1{e|\log q|},
	\]
	which gives \eqref{eq:C-weight-simple}.
\end{proof}

For Taylor series we use the following corresponding estimate.

\begin{lemma}\label{lem:derivative}
	Let $N\ge0$,
	\[
	1<\rho_2<\rho_1,
	\]
	and let
	\[
	E(z)=\sum_{k=N}^\infty f_kz^k
	\]
	satisfy
	\[
	\|E\|_{\rho_1}
	=
	\sum_{k=N}^\infty|f_k|\rho_1^k
	<\infty.
	\]
	Then
	\[
	E'(z)
	=
	\sum_{k=N}^\infty kf_kz^{k-1}
	\]
	belongs to $\mathcal B_{\rho_2}$ and
	\begin{equation}\label{eq:derivative-tail-general}
		\|E'\|_{\rho_2}
		\le
		\frac1{\rho_2}
		\sup_{k\ge N}
		\left[
		k\left(\frac{\rho_2}{\rho_1}\right)^k
		\right]
		\|E\|_{\rho_1}.
	\end{equation}
\end{lemma}

\begin{proof}
	Set
	\[
	q:=\frac{\rho_2}{\rho_1}<1.
	\]
	Then
	\[
		\|E'\|_{\rho_2}=\sum_{k=N}^\infty k|f_k|\rho_2^{k-1}=
		\frac1{\rho_2}\sum_{k=N}^\infty kq^k|f_k|\rho_1^k\le
		\frac1{\rho_2}\sup_{k\ge N}(kq^k)\|E\|_{\rho_1}.
	\]
\end{proof}

If
\[
N\ge-\frac1{\log q},
\]
then $kq^k$ is decreasing for $k\ge N$, and therefore
\begin{equation}\label{eq:derivative-tail-sharp}
	\|E'\|_{\rho_2}
	\le
	\frac{N}{\rho_2}
	\left(\frac{\rho_2}{\rho_1}\right)^N
	\|E\|_{\rho_1}.
\end{equation}
More generally,
\begin{equation}\label{eq:derivative-tail-uniform}
	\|E'\|_{\rho_2}
	\le
	\frac{
		\|E\|_{\rho_1}
	}{
		e\rho_2\log(\rho_1/\rho_2)
	}.
\end{equation}

% ------------------------------------------------------------
\subsection{Composition between weighted Wiener spaces}
\label{sec:appendix-composition}
% ------------------------------------------------------------

We now record the composition estimates used in
Section~\ref{sec:stability}.

\begin{lemma}\label{lem:composition-appendix}
	Let $1<\rho<\sigma$, let $g\in\mathcal A_\sigma$, and let
	$u\in\mathcal A_\rho$ satisfy
	\[
	\|u\|_\rho\le c.
	\]
	If
	\[
	\rho e^c\le\sigma,
	\]
	then
	\[
	g(\,\cdot+u(\cdot)\,)\in\mathcal A_\rho
	\]
	and
	\begin{equation}\label{eq:composition-appendix}
		\|g(\,\cdot+u(\cdot)\,)\|_\rho
		\le
		\|g\|_{\rho e^c}
		\le
		\|g\|_\sigma.
	\end{equation}
\end{lemma}

\begin{proof}
	Since the complexification of $\mathcal A_\rho$ is a Banach algebra,
	\[
	\|e^v\|_\rho
	\le
	e^{\|v\|_\rho}.
	\]
	Consequently,
	\[
	\|e^{\pm iku}\|_\rho
	\le
	e^{k\|u\|_\rho}
	\le
	e^{kc}.
	\]
	Since
	\[
	\|e^{\pm ikt}\|_\rho=\rho^k,
	\]
	we obtain
	\[
	\|e^{\pm ik(t+u(t))}\|_\rho
	\le
	(\rho e^c)^k.
	\]
	Thus
	\[
	\|\cos(k(\,\cdot+u\,))\|_\rho
	\le
	(\rho e^c)^k,
	\]
	and likewise for the sine modes.
	
	Writing
	\[
	g(t)
	=
	a_0+
	\sum_{k\ge1}
	\bigl(a_k\cos(kt)+b_k\sin(kt)\bigr)
	\]
	therefore gives
	\[
	\|g(\,\cdot+u(\cdot)\,)\|_\rho
	\le
	|a_0|
	+
	\sum_{k\ge1}
	(|a_k|+|b_k|)(\rho e^c)^k
	=
	\|g\|_{\rho e^c}.
	\]
\end{proof}

\begin{corollary}\label{cor:composition-difference}
	Under the assumptions of
	Lemma~\ref{lem:composition-appendix}, if
	$g_1,g_2\in\mathcal A_\sigma$, then
	\[
	\left\|
	g_1(\,\cdot+u\,)
	-
	g_2(\,\cdot+u\,)
	\right\|_\rho
	\le
	\|g_1-g_2\|_\sigma.
	\]
\end{corollary}

\begin{proof}
	Apply Lemma~\ref{lem:composition-appendix} to $g_1-g_2$.
\end{proof}

\begin{lemma}\label{lem:composition-u-appendix}
	Let
	\[
	1<\rho,
	\qquad
	\rho e^c<\tau,
	\]
	and let $g\in\mathcal A_\tau$. If
	\[
	u_1,u_2\in\mathcal A_\rho,
	\qquad
	\|u_i\|_\rho\le c,
	\]
	then
	\begin{equation}\label{eq:composition-u-appendix}
		\|g(\,\cdot+u_1\,)-g(\,\cdot+u_2\,)\|_\rho
		\le
		C_{\tau,\rho e^c}
		\|g\|_\tau
		\|u_1-u_2\|_\rho,
	\end{equation}
	where
	\[
	C_{\tau,\rho e^c}
	=
	\sup_{k\ge1}
	k
	\left(
	\frac{\rho e^c}{\tau}
	\right)^k.
	\]
\end{lemma}

\begin{proof}
	For $s\in[0,1]$, set
	\[
	u_s=u_2+s(u_1-u_2).
	\]
	Then
	\[
	g(\,\cdot+u_1\,)-g(\,\cdot+u_2\,)
	=
	\int_0^1
	g'(\,\cdot+u_s\,)(u_1-u_2)\,ds.
	\]
	Since
	\[
	\|u_s\|_\rho\le c,
	\]
	Lemma~\ref{lem:composition-appendix} gives
	\[
	\|g'(\,\cdot+u_s\,)\|_\rho
	\le
	\|g'\|_{\rho e^c}.
	\]
	The differentiation estimate gives
	\[
	\|g'\|_{\rho e^c}
	\le
	C_{\tau,\rho e^c}\|g\|_\tau.
	\]
	The result follows from the Banach algebra property.
\end{proof}

% ============================================================
\section{Wiener--Hopf factorisation}
\label{sec:appendix-WH}
% ============================================================

In this appendix we describe the construction and a posteriori
validation of the Wiener--Hopf factors used in
Section~\ref{sec:WH}. Numerical root splitting provides accurate
initial approximations, while
Proposition~\ref{prop:WH-validation} certifies the factors and their
inverses in the weighted Laurent algebra. Thus no rigorous enclosure
of the individual roots is required.

We use the closed subalgebras
\[
\mathcal B_\rho^+
=
\left\{
\sum_{k\ge0}f_kz^k\in\mathcal B_\rho
\right\},
\qquad
\mathcal B_\rho^-
=
\left\{
\sum_{k\le0}f_kz^k\in\mathcal B_\rho
\right\}.
\]
We also set
\[
Q_+:=P_++P_0,
\qquad
Q_-:=P_-,
\]
so that
\[
I=Q_++Q_-.
\]

% ------------------------------------------------------------
\subsection{Root splitting for Laurent polynomials}
\label{subsec:WH}
% ------------------------------------------------------------

Let
\[
p(z)
=
\sum_{k=-m}^{n}a_kz^k
=
z^{-m}P(z),
\]
where
\[
P(z)=\sum_{j=0}^{m+n}b_jz^j.
\]
Assume that
\[
p(z)\neq0
\qquad
\text{for }|z|=1.
\]
Writing
\[
P(z)
=
c\prod_{j=1}^{m+n}(z-\alpha_j),
\]
we split the roots according to the unit circle,
\[
\{\alpha_j\}
=
\{\alpha_j^-:|\alpha_j^-|<1\}
\cup
\{\alpha_j^+:|\alpha_j^+|>1\}.
\]

If $N_{\rm in}$ denotes the number of zeros of $P$ in the unit disk,
counted with multiplicity, then
\[
\operatorname{ind}_{|z|=1}p
=
N_{\rm in}-m.
\]
Thus the factorisation has index zero precisely when
\[
N_{\rm in}=m.
\]

For $|\alpha|<1$ and $|z|=1$,
\[
\frac1{z-\alpha}
=
\frac1z\frac1{1-\alpha/z}
=
\sum_{\ell=0}^{\infty}
\alpha^\ell z^{-\ell-1},
\]
whereas for $|\alpha|>1$,
\[
\frac1{z-\alpha}
=
-\frac1\alpha\frac1{1-z/\alpha}
=
-\sum_{\ell=0}^{\infty}
\alpha^{-\ell-1}z^\ell.
\]
Thus roots inside and outside the unit circle yield respectively
negative- and nonnegative-mode expansions of the corresponding
inverse factors.

This splitting is used only to construct accurate initial
approximations. Membership and invertibility in $\mathcal B_\rho$ are
certified a posteriori below.

% ------------------------------------------------------------
\subsection{Improvement of approximate Wiener--Hopf factors}
% ------------------------------------------------------------

Suppose that
\[
G\approx G_-G_+,
\]
where
\[
G_-\in\mathcal B_\rho^-,
\qquad
G_+\in\mathcal B_\rho^+.
\]
Define
\[
R
:=
G_-^{-1}GG_+^{-1}
=
1+E.
\]
We seek corrected factors
\[
\widetilde G_-
=
(1+A_-)G_-,
\qquad
\widetilde G_+
=
G_+(1+A_+),
\]
where
\[
A_-=Q_-A_-,
\qquad
A_+=Q_+A_+.
\]
Then
\[
\widetilde G_-^{-1}
G
\widetilde G_+^{-1}
=
(1+A_-)^{-1}
R
(1+A_+)^{-1}.
\]
To first order,
\[
(1+A_-)^{-1}
R
(1+A_+)^{-1}
=
1+E-A_--A_+
+
O(\|E\|_\rho^2).
\]
Using
\[
E_+:=Q_+E,
\qquad
E_-:=Q_-E,
\]
we choose
\[
A_+=E_+,
\qquad
A_-=E_-.
\]
Thus
\[
G_-\leftarrow(1+E_-)G_-,
\qquad
G_+\leftarrow G_+(1+E_+).
\]
Repeating the correction reduces the residual before the final
interval validation.

% ------------------------------------------------------------
\subsection{A posteriori validation of Wiener--Hopf factors}
% ------------------------------------------------------------

\begin{lemma}\label{lem:WHlog}
	Let $R\in\mathcal B_\rho$ satisfy
	\[
	\|R-1\|_\rho\le\eta<1.
	\]
	Then
	\[
	L:=\log R
	\]
	is well defined in $\mathcal B_\rho$ and
	\[
	\|L\|_\rho
	\le
	-\log(1-\eta).
	\]
\end{lemma}

\begin{proof}
	Write
	\[
	R=1+E.
	\]
	Since $\|E\|_\rho<1$,
	\[
	L
	=
	\sum_{n=1}^{\infty}
	(-1)^{n+1}\frac{E^n}{n}
	\]
	converges in $\mathcal B_\rho$, and
	\[
	\|L\|_\rho
	\le
	\sum_{n=1}^{\infty}\frac{\eta^n}{n}
	=
	-\log(1-\eta).
	\]
\end{proof}

The Wiener--Hopf factors and their inverse do not need to be validated, since they are only needed to build an {\em approximate} inverse $L_K$ of the linear operator $\Phi$. The rigorous computation only involves the operator $I-L_K D\mathcal F_r(u)$. Nonetheless they could be validated for other purposes, and for completeness we describe the procedure to obtain such validation.

\begin{proposition}[A posteriori validation of Wiener--Hopf factors]
	\label{prop:WH-validation}
	Let
	\[
	G\in\mathcal B_\rho,
	\]
	and let
	\[
	G_-\in\mathcal B_\rho^-,
	\qquad
	G_+\in\mathcal B_\rho^+
	\]
	be approximate Wiener--Hopf factors with approximate inverses
	\[
	F_-\in\mathcal B_\rho^-,
	\qquad
	F_+\in\mathcal B_\rho^+.
	\]
	Define
	\[
	\varepsilon_1
	:=
	\|G-G_-G_+\|_\rho,
	\]
	and
	\[
	\varepsilon_-
	:=
	\|1-G_-F_-\|_\rho,
	\qquad
	\varepsilon_+
	:=
	\|1-G_+F_+\|_\rho.
	\]
	Assume
	\[
	\varepsilon_-<1,
	\qquad
	\varepsilon_+<1.
	\]
	Then $G_-$ and $G_+$ are invertible and
	\[
	\|G_-^{-1}\|_\rho
	\le
	\frac{\|F_-\|_\rho}{1-\varepsilon_-},
	\qquad
	\|G_+^{-1}\|_\rho
	\le
	\frac{\|F_+\|_\rho}{1-\varepsilon_+},
	\]
	as well as
	\[
	\|G_-^{-1}-F_-\|_\rho
	\le
	\frac{\|F_-\|_\rho\varepsilon_-}{1-\varepsilon_-},
	\]
	and
	\[
	\|G_+^{-1}-F_+\|_\rho
	\le
	\frac{\|F_+\|_\rho\varepsilon_+}{1-\varepsilon_+}.
	\]
	
	If, in addition,
	\[
	\eta
	:=
	\frac{
		\|F_-\|_\rho\,
		\varepsilon_1\,
		\|F_+\|_\rho
	}{
		(1-\varepsilon_-)(1-\varepsilon_+)
	}
	<1,
	\]
	then $G$ admits exact factors
	\[
	G=\widehat G_-\widehat G_+,
	\]
	with
	\[
	\widehat G_-\in\mathcal B_\rho^-,
	\qquad
	\widehat G_+\in\mathcal B_\rho^+,
	\]
	and
	\[
	\|\widehat G_-^{-1}-F_-\|_\rho
	\le
	\|F_-\|_\rho
	\left[
	\frac{\varepsilon_-}
	{(1-\eta)(1-\varepsilon_-)}
	+
	\frac{\eta}{1-\eta}
	\right],
	\]
	\[
	\|\widehat G_+^{-1}-F_+\|_\rho
	\le
	\|F_+\|_\rho
	\left[
	\frac{\varepsilon_+}
	{(1-\eta)(1-\varepsilon_+)}
	+
	\frac{\eta}{1-\eta}
	\right].
	\]
\end{proposition}

\begin{proof}
	Lemma~\ref{lem:inverse-banach} gives the inverse estimates for
	$G_\pm$. Define
	\[
	R:=G_-^{-1}GG_+^{-1}.
	\]
	Then
	\[
	R-1
	=
	G_-^{-1}
	(G-G_-G_+)
	G_+^{-1},
	\]
	and therefore
	\[
	\|R-1\|_\rho\le\eta<1.
	\]
	By Lemma~\ref{lem:WHlog},
	\[
	L:=\log R
	\]
	is well defined and
	\[
	\|L\|_\rho
	\le
	\ell,
	\qquad
	\ell:=-\log(1-\eta).
	\]
	Set
	\[
	L_-:=Q_-L,
	\qquad
	L_+:=Q_+L.
	\]
	Then
	\[
	L=L_-+L_+.
	\]
	Since $\mathcal B_\rho$ is commutative,
	\[
	R=e^L=e^{L_-}e^{L_+}.
	\]
	Hence exact factors are
	\[
	\widehat G_-:=G_-e^{L_-},
	\qquad
	\widehat G_+:=e^{L_+}G_+.
	\]
	
	For the minus factor,
	\[
	\widehat G_-^{-1}
	=
	e^{-L_-}G_-^{-1},
	\]
	and therefore
	\[
	\widehat G_-^{-1}-F_-
	=
	e^{-L_-}(G_-^{-1}-F_-)
	+
	(e^{-L_-}-1)F_-.
	\]
	Since
	\[
	\|L_-\|_\rho\le\ell,
	\]
	we have
	\[
	\|e^{-L_-}\|_\rho
	\le
	e^\ell
	=
	\frac1{1-\eta},
	\]
	and
	\[
	\|e^{-L_-}-1\|_\rho
	\le
	e^\ell-1
	=
	\frac{\eta}{1-\eta}.
	\]
	The stated estimate follows. The proof for the plus factor is
	identical.
\end{proof}

% ============================================================
\section{The Zernike representation}
\label{sec:appendix-zernike}
% ============================================================

\subsection{Taylor series and the complexified Zernike algebra}

We shall also use the complexification
$\mathcal C_\varrho^{\mathbb C}$ of the Zernike algebra. It is
convenient to write its elements in complex angular form,
\[
u(r,\vartheta)
=
\sum_{m\in\mathbb Z}
\sum_{l\ge0}
c_{m,l}
R_{|m|+2l}^{|m|}(r)e^{im\vartheta},
\]
with norm
\begin{equation}\label{eq:complex-zernike-norm}
	\|u\|_{\mathcal C_\varrho^{\mathbb C}}
	:=
	\sum_{m\in\mathbb Z}
	\sum_{l\ge0}
	|c_{m,l}|\,
	\varrho^{|m|+2l}.
\end{equation}
With this norm, $\mathcal C_\varrho^{\mathbb C}$ is a complex Banach
algebra and complex conjugation is an isometry:
\begin{equation}\label{eq:zernike-conjugation}
	\|\overline u\|_{\mathcal C_\varrho^{\mathbb C}}
	=
	\|u\|_{\mathcal C_\varrho^{\mathbb C}}.
\end{equation}

For a Taylor series
\[
f(z)=\sum_{n\ge0}c_nz^n,
\qquad
\|f\|_\varrho^{\rm T}
:=
\sum_{n\ge0}|c_n|\varrho^n,
\]
define
\[
(\mathcal T f)(r,\vartheta)
:=
f(re^{i\vartheta}).
\]
\begin{lemma}\label{lem:zernike-real-complex}
	Let $\varrho>1$. Consider the real Zernike expansion
	\[
	u(r,\vartheta)
	=
	\sum_{m,l\ge 0}
	R_{m+2l}^{m}(r)
	\left[
	a_{m,l}\cos(m\vartheta)
	+
	b_{m,l}\sin(m\vartheta)
	\right],
	\]
	with $b_{0,l}=0$, and the corresponding complex angular expansion
	\[
	u(r,\vartheta)
	=
	\sum_{m\in\mathbb Z}
	\sum_{l\ge 0}
	c_{m,l}
	R_{|m|+2l}^{|m|}(r)e^{im\vartheta}.
	\]
	Then, for $m\ge1$,
	\[
	c_{m,l}=\frac{a_{m,l}-ib_{m,l}}{2},
	\qquad
	c_{-m,l}=\frac{a_{m,l}+ib_{m,l}}{2},
	\]
	while
	\[
	c_{0,l}=a_{0,l}.
	\]
	In particular, the complex coefficients satisfy
	\[
	c_{-m,l}=\overline{c_{m,l}},
	\]
	and
	\[
	\|u\|_{\mathcal C_\varrho^{\mathbb C}}
	\le
	\|u\|_{\mathcal C_\varrho}
	\le
	\sqrt2\,
	\|u\|_{\mathcal C_\varrho^{\mathbb C}}.
	\]
	Thus the real Zernike algebra $\mathcal C_\varrho$ is continuously
	identified with the real-symmetric subspace
	\[
	\mathcal C_{\varrho,\mathrm{sym}}^{\mathbb C}
	:=
	\left\{
	u\in\mathcal C_\varrho^{\mathbb C}:
	c_{-m,l}=\overline{c_{m,l}}
	\right\}.
	\]
\end{lemma}

\begin{proof}
	For $m\ge1$,
	\[
	\cos(m\vartheta)
	=
	\frac{e^{im\vartheta}+e^{-im\vartheta}}{2},
	\qquad
	\sin(m\vartheta)
	=
	\frac{e^{im\vartheta}-e^{-im\vartheta}}{2i}.
	\]
	Hence
	\[
	a_{m,l}\cos(m\vartheta)
	+
	b_{m,l}\sin(m\vartheta)
	=
	\frac{a_{m,l}-ib_{m,l}}{2}e^{im\vartheta}
	+
	\frac{a_{m,l}+ib_{m,l}}{2}e^{-im\vartheta},
	\]
	which gives the stated relation between the real and complex
	coefficients.
	
	Moreover,
	\[
	|c_{m,l}|+|c_{-m,l}|
	=
	\sqrt{a_{m,l}^2+b_{m,l}^2}.
	\]
	Therefore
	\[
	\sqrt{a_{m,l}^2+b_{m,l}^2}
	\le
	|a_{m,l}|+|b_{m,l}|
	\le
	\sqrt2\,
	\sqrt{a_{m,l}^2+b_{m,l}^2}.
	\]
	Multiplying by $\varrho^{m+2l}$ and summing over $m\ge1$ and $l\ge0$,
	together with the identical contribution of the $m=0$ modes, yields
	\[
	\|u\|_{\mathcal C_\varrho^{\mathbb C}}
	\le
	\|u\|_{\mathcal C_\varrho}
	\le
	\sqrt2\,
	\|u\|_{\mathcal C_\varrho^{\mathbb C}}.
	\]
\end{proof}
\begin{lemma}\label{lem:taylor-zernike}
	The map
	\[
	\mathcal T:
	\left\{
	f(z)=\sum_{n\ge0}c_nz^n:
	\|f\|_\varrho^{\rm T}<\infty
	\right\}
	\longrightarrow
	\mathcal C_\varrho^{\mathbb C}
	\]
	is an isometric embedding. More precisely,
	\begin{equation}\label{eq:taylor-zernike-isometry}
		\|\mathcal T f\|_{\mathcal C_\varrho^{\mathbb C}}
		=
		\|f\|_\varrho^{\rm T}.
	\end{equation}
\end{lemma}

\begin{proof}
	Since
	\[
	R_n^n(r)=r^n,
	\]
	we have
	\[
	z^n
	=
	r^ne^{in\vartheta}
	=
	R_n^n(r)e^{in\vartheta}.
	\]
	Therefore
	\[
	(\mathcal T f)(r,\vartheta)
	=
	\sum_{n\ge0}
	c_nR_n^n(r)e^{in\vartheta}.
	\]
	In the representation
	\eqref{eq:complex-zernike-norm}, the only nonzero coefficients are
	\[
	c_{n,0}=c_n,
	\qquad n\ge0.
	\]
	Hence
	\[
	\|\mathcal T f\|_{\mathcal C_\varrho^{\mathbb C}}
	=
	\sum_{n\ge0}|c_n|\varrho^n
	=
	\|f\|_\varrho^{\rm T}.
	\]
\end{proof}

As a consequence, if
\[
f'(z)=\sum_{n\ge0}d_nz^n
\]
has finite Taylor norm at weight $\varrho$, then, identifying $f'$
with its image under $\mathcal T$,
\begin{equation}\label{eq:fprime-taylor-zernike}
	\|f'\|_{\mathcal C_\varrho^{\mathbb C}}
	=
	\sum_{n\ge0}|d_n|\varrho^n.
\end{equation}
Moreover, by the Banach algebra property and
\eqref{eq:zernike-conjugation},
\begin{equation}\label{eq:q-zernike-bound}
	\||f'|^2\|_{\mathcal C_\varrho^{\mathbb C}}
	=
	\|f'\overline{f'}\|_{\mathcal C_\varrho^{\mathbb C}}
	\le
	\|f'\|_{\mathcal C_\varrho^{\mathbb C}}^2.
\end{equation}

For the elliptic problem, holomorphic functions arising from the
conformal map must be represented in the Zernike basis. Whenever
complex-valued functions occur, we use the complexification of
$\mathcal C_\varrho$.

\begin{lemma}\label{lem:Taylor2Zerni}
	Let
	\[
	c=c_r+ic_i.
	\]
	For every $n\ge0$,
	\[
	c z^n
	=
	c_r Z^{(c)}_{n,0}
	-
	c_i Z^{(s)}_{n,0}
	+
	i
	\left(
	c_i Z^{(c)}_{n,0}
	+
	c_r Z^{(s)}_{n,0}
	\right),
	\]
	where
	\[
	Z^{(c)}_{n,0}(r,\theta)
	=
	R_n^n(r)\cos(n\theta),
	\]
	\[
	Z^{(s)}_{n,0}(r,\theta)
	=
	R_n^n(r)\sin(n\theta),
	\]
	and
	\[
	R_n^n(r)=r^n.
	\]
\end{lemma}

\begin{proof}
	Since
	\[
	z^n=r^ne^{in\theta},
	\]
	we have
	\[
	z^n
	=
	Z^{(c)}_{n,0}
	+
	iZ^{(s)}_{n,0}.
	\]
	Multiplication by $c$ gives the result.
\end{proof}

\begin{corollary}\label{cor:Taylor2Zerni}
	Let
	\[
	f(z)=\sum_{n\ge0}c_nz^n
	\]
	satisfy
	\[
	\sum_{n\ge0}|c_n|\rho^n<\infty.
	\]
	Then, for every
	\[
	1<\varrho\le\rho,
	\]
	the corresponding function on the unit disk belongs to the
	complexification of $\mathcal C_\varrho$. Moreover, only Zernike
	modes with radial index $l=0$ occur.
\end{corollary}

\begin{proof}
	Apply Lemma~\ref{lem:Taylor2Zerni} term by term. Since
	\[
	R_n^n(r)=r^n
	\]
	and $\varrho\le\rho$, the resulting Zernike coefficients are
	absolutely summable with weight $\varrho^n$.
\end{proof}

\begin{lemma}\label{lem:invlapnorm}
	Let $\rho>1$ and $(-\Delta)^{-1}$ be the inverse Laplacian on the unit disk, with homogeneous boundary conditions. Then
	\[
	\left\|(-\Delta)^{-1}\right\|_{\mathcal C_\rho\to\mathcal C_\rho}
	=
	\frac{1+\rho^2}{8}.
	\]
\end{lemma}

\begin{proof}
	Since the norm of $\mathcal C_\rho$ is a weighted $\ell^1$-norm on
	the Zernike coefficients, the operator norm of $(-\Delta)^{-1}$ is
	the supremum of its weighted column sums.
	
	For $n=m$,
	\[
	(-\Delta)^{-1}R_n^n=\frac{1}{4(n+1)(n+2)}
	\left(R_n^n-R_{n+2}^n\right)
	\]
	(see \cite{ArioliKoch1}), hence the corresponding weighted column sum is
	\[
	\frac{1+\rho^2}{4(n+1)(n+2)}
	\le
	\frac{1+\rho^2}{8},
	\]
	with equality for $n=0$.
	
	For $n>m$,
	\[
	(-\Delta)^{-1}R_n^m=
	-\frac{1}{4(n+1)(n+2)}R_{n+2}^m+
	\frac{1}{2n(n+2)}R_n^m-\frac{1}{4n(n+1)}R_{n-2}^m,
	\]
	so the weighted column sum equals
	\[
	\frac{\rho^2}{4(n+1)(n+2)}
	+
	\frac{1}{2n(n+2)}
	+
	\frac{\rho^{-2}}{4n(n+1)}.
	\]
	Since $n\ge2$, this is bounded by
	\[
	\frac{\rho^2}{48}
	+\frac{1}{16}
	+\frac{\rho^{-2}}{24}
	<
	\frac{1+\rho^2}{8},
	\]
	since
	$$
\frac{1+\rho^2}{8}-\left(\frac{\rho^2}{48}+\frac1{16}+\frac{\rho^{-2}}{24}\right)
	=\frac{3+5\rho^2-2\rho^{-2}}{48}>0
	$$
	Therefore
	\[
	\left\|(-\Delta)^{-1}\right\|
	\le
	\frac{1+\rho^2}{8}.
	\]
	
	Finally, observe that
	\[
	(-\Delta)^{-1}1=\frac{1-r^2}{4}=\frac18\left(R_0^0-R_2^0\right),
	\]
	and hence
	\[
	\left\|(-\Delta)^{-1}1\right\|_\rho=\frac{1+\rho^2}{8},\qquad\|1\|_\rho=1,
	\]
	hence the upper bound is attained.
\end{proof}

% ============================================================
\section{Positivity of the elliptic solutions}
\label{sec:positivity}
% ============================================================

The following results provide a quantitative criterion for proving
positivity of solutions of \eqref{eq:elliptic-disk}. Since all
inequalities involved are strict, the criterion is stable under
sufficiently small perturbations of both the coefficient $q$ and the
solution. This observation is used in
Corollary~\ref{cor:qualitative-persistence}.

Let
\[
0<a<b\le1,
\qquad
A_{a,b}=\{x\in\mathbb R^2:a<|x|<b\}.
\]

\begin{lemma}\label{lem:annulus}
	Let $u\in\mathcal C_\varrho$ solve
	\[
	-\Delta u=qu^3
	\]
	in $A_{a,b}$, with $q(x)>0$. Assume
	\[
	u>0
	\quad\text{on }|x|=a,
	\qquad
	u\ge0
	\quad\text{on }|x|=b,
	\]
	and
	\[
	\lambda_1(A_{a,b})
	>
	\|qu^2\|_{L^\infty(A_{a,b})}.
	\]
	Then
	\[
	u>0
	\qquad
	\text{in }A_{a,b}.
	\]
\end{lemma}

\begin{proof}
	Rewrite the equation as
	\[
	-\Delta u=c(x)u,
	\qquad
	c(x):=q(x)u(x)^2\ge0.
	\]
	Let
	\[
	u^-:=\max\{-u,0\}.
	\]
	Since $u\ge0$ on both boundary components,
	\[
	u^-\in H_0^1(A_{a,b}).
	\]
	Testing the equation with $u^-$ gives
	\[
	\int_{A_{a,b}}|\nabla u^-|^2\,dx
	=
	\int_{A_{a,b}}
	q(x)u(x)^2(u^-)^2\,dx.
	\]
	Hence
	\[
	\int_{A_{a,b}}|\nabla u^-|^2\,dx
	\le
	\|qu^2\|_{L^\infty(A_{a,b})}
	\int_{A_{a,b}}(u^-)^2\,dx.
	\]
	On the other hand,
	\[
	\lambda_1(A_{a,b})
	\int_{A_{a,b}}(u^-)^2\,dx
	\le
	\int_{A_{a,b}}|\nabla u^-|^2\,dx.
	\]
	The strict inequality in the hypothesis therefore implies
	\[
	u^-\equiv0.
	\]
	Thus $u\ge0$ in the annulus. Since $u$ is not identically zero and $u\in\mathcal C_\varrho$, the	strong maximum principle gives $u>0$.
\end{proof}
A similar argument proves the following

\begin{lemma}\label{lem:disk}Let $D_b=\{x\in\mathbb R^2:|x|<b\}$.
	Let $u\in\mathcal C_\varrho$ solve
	\[
	-\Delta u=qu^3
	\]
	in $D_b$, with $q>0$. Assume
	\[
	u>0
	\quad\text{on }|x|=b,
	\]
	and
	\[
	\lambda_1(D_b)>	\|qu^2\|_{L^\infty(D_b)}.
	\]
	Then
	\[
	u>0	\qquad	\text{in }D_b.
	\]
\end{lemma}

\begin{lemma}\label{lem:firstEVb}
	The first eigenvalue of the Dirichlet Laplacian on $A_{a,b}$ satisfies
	\[
	\lambda_1(A_{a,b})
	\ge
	\frac{\pi^2}{(b-a)^2}
	-
	\frac1{4a^2}.
	\]
\end{lemma}

\begin{proof}
The first eigenvalue is simple and the annulus is rotation invariant, hence its positive first eigenfunction is rotation invariant. Hence
	\[
	\lambda_1(A_{a,b})
	=
	\inf_{\substack{u\in H_0^1(a,b)\\u\ne0}}
	\frac{
		\displaystyle
		\int_a^b r|u'(r)|^2\,dr
	}{
		\displaystyle
		\int_a^b r|u(r)|^2\,dr
	}.
	\]
	Set
	\[
	v(r)=\sqrt r\,u(r).
	\]
	Then
	\[
	u'(r)
	=
	\frac{v'(r)}{\sqrt r}
	-
	\frac{v(r)}{2r^{3/2}},
	\]
	and integration by parts gives
	\[
	\int_a^b r|u'(r)|^2\,dr
	=
	\int_a^b
	\left(
	|v'(r)|^2
	-
	\frac{|v(r)|^2}{4r^2}
	\right)dr.
	\]
	Moreover,
	\[
	\int_a^b r|u(r)|^2\,dr
	=
	\int_a^b|v(r)|^2\,dr.
	\]
	Since $r\ge a$,
	\[
	-\frac1{4r^2}
	\ge
	-\frac1{4a^2},
	\]
	while the one-dimensional Dirichlet Poincar\'e inequality gives
	\[
	\int_a^b|v'|^2\,dr
	\ge
	\frac{\pi^2}{(b-a)^2}
	\int_a^b|v|^2\,dr.
	\]
	The result follows.
\end{proof}

\begin{lemma}\label{lem:firstevB}
	For every $r>0$,
	\[
	\lambda_1(\mathbb D_r)
	=
	\frac{j_{0,1}^2}{r^2},
	\]
	where $j_{0,1}$ is the first positive zero of $J_0$.
\end{lemma}

\begin{proposition}\label{prop:positive}
	Let $u\in\mathcal C_\varrho$ solve
	\eqref{eq:elliptic-disk}, and let
	\[
	M\ge
	\|qu^2\|_{L^\infty(\mathbb D)}.
	\]
	Let
	\[
	0=r_0<r_1<\cdots<r_N=1.
	\]
	Assume that
	\[
	u>0
	\qquad
	\text{on }|x|=r_i,
	\quad
	i=1,\ldots,N-1,
	\]
	and that
	\[
	\frac{j_{0,1}^2}{r_1^2}>M,
	\]
	as well as
	\[
	\frac{\pi^2}{(r_i-r_{i-1})^2}
	-
	\frac1{4r_{i-1}^2}
	>M,
	\qquad
	i=2,\ldots,N.
	\]
	Then
	\[
	u>0
	\qquad
	\text{in }\mathbb D.
	\]
\end{proposition}

\begin{proof}
	Lemma~\ref{lem:firstevB} and Lemma~\ref{lem:disk} imply positivity
	in $\mathbb D_{r_1}$. For each	$i=2,\ldots,N$, apply
	Lemmas~\ref{lem:firstEVb} and \ref{lem:annulus} to
$A_{r_{i-1},r_i}$.
	Proceeding inductively gives positivity throughout $\mathbb D$.
\end{proof}

% ============================================================
\section{Computer implementation and rigorous estimates}
\label{sec:cap}
% ============================================================

The analytical results of the preceding sections reduce the proofs of
Theorems~\ref{thm:conformal-stability} and
\ref{thm:elliptic-stability} to the verification of explicit
inequalities at a finite collection of reference configurations.

For the conformal-mapping problem, the rigorous computation produces
bounds for the residual and inverse defect at a reference boundary
function $r_0$, together with the constants entering the stability
estimates of Section~\ref{sec:stability}. These quantities determine
an explicit radius $\delta_r>0$ such that the conclusions of the
Theodorsen analysis hold for every
\[
r\in\mathcal A_\sigma,
\qquad
\|r-r_0\|_\sigma<\delta_r.
\]

For the elliptic problem, no new validation scheme is required.
The estimate
\[
\|q_r-q_{r_0}\|_{\rho_q}
\le
C_q\|r-r_0\|_\sigma,
\]
where $C_q(\delta):=\sqrt{2}C_D C_f\left(2\|f_{r_0}'\|_{\rho_q}+C_D C_f\delta\right)$,
provides a rigorous enclosure of the coefficients corresponding to
the entire boundary neighborhood. The existing Zernike-based
validation is applied directly to this enlarged enclosure.
Consequently, one computer-assisted calculation proves existence of
the elliptic solution simultaneously for all domains in the
validated neighborhood.

The implementation is written in Ada \cite{Ada,Gnat}. The source code and
data files are available at \cite{program}.

% ------------------------------------------------------------
\subsection{Approximation and validation programs}
% ------------------------------------------------------------

The numerical approximations are computed by
\[
{\tt Find\_Theo},
\qquad
{\tt PrepareWH},
\qquad
{\tt Find\_Elliptic}.
\]
The program {\tt Find\_Theo} computes a finite-dimensional
approximation of the Theodorsen solution. The program
{\tt PrepareWH} constructs accurate approximate Wiener--Hopf factors
and their inverses, see Appendix \ref{sec:appendix-WH}. Among other things, it computes the numerical approximation of the eigenvalues of very large matrices. This cannot be done with standard floating point numbers \cite{IEEE}, instead we use a 1024-bit representation provided by the library MPFR \cite{MPFR}. The program {\tt Find\_Elliptic} computes a finite
Zernike approximation of a solution of the transformed elliptic
equation.

These computations are nonrigorous and serve only to produce input
for the validation programs.

The rigorous calculations are carried out by
\[
{\tt Check\_Theo}
\qquad\text{and}\qquad
{\tt Check\_Elliptic}.
\]

For a reference boundary function $r_0$, {\tt Check\_Theo} computes
rigorous enclosures of the constants appearing in Sections~\ref{sec:stability} and
\ref{sec:verification}, checks the invertibility of $L_K$ and $D\mathcal F_{r_0}(\bar u)$ and determines a radius $\delta_r$ for which
the inequalities
\[
Z_0+\Lambda C_\Phi(\delta_r)\delta_r<1
\]
and
\[
Y_0
+
\Lambda C_F(\delta_r)\delta_r
+
\bigl(
Z_0+\Lambda C_\Phi(\delta_r)\delta_r
\bigr)R
\le R
\]
hold.

The program {\tt Check\_Elliptic} reads the validated reference
coefficient $q_{r_0}$ together with the rigorous perturbation bound
\[
\|q_r-q_{r_0}\|_{\rho_q}
\le
C_q\delta_{\mathrm{PDE}}.
\]
It incorporates this ball directly into the coefficient enclosure and
then applies the same Zernike-based a posteriori estimates used for a
fixed coefficient. Since the code implementing the elliptic validation accepts interval coefficients \(q\), the algorithm itself need not change.

% ------------------------------------------------------------
\subsection{Validated scalar and finite-dimensional enclosures}
% ------------------------------------------------------------

The rigorous computations use inclusion-preserving representations.
Let $\mathcal X$ and $\mathcal Y$ be normed spaces and let
\[
f:\mathcal X\longrightarrow\mathcal Y.
\]
An enclosure procedure $F$ associates to every admissible set
$X\subset\mathcal X$ a set
\[
F(X)\subset\mathcal Y
\]
such that
\[
f(X)\subset F(X).
\]

A representable real interval is encoded by
\[
{\tt S=(S.C,S.R)},
\]
where {\tt S.C} is the center and {\tt S.R} is a nonnegative radius.
For real scalars the corresponding type is {\tt Ball}; complex
enclosures are represented by the type {\tt Complex}. Elementary functions on
these types are implemented with rounding control. Vector and matrix
operations are extended componentwise.

These data types are standard components of our computer-assisted
proof library. Below we describe only the representations specific to
the present problem.

% ------------------------------------------------------------
\subsection{Representations of the function spaces}
% ------------------------------------------------------------

Elements of $\mathcal A_\rho$ are represented by the type
{\tt Fourier1}, introduced in \cite{ArioliKochTerracini} together with all necessary algebraic operations, and then expanded to analytic functions in \cite{ArioliKoch1}. Functions in the Zernike algebra $\mathcal C_\rho$ are represented by
{\tt Zernike2}. The related type {\tt Zernike}, also introduced in
\cite{ArioliKoch1}, represents functions with prescribed parity.
The type {\tt Zernike2} permits both parities and is implemented as a
pair of {\tt Zernike} objects.
For the Laurent algebra $\mathcal B_\rho$ we use the type
{\tt Laurent}, described next.

% ------------------------------------------------------------
\subsection{Rigorous representation of Laurent series}
\label{sec:implementation-laurent}
% ------------------------------------------------------------

An object of type {\tt Laurent} represents
\[
f\in\mathcal B_\rho
\]
as
\[f=f_K+f_0+f_P+f_M,\]
where
\[
f_K(z)
=
\sum_{|k|\le K}a_kz^k
\]
has interval coefficients $\{a_k\}$, while
\[
f_P(z)=\sum_{k>K}a_kz^k,
\qquad
f_M(z)=\sum_{k<-K}a_kz^k,
\]
play the role respectively of positive and negative tails, and
$$f_0=\sum_{k\in\mathbb Z}b_kz^k$$
is a generic remainder.
The implementation stores the coefficients $\{a_k\}$ and nonnegative bounds
\[
E_0^f,
\qquad
E_P^f,
\qquad
E_M^f
\]
such that
\[
\|f_0\|_\rho
\le
E_0^f,
\qquad
\|f_P\|_\rho
\le
E_P^f,
\qquad
\|f_M\|_\rho
\le
E_M^f.
\]
Concretely, an object {\tt Laurent} is
\[
{\tt F=(F.C,F.EM,F.EP,F.E0)},
\]
where {\tt F.C} is a set of interval coefficients indexed by
$-K,\ldots,K$ and {\tt F.EM,F.EP,F.E0} are positive representable floating point numbers.

Separating positive and negative tails is useful in the
Wiener--Hopf calculations because products of tails with opposite
signs may contribute to the retained central modes.

\subsection{Rigorous multiplication of Laurent series}
\label{sec:multiplication-laurent}

Let $f$ and $g$ have finite coefficients $a_k$ and $b_k$. For
$|m|\le K$, the retained finite part of
\[
h=fg
\]
is computed by interval convolution,
\[
c_m
=
\sum_{\substack{k+\ell=m\\-K\le k,\ell\le K}}
a_kb_\ell.
\]
Define
\[
A_+^f
=
\sum_{k=0}^K|a_k|\rho^k,
\qquad
A_-^f
=
\sum_{k=-K}^{0}|a_k|\rho^{|k|},
\]
and similarly for $g$.

The finite--finite product outside the retained range is bounded by
\[
E_P^{KK}
=
\sum_{m=K+1}^{2K}
\left|
\sum_{\substack{k+\ell=m\\-K\le k,\ell\le K}}
a_kb_\ell
\right|
\rho^m,
\]
and
\[
E_M^{KK}
=
\sum_{m=-2K}^{-K-1}
\left|
\sum_{\substack{k+\ell=m\\-K\le k,\ell\le K}}
a_kb_\ell
\right|
\rho^{|m|}.
\]
The positive and negative tail bounds may therefore be chosen as
\[
E_P^h
=
E_P^{KK}
+
E_P^gA_+^f
+
E_P^fA_+^g
+
E_P^fE_P^g,
\]
and
\[
E_M^h
=
E_M^{KK}
+
E_M^gA_-^f
+
E_M^fA_-^g
+
E_M^fE_M^g.
\]

The remaining contributions are absorbed into
\[
	E_0^h
	\le
	E_0^f\|g\|_\rho+E_0^g\|f\|_\rho+E_P^gA_-^f+E_M^gA_+^f+E_P^fA_-^g+E_M^fA_+^g
	+E_P^fE_M^g+E_M^fE_P^g.
\]
All convolutions and all bounds are evaluated using interval
arithmetic.

% ------------------------------------------------------------
\subsection{Numerical construction of Wiener--Hopf factors}
\label{sec:implementation-WH}
% ------------------------------------------------------------

The initial Wiener--Hopf factors are obtained from the roots of a
Laurent-polynomial approximation. This stage is used only to produce
accurate numerical approximations; the factors and their inverses are
subsequently certified by
Proposition~\ref{prop:WH-validation}.

For a monic polynomial
\[
p(z)
=
z^N+a_{N-1}z^{N-1}+\cdots+a_1z+a_0,
\]
initial approximations to its zeros are computed as the eigenvalues
of the companion matrix
\[
C=
\begin{pmatrix}
	0 & 0 & \cdots & 0 & -a_0\\
	1 & 0 & \cdots & 0 & -a_1\\
	0 & 1 & \cdots & 0 & -a_2\\
	\vdots & \vdots & \ddots & \vdots & \vdots\\
	0 & 0 & \cdots & 1 & -a_{N-1}
\end{pmatrix}.
\]
The calculation is performed in multiple precision.

The approximations are refined using the Aberth iteration. Given
\[
z_1^{(m)},\ldots,z_N^{(m)},
\]
define
\[
N_i^{(m)}
=
\frac{p(z_i^{(m)})}{p'(z_i^{(m)})},
\qquad
S_i^{(m)}
=
\sum_{\substack{j=1\\j\ne i}}^N
\frac1{z_i^{(m)}-z_j^{(m)}}.
\]
The update is
\[
z_i^{(m+1)}
=
z_i^{(m)}
-
\Delta_i^{(m)},
\]
where
\[
\Delta_i^{(m)}
=
\frac{N_i^{(m)}}
{1-N_i^{(m)}S_i^{(m)}}.
\]

The stopping criterion is purely numerical and is not used as a
rigorous root certification. The resulting roots are used only to
construct approximate Wiener--Hopf factors.

% ------------------------------------------------------------
\subsection{Generic Newton and contraction machinery}
% ------------------------------------------------------------

The packages {\tt Linear} and {\tt Linear.Contr} implement the
Newton-like and contraction arguments used in the
computer-assisted proofs. They are generic with respect to the
underlying function representation.

In the present work they are used to construct quasi-Newton maps,
compute rigorous residual and derivative bounds, and verify the
a posteriori inequalities at the reference configurations. The same
infrastructure has been used in previous computer-assisted proofs; see
\cite{ArioliKoch1,ArioliKoch3}.

\end{document}